\documentclass{amsart}
\pdfoutput=1
\usepackage[utf8]{inputenc} 
\usepackage{tikz-cd}

\usepackage[T1]{fontenc}    % use new font-encoding (to have bold
\usepackage{ae,aecompl}     % changes to Type 1 fonts
\usepackage{url}

\usepackage{pdfsync}

\usepackage{tensor}

\input xypic

\xyoption{all}

\xyoption{arc}

\xyoption{2cell}

\usepackage{amsmath}
\usepackage{amsfonts}
\usepackage{amssymb}
\usepackage{amsthm} % 
\usepackage{mathrsfs} % 
\usepackage{enumerate} 
\usepackage{verbatim}	%To comment parts
\usepackage{dynkin-diagrams}

\usepackage{tikz}
\usepackage{tikz-cd}
\usepackage{pgfplots}
\usetikzlibrary{calc}
\usetikzlibrary{shapes}

\input xypic
\xyoption{all}
\xyoption{arc}
\xyoption{2cell}

\usetikzlibrary{intersections}
\usetikzlibrary{patterns}
\usepackage{circuitikz}

\usepackage{caption}
\usepackage[hyperfootnotes]{hyperref}

\theoremstyle{plain} % style plain
\newtheorem{thm}{Theorem}[section]
\newtheorem{cor}[thm]{Corollary}
\newtheorem{prop}[thm]{Proposition}

\theoremstyle{definition}

\newtheorem{remark}[thm]{Remark}

\newtheorem{ex}[thm]{Example}

\usepackage{color}
\definecolor{Ccolor}{rgb}{0,0.5,0}
\definecolor{Mcolor}{rgb}{1,0,0}
\definecolor{lightgray}{rgb}{0.6,0.6,0.6}

\newcommand{\h}{\widehat}
\newcommand{\sm}{\setminus}
\newcommand{\cone}{\operatorname{cone}}
\newcommand{\Span}{\operatorname{span}}

\newcommand{\supp}{\operatorname{supp}}

\newcommand{\wro}{\leq_{R}}
\newcommand{\seq}{\subseteq}
\def\grp#1{ \langle {#1} \rangle_{group}}
\def\min#1{ \textnormal{min} \{ {#1} \}  }
\def\objects#1{\mathcal{O}({#1})}

\newcommand\sbullet[1][.5]{\mathbin{\vcenter{\hbox{\scalebox{#1}{$\bullet$}}}}}

\def\Cubes#1#2{\textnormal{Cubes}_{{#1}}({#2}) } 
\def\Prod{\textnormal{Prod}}
\def\Edge#1{\textnormal{Edge}_{{#1}}}

\author[H.~Gimenez]{Harrison Gimenez}
\address[H.~Gimenez]{Department of Mathematics \\ B26 Hayes-Healy Building \\ University of Notre Dame, Indiana 46556, U.S.A.}
\email{hgimenez@nd.edu}

\subjclass[2010]{Primary 20F55}

\title{Coxeter systems, left inversion sets, and higher dimensional cubes}

\subjclass[2020]{20F55; 17B22, 20L05}

\begin{document}

% command for left quotation marks: \textquotedblleft

\begin{abstract}
    Let $ (W,S)$ be a Coxeter system. We investigate the equation $ w(\Phi_{x}) = \Phi_{y}$ where $ w,x,y\in W$ and $ \Phi_{x}$, $\Phi_{y}$ denote the left inversion sets of $ x$ and $ y$. We then define a commutative square diagram called a Coxeter square which describes the relationship between 4 non-identity elements of the Coxeter group $ W$ and the equation $ w(\Phi_{x}) = \Phi_{y}$. Coxeter squares were first introduced by Dyer, Wang in \cite{dyer2011groupoids2} and \cite{dyer2019characterization}. Coxeter squares can be \textquotedblleft glued" together by compatible edges to form commutative diagrams in the shape of  higher dimensional cubes called Coxeter $n$-cubes, which were first defined by Dyer in Example 12.5 of \cite{dyer2011groupoids2}. When $ |W| < \infty$ and $ |S| = n$, we show that Coxeter $n$-cubes must exist within $ (W,S)$. We then prove results about Coxeter $n$-cubes in the $A_{n}$ Coxeter system.  We establish an explicit bijection between Coxeter $n$-cubes (modulo orientation) in $ A_{n}$ and binary trees with $n+1$ leaves. We also show that an element $x$ of $ A_{n}$ appears as the edge of some Coxeter $n$-cube if and only if $ x$ is a bigrassmannian permutation.

 \end{abstract}

%%%%%%%%%%%%%%%%%%%%%%%%%%%%%%%%%

 \maketitle

%%%%%%%%%%%%%%%%%%%%%%%%%%%%%%%
%\section{Introduction}
%%%%%%%%%%%%%%%%%%%%%%%%%%%%%%%

\section{Introduction}
Coxeter systems are interesting mathematical objects of study due to their relationship with group theory, representation theory, and combinatorics. The canonical root system $\Phi$ associated to a Coxeter system $ (W,S)$ as defined in \cite{humphreys1992reflection} encodes geometric and combinatorial information about the Coxeter group $ W$. Given an element $ x\in W$, we let $ \Phi_{x}$ denote the set of elements of $ \Phi^{+}$ that are sent into $ \Phi^{-}$ by the element $ x^{-1}$. The set $ \Phi_{x}$ is called the left inversion set of $ x$. This paper studies the equation $ w(\Phi_{x}) = \Phi_{y}$ where $ w,x,y\in W$. In other words, we are interested in determining when a left inversion set is transferred to a left inversion by an element of $ W$. As motivation for studying left inversion sets, we note that left inversion sets relate to reduced expressions of elements, the weak right order, and the root poset of a Coxeter system. In the case when $ x = s$, $ y =s'$ where $ s,s' \in S$, the equation $ w(\Phi_{s}) = \Phi_{s'}$ becomes $ w(\alpha_{s}) = \alpha_{s'}$ where $ \alpha_{s}$, $ \alpha_{s'}$ denote the simple roots corresponding to $ s$, $ s'$ respectively. Thus, the specific case of $w(\Phi_{s}) = \Phi_{s'} $ asks when $ w$ maps a simple root to a simple root. The transfer of simple roots is important in the fundamental \textquotedblleft exchange condition" of Coxeter systems and for studying morphisms of Brink-Howlett groupoids as defined by Brink and Howlett in \cite{brink1999normalizers}. Morphisms of a Brink-Howlett groupoid are elements of $ W$ that map subsets of simple roots to subsets of simple roots. Studying the more general equation $ w(\Phi_{x}) = \Phi_{y}$ may provide insights to when $ w$ maps a collection of left inversion sets to a collection of left inversion sets. Transferring collections of left inversion sets to collections of left inversion sets allows for a generalization of the Brink-Howlett groupoid to a wider class of groupoids (see Example 4 of \cite{dyer2019characterization}).

In Section \ref{backgroundmaterial}, we review some basic definitions and constructions related to the Coxeter system $ (W,S)$. We define the canonical root system associated to $ (W,S)$ as in Humphreys \cite{humphreys1992reflection}. We also state some propositions that connect left inversion sets to the reflection cocycle and weak right order of $ (W,S)$.

In Section \ref{backgroundgroupoids}, we define Brink-Howlett groupoids and establish their basic properties. We describe standard generators for Brink-Howlett groupoids, defined originally by Howlett and Deodhar. % need to add citations for Howlett and Deodhar here

In Section \ref{theequation}, we prove some propositions regarding the equation $ w(\Phi_{x}) = \Phi_{y}$. We relate that equation to the equation $ wN(x)w^{-1}  =N(y)$ where $ N(x)$ denotes the reflection cocycle of $x$. In the specific case when $ |W|< \infty$, the equation $ w(\Phi_{x}) = \Phi_{y}$ implies that $ w$ is a right divisor of $ w_{0}x^{-1}$ where $ w_{0}$ denotes the longest element of $ W$.

In Section \ref{squaresandcubes}, we define Coxeter squares and Coxeter $n$-cubes. When $ |W| < \infty$ and $ |S| = n$, we show that Coxeter $n$-cubes must exist within $ (W,S)$. We then prove results about Coxeter $n$-cubes in the $A_{n}$ Coxeter system.  We establish an explicit bijection between Coxeter $n$-cubes (modulo orientation) in $ A_{n}$ and binary trees with $n+1$ leaves. We also show that an element $x$ of $ A_{n}$ appears as the edge of some Coxeter $n$-cube if and only if $ x$ is a bigrassmannian permutation.

\section{Background Material on Coxeter Systems} \label{backgroundmaterial}

We let $ \mathbb{N}$ denote the set of non-negative integers (so that $0 \in \mathbb{N}$). We define $ \mathbb{N}_{\geq 2} : = \{ n\in \mathbb{N} \mid n \geq 2 \} $. For a comprehensive background on Coxeter systems, we recommend both \cite{bjorner2005combinatorics} and \cite{humphreys1992reflection}.

\subsection{Coxeter systems}
% Define what a Coxeter system is
Let $  S$ be a set. We define a function $ m : S\times S \to \mathbb{N} \cup \{ \infty \}$ such that the following hold:
\begin{enumerate}
    \item Let $ s , s' \in S$. Then $m(s,s') = 1$ if and only if $ s =s'$. If $ s \neq s'$, then $ m(s,s')\in \mathbb{N}_{\geq 2} \cup \{ \infty \}$.

    \item Let $ s , s' \in S$. Then $ m(s,s') = m(s' ,s)$.
\end{enumerate}
Such a function $ m: S \times S \to \mathbb{N} \cup \{ \infty \}$ satisfying the above properties is called a \emph{Coxeter matrix}. A \emph{Coxeter System} is an ordered pair $ (W,S)$ where 
\begin{enumerate}
    \item $ S$ is a set.

    \item $ m : S \times S \to \mathbb{N} \cup \{ \infty \}$ is a Coxeter matrix.

    \item $ W$ is a group such that $ W = \grp{s \in S \mid \forall s, s' \in S , \ (ss')^{m(s,s')} = 1}$.
\end{enumerate}
If $ m(s,s') = \infty$, then we omit the relation $ (ss')^{\infty} = 1$ from the above presentation of the group $W$. $W$ and $ S$ are referred to as the \emph{Coxeter group} and the \emph{Coxeter generators} respectively. We also define $T:= \{ wsw^{-1} \mid s\in S , \ w\in W \}$. $T$ is called the \emph{set of reflections of $ (W,S)$}.

\subsection{Standard parabolic subgroups}
% Define what a standard parabolic subgroup is
Let $ (W,S)$ be a Coxeter system. Let $ J \seq S$. Define $ W_{J}$ to be the subgroup of $W$ generated by $J$. The subgroup $ W_{J}$ is called a \emph{standard parabolic subgroup of $W$}. $ (W_{J} , J)$ is a Coxeter system with the following presentation

$$ W_{J} : = \grp{s\in J \mid \forall s, s' \in J , \ (ss')^{m(s,s')} =1}$$
and $ (W_{J} ,J)$ is called a \emph{standard parabolic subsystem of $(W,S)$}.

\subsection{The length function}
% Define the length \ell(w)
Let $ (W,S)$ be a Coxeter system. Let $ w \in W$. We define the \emph{length function} $ \ell : W \to \mathbb{N}$ so that:

$$ \ell(w) : = \min{ n \mid w = s_{1}s_{2} \dots s_{n} , \   s_{i} \in S \textrm{ for all } i }$$
Furthermore, $ \ell(w) = 0$ if and only if $ w = 1$, where $ 1$ denotes the identity of the group $ W$. If $ J \seq S$, we let $ \ell_{J} : W_{J} \to \mathbb{N}$ denote the length function associated to the standard parabolic subsystem $ (W_{J} , J)$ of $(W,S)$. If $ w \in W_{J}$, then $ \ell_{J}(w) = \ell(w)$ by Theorem 5.5 part (b) of \cite{humphreys1992reflection}.
\subsection{The canonical root system}
% Define the canonical root system associated to a Coxeter system as defined in Humphreys
Let $ (W,S)$ be a Coxeter system. Define the set of \emph{simple roots} $ \Pi : = \{ \alpha_{s} \mid s \in S  \}$ where $ \alpha_{s}$ is treated as a formal symbol. Let $ V$ denote the real vector space with basis $ \Pi$. Define a function $ B : \Pi \times \Pi \to \mathbb{R}$ in the following way:

$$ B(\alpha_{s} , \alpha_{s'}) = -\cos(\frac{\pi}{m(s, s')}) \textrm{ for all } s,s' \in S$$
where if $m(s,s') = \infty $, then $ - \cos(\frac{\pi}{ \infty}) := -1$. Note that $ B(\alpha_{s}, \alpha_{s'}) =B(\alpha_{s'}, \alpha_{s})$ since $ m(s,s') = m(s',s)$. Since $ \Pi$ is a basis of $V$, we can extend $B(-, - )$ via bilinearlity to become a symmetric bilinear form $ B: V \times V \to \mathbb{R}$. We will call $ B: V \times V \to \mathbb{R}$ \emph{the Coxeter form of $(W,S)$}.

Given $ s\in S $ and $ v\in V$, we define a map $ s: V \to V$ such that

$$ s(v) = v -2B(v, \alpha_{s}) \alpha_{s}$$
One can verify that the maps $ \{ s :V \to V \mid s\in S \}$ satisfy the relations of the Coxeter system $ (W,S)$ (see Proposition 5.3 of \cite{humphreys1992reflection}). There is a unique representation of $W$ on the vector space $ V$ such that $s\in S \seq W$ acts as the linear map defined above. This representation $ W \to \textrm{End}(V)$ is called \emph{the canonical representation of $(W,S)$}. The canonical representation is faithful (Corollary 5.4 of \cite{humphreys1992reflection}).

Using the canonical representation $ W \to \textrm{End}(V)$, we define \emph{the root system} $ \Phi \seq V$ such that

$$ \Phi : = \{ w(\alpha_{s}) \mid s\in S , \ w\in W  \}$$
Given some subset $H$ of the vector space $V$, we define 

$$ \cone(H) : = \{ c_{1}h_{1}+c_{2}h_{2}+ \dots + c_{k}h_{k} \mid k \in \mathbb{N}, \  c_{i} \geq 0, \  h_{i} \in H \ \forall i=1,2,\dots , k \}$$ 
where $ \cone(\emptyset) = \{ 0 \}$. The positive roots $ \Phi^{+} \seq \Phi$ and the negative roots $ \Phi^{-} \seq \Phi $ are defined as follows:

$$ \Phi^{+} = \Phi \cap \cone(\Pi)$$

$$ \Phi^{-} : = -(\Phi^{+}) = \Phi \cap \big( -\cone(\Pi) \big) $$
Note that $ \Pi \seq \Phi^{+}$. A general fact about $\Phi$ is that $ \Phi = \Phi^{+} \coprod \Phi^{-}$, where $ \coprod$ denotes disjoint union. If $ J \seq S$, then we define $\Pi_{J} : = \{ \alpha_{s} \mid s\in J  \}$, $ \Phi_{J} : = \{ w(\alpha_{s}) \mid s \in J , \ w\in W_{J} \}$, $ \Phi_{J}^{+} : = \Phi_{J} \cap \cone(\Pi_{J})$, and $\Phi_{J}^{-} := -(\Phi_{J}^{+}) = \Phi_{J} \cap (-\cone(\Pi_{J})) $. We have $ \Phi_{J} = \Phi_{J}^{+} \coprod \Phi_{J}^{-}$. Furthermore, we also have $ \Phi_{J} = \Phi \cap \Span(\Pi_{J})$.

% state generalized exchange condition
If $ t\in T $ is a reflection, then we let $ \alpha_{t} \in \Phi^{+}$ denote the unique positive root associated to the reflection $ t$ (section 5.7 of \cite{humphreys1992reflection}). The generalized exchange condition is one of the most important properties of the canonical representation:

\begin{prop}[Generalized Exchange Condition] \label{generalizedexchange} 
    Let $ w = s_{1}s_{2} \dots s_{r}$. Let $t\in T $ such that $\ell(wt) < \ell(w)$. Then the following are true: 

    \begin{itemize}
        \item $w(\alpha_{t}) \in \Phi^{-}$

        \item There is an index $ i$ such that $ wt = s_{1} \dots \h{s_{i}} \dots s_{r}$. If $ \ell(w) = r$, then the index $i$ is unique.
    \end{itemize}
    
\end{prop}

\begin{proof}
    Proposition 5.7 and Theorem 5.8 of \cite{humphreys1992reflection}.
\end{proof}

\subsection{Left Inversion Sets}
% Define the left inversion set of an element w\in W
Given a Coxeter system $ (W,S)$ and an element $w\in W$, we define \emph{the left inversion set of $w$}:

$$ \Phi_{w} : = \{ \beta \in \Phi^{+} \mid w^{-1}(\beta) \in \Phi^{-}  \} $$
% Describe how the left inversion set relates to the length
\begin{prop} \label{lengthandlset}
    Let $ (W,S)$ be a Coxeter system. Let $ w\in W$. Then $\ell(w) = |\Phi_{w}|$
\end{prop}

\begin{proof}
    $\ell(w) = \ell(w^{-1})$, and $ \ell(w^{-1})$ is equal to the number of positive roots sent into $ \Phi^{-}$ by the element $ w^{-1}$ (Proposition 5.6 part (b) of \cite{humphreys1992reflection}). Hence,

    $$ \ell(w) = \ell(w^{-1}) = |\Phi_{w}|$$
\end{proof}

\begin{prop} \label{leftinversionset}
    Let $ (W,S)$ be a Coxeter system. Let $ w\in W$. Let $ w = s_{1}s_{2} \dots s_{n}$ with $ \ell(w) = n$. Also let $ \alpha_{i} \in \Pi$ denote the simple root corresponding to $ s_{i} \in S$. Then $  \Phi_{w} = \{ \beta_{i} \mid 1 \leq i \leq n  \} $ where

    $$ \beta_{i} : = s_{1}s_{2} \dots s_{i-1}(\alpha_{i})$$
    and $ \beta_{1} : = \alpha_{1}$.
\end{prop}

\begin{proof}
    Since $ w = s_{1}s_{2} \dots s_{n}$ is a reduced expression, it follows that $ w^{-1} = s_{n} \dots s_{2}s_{1}$ is a reduced expression. The result now follows from Exercise 1 of Section 5.6 of \cite{humphreys1992reflection}.
\end{proof}

\subsection{The weak right order}
% Define the weak right order
Given a Coxeter system $ (W,S)$ and elements $ x,y \in W$, we define a relation $ x \wro y$ if and only if

$$ y = x s_{1}s_{2} \dots s_{n} \textrm{ and } \ell(xs_{1}s_{2} \dots s_{i}) = \ell(x) + i \textrm{ for all } 0 \leq i \leq n$$
where $ s_{i} \in S$ for all $i$. The relation $ \wro $ is called \emph{the weak right order of $(W,S)$}. The weak right order is a partial order.

% describe how the weak right order relates to left inversion sets
\begin{prop} \label{wroandlset}
    Let $ (W,S)$ be a Coxeter system and let $ x,y \in W$. Then $ x \wro y$ if and only if $ \Phi_{x} \seq \Phi_{y}$.
\end{prop}

\begin{proof}
    The proposition is a well-known fact. See section 1.3 of \cite{dyer2019weak}.
\end{proof}

\subsection{The reflection cocycle}
% Define the reflection cocycle
Let $ (W,S)$ be a Coxeter system with set of reflections $ T = \{ wsw^{-1} \mid s\in S , \ w\in W \}$. Let $ N : W \to \mathcal{P}(T)$ be defined by

$$ w \longmapsto N(w) : = \{  t\in T \mid \ell(tw) < \ell(w)  \}$$
We call $ N :W \to \mathcal{P}(T)$ the reflection cocycle of $ (W,S)$.

% describe properties of the reflection cocycle
\begin{prop} \label{reflectioncocycle}
    Let $ (W,S)$ be a Coxeter system. Let $ x,y \in W$. Then:

    $$ N(xy) = N(x) \Delta xN(y)x^{-1}$$
    where $ \Delta$ denotes the symmetric differences of sets. Also, 
    
    $$ \ell(xy) = \ell(x) +\ell(y) \iff  N(xy) = N(x) \coprod xN(y)x^{-1}$$
    where $ \coprod $ denotes the disjoint union.
\end{prop}

\begin{proof}
    See the beginning of the proof of Lemma 4.1 in \cite{dyer2019weak}.
\end{proof}

\begin{prop} \label{equivalenceofwro}
    Let $ (W,S)$ be a Coxeter system. Let $ x,y \in W $. Then the following are true:

    \begin{enumerate}
        \item There is a bijective correspondence between the sets $ N(x)$ and $ \Phi_{x}$ given by $ t \longmapsto \alpha_{t}$ where $ t \in N(x)$.

        \item $ N(x) \seq N(y) \iff \Phi_{x} \seq \Phi_{y} \iff x \wro y $ 

        \item $ |N(x)| = |\Phi_{x}| = \ell(x)$
        
    \end{enumerate}
\end{prop}

\begin{proof}
    Part (1) follows immediately from \ref{generalizedexchange}. Statements (2) and (3) are consequences of part (1) and the previous propositions.
\end{proof}

\section{Background Material on Brink-Howlett groupoids} \label{backgroundgroupoids}

\subsection{Definition of Brink-Howlett groupoids}
% Define the Brink-Howlett groupoid
 Recall that a groupoid is a small category in which every morphism is invertible. Let $ (W,S)$ be a Coxeter system. Let $ J \seq S$. Define the following:

$$  \objects{\Pi_{J}} : = \{  \Pi_{K} \seq \Pi \mid K \seq S , \ w\in W , \ w(\Pi_{J}) = \Pi_{K} \}$$
We let $ G(\Pi_{J})$ denote \emph{the Brink-Howlett groupoid generated at the object $ \objects{\Pi_{J}}$}. The objects of $ G(\Pi_{J})$ are the elements of the set $ \objects{\Pi_{J}}$. A morphism of $ G(\Pi_{J})$ is defined to be an ordered triple $ ( \Pi_{L} , w , \Pi_{K})$ where $ \Pi_{K}, \Pi_{L} \in \objects{\Pi_{J}}$ and $ w\in W$ such that $ w(\Pi_{K}) = \Pi_{L}$. If $ (\Pi_{L'}, w ' , \Pi_{K'})$ and $ (\Pi_{L}  ,w  ,\Pi_{K})$ are morphisms of $ G(\Pi_{J})$ such that $ K' = L$, then we define the composition of morphisms as:

$$ (\Pi_{L'}, w ' , \Pi_{K'})(\Pi_{L}  ,w  ,\Pi_{K}) = (\Pi_{L'} , w'w , \Pi_{K})$$
It is trivial to check that $w'w(\Pi_{K}) = \Pi_{L'} $. If $ (\Pi_{L}  ,w  ,\Pi_{K})$ is a morphism, then $ (\Pi_{K}  ,w^{-1}  ,\Pi_{L})$ is the inverse morphism.

\subsection{Generators of the groupoid}
% Define a BH generator/state what a BH generator looks like
The following propostion was proven by Brink and Howlett in \cite{brink1999normalizers} by the use of a proposition by Deodhar. The proposition by Deodhar is stated as Proposition 2.2 of \cite{brink1999normalizers} and as Proposition 4.2 of \cite{deodhar1982root}.

\begin{prop} \label{generatorexistence}
Let $ I_{1} , I_{2} \seq S$ such that $ I_{1} \cap I_{2} = \emptyset$. If $ \Phi_{I_{1} \cup I_{2}} \sm \Phi_{I_{1}}$ is a finite set, then there exists a unique element $ \nu( \Pi_{I_{2}} , \Pi_{I_{1}} ) \in W_{I_{1} \cup I_{2}}$ such that \newline $ \nu( \Pi_{I_{2}}, \Pi_{I_{1}}) ( \Phi^{+}_{I_{1} \cup I_{2}} \sm \Phi^{+}_{I_{1}}) \seq \Phi^{-}_{I_{1} \cup I_{2}}$ and $ \nu( \Pi_{I_{2}} , \Pi_{I_{1}})(\Phi_{I_{1}}^{+}) \seq \Phi^{+}_{I_{1} \cup I_{2}}$. Furthermore, $ \nu(\Pi_{I_{2}} , \Pi_{I_{1}}) (\Pi_{I_{1}}) \seq \Pi_{I_{1} \cup I_{2}}$, so $ \nu(\Pi_{I_{2}}, \Pi_{I_{1}})$ is a morphism of the Brink-Howlett groupoid $ G(\Pi_{I_{1}})$.

\end{prop}

\begin{proof}
    See the paragraph after Proposition 2.1 in \cite{brink1999normalizers}.
\end{proof}
Let $ s \in S$. Let $ \alpha \in \Pi$ be the simple root associated to $s$. If $ I_{2} = \{ s \} $, then $ \Pi_{I_{2}} = \{ \alpha \} $ and we write $ \nu(\alpha , \Pi_{I_{1}}) := \nu(\{ \alpha \} , \Pi_{I_{1}})$. We call elements of the form $ \nu( \alpha, \Pi_{I_{1}})$ \emph{Brink-Howlett generators}. Theorem \ref{decomposingmorphisms} justifies calling these elements generators of the Brink-Howlett groupoid.

\subsection{Decomposing morphisms}
% Describe how any morphism of the groupoid can be decomposed into a product of BH generators
\begin{thm} \label{decomposingmorphisms}
    Let $( \Pi_{L} , w ,\Pi_{K})$ be a morphism of the groupoid $G(\Pi_{J}) $. Then:

    $$ ( \Pi_{L} , w ,\Pi_{K}) = g_{n}g_{n-1} \dots g_{2}g_{1}$$
    where $ g_{i} := (\Pi_{J_{i}} , \nu(\alpha_{i} , \Pi_{J_{i-1}}) , \Pi_{J_{i-1}})$, $ \Pi_{J_{0}} = \Pi_{K}$, $\Pi_{J_{n}} = \Pi_{L} $, and \newline $ \ell(w) = \sum_{i=1}^{n} \ell(\nu(\alpha_{i} , \Pi_{J_{i-1}}))$. Furthermore, for any $ \alpha \in \Pi$ such that $ w(\alpha) \in \Phi^{-}$, we can always choose $ \alpha_{1} = \alpha$.
\end{thm}

\begin{proof}
    See Proposition 2.3 of \cite{brink1999normalizers}.
\end{proof}

\section{The equation $ w(\Phi_{x}) = \Phi_{y}$} \label{theequation}

In this section, we establish some basic facts regarding the equation $ w(\Phi_{x}) = \Phi_{y}$. These facts will later be used to prove properties about Coxeter squares and Coxeter $n$-cubes in the next section.

\begin{prop} \label{lengthsadd}
    We have 
    
    $$w(\Phi_{x}) \subseteq \Phi^{+} \iff \ell(wx) = \ell(w) + \ell(x)$$ 
    In particular, $ w(\Phi_{x}) = \Phi_{y}$ implies $\ell(wx) = \ell(w) + \ell(x)$.
\end{prop}

\begin{proof}
    Note that $w(\Phi_{x}) \seq \Phi^{+} $ if and only if $ \Phi_{w^{-1}} \cap \Phi_{x} = \emptyset$. But by Lemma 3.2 of \cite{brink1999normalizers}, $ \Phi_{w^{-1}} \cap \Phi_{x} = \emptyset$ if and only if $ x$ is a right divisor of $ wx$, meaning that $ \ell(wx) = \ell(w) + \ell(x)$. (In \cite{brink1999normalizers}, their notation differs from the notation that we use. In \cite{brink1999normalizers}, they use $ N(w^{-1})$ to mean $ \Phi_{w}$).
\end{proof}

\begin{prop} \label{equationandcocycle}
    We have

    $$ w(\Phi_{x}) = \Phi_{y} \implies wN(x)w^{-1} = N(y)$$
    If $ \ell(wx) = \ell(w) + \ell(x)$, then:

    $$ wN(x)w^{-1} = N(y) \implies w(\Phi_{x}) = \Phi_{y}$$
\end{prop}

\begin{proof}
    The implication $w(\Phi_{x}) = \Phi_{y} \implies w N(x) w^{-1} = N(y)$ follows almost immediately from Proposition \ref{generalizedexchange} and also from the conjugation formula $$wt_{\alpha}w^{-1} = t_{w(\alpha)} \textrm{ where } \alpha \in \Phi$$
    For the reverse direction, suppose $w N(x) w^{-1} = N(y)$. Let 
    
    $$ \Phi_{x} = \{ \alpha_{1}, \dots , \alpha_{n} \}$$
    and
    $$ \Phi_{y} = \{ \beta_{1}, \dots ,\beta_{n} \} $$
    Then we have
    
    $$ N(x) = \{t_{\alpha_{1}}, \dots , t_{\alpha_{n}} \}  $$
    and
    $$ N(y) = \{t_{\beta_{1}}, \dots , t_{\beta_{n}}  \} $$
    Thus,
    
    $$ w N(x)w^{-1} = \{wt_{\alpha_{1}}w^{-1}, \dots , wt_{\alpha_{n}}w^{-1} \}  $$
    apply the conjugation formula to get:
    
    $$ w N(x)w^{-1} = \{t_{w(\alpha_{1})}, \dots , t_{w(\alpha_{n})} \}  $$
    Thus, we get:
    
    $$ \{t_{w(\alpha_{1})}, \dots , t_{w(\alpha_{n})} \} = \{t_{\beta_{1}}, \dots , t_{\beta_{n}}  \} $$
    Let us assume without loss of generality that $t_{w(\alpha_{i})} = t_{\beta_{i}}$ for each $i$. This implies that $w(\alpha_{i}) = \pm \beta_{i}$ for each $i$. Note that $\alpha_{i}$ and $\beta_{i}$ are elements of $\Phi^{+}$ for each $i$. Since $l(wx) = l(w) + l(x)$, we can apply Proposition \ref{lengthsadd} to conclude that $w(\alpha_{i}) \in \Phi^{+}$ for each $i$. But this then implies that $w(\alpha_{i}) = \beta_{i}$ for each $i$. Hence, $w(\Phi_{x}) = \Phi_{y}$.
\end{proof}

\begin{prop} \label{flippingthesquare}
    Let $w,x,y,z \in W$. Suppose further that $yz = wx $. Then the following conditions are equivalent:
    
    \begin{enumerate}
        \item $w(\Phi_{x}) = \Phi_{y}$
        
        \item $y(\Phi_{z}) = \Phi_{w}$
        
    \end{enumerate}
\end{prop}

\begin{remark}
    This was first proven as Lemma 12.3 of \cite{dyer2011groupoids2} (see both \cite{dyer2011groupoids1} and \cite{dyer2011groupoids2}). It was also proven as Lemma 3.4 part (a) in \cite{dyer2019characterization}. However, the proof that is given below is slightly different than the one given in \cite{dyer2019characterization}.
\end{remark}

\begin{proof}
    Suppose $w(\Phi_{x}) = \Phi_{y}$. Since $ yz = wx$, we have $ N(yz) = N(wx) $. Proposition \ref{reflectioncocycle} allows us to deduce:
    
    $$ N(y) \Delta yN(z)y^{-1} = N(w) \Delta wN(x)w^{-1} $$
    By Proposition \ref{equationandcocycle}, we have $ w N(x) w^{-1} = N(y) $. Thus, $ yN(z)y^{-1} = N(w)$. Because of $ w(\Phi_{x}) = \Phi_{y}$, Proposition \ref{lengthsadd} implies that $ \ell(wx) = \ell(w) + \ell(x)$. But since $ \ell(wx) = \ell(w) + \ell(x)$, Proposition \ref{reflectioncocycle} implies that $ N(w)$ and $ wN(x)w^{-1}$ are disjoint. Hence, $ yN(z)y^{-1}$ and $ N(y)$ are disjoint. Thus, by Proposition \ref{reflectioncocycle}, we deduce $ \ell(yz) = \ell(y) + \ell(z)$. Both $ \ell(yz) = \ell(y) + \ell(z)$ and $yN(z)y^{-1} = N(w) $ imply that $ y(\Phi_{z}) = \Phi_{w}$ by Proposition \ref{equationandcocycle}. This proves that (1) implies (2).

For the reverse direction, note that there is a symmetry of the assumption $yz = wx$ and the two conditions $ w(\Phi_{x}) = \Phi_{y}$ and $y(\Phi_{z}) = \Phi_{w}$. To be specific, if we interchange the roles of $ w$ and $y$ and the roles of $z$ and $x$, we get the reverse implication. This proves the equivalence.
\end{proof}

\begin{prop}
    Let $ W$ be a finite Coxeter group. Let $ w_{0}$ denote the unique longest element of $ W$. Let $ x,y \in W$ such that $ \ell(x) = \ell(w_{0}) -1$. If there exists a $ w\in W \sm \{ 1 \} $ such that $ w(\Phi_{x}) = \Phi_{y}$, then $ x = y$.
\end{prop}

\begin{proof}
    Let $ w ,x ,y$ satisfy the hypotheses of the proposition. Since $ w(\Phi_{x}) = \Phi_{y}$, we have $ \ell(wx) = \ell(w) + \ell(x)$ by Proposition \ref{lengthsadd}. Since $ w \neq 1$, $\ell(w) > 0$. Since $ \ell(x) = \ell(w_{0}) - 1$ where $ w_{0}$ is the longest element, we deduce $ \ell(w) = 1$. Thus, $ w = s $ for some $ s\in S$ such that $ sx = w_{0}$. The equation $ w(\Phi_{x}) = \Phi_{y}$ becomes:

    $$ s(\Phi_{x}) = \Phi_{y}$$
    Since $s$ is an involution, we can move $s$ over to the other side: $ \Phi_{x} = s(\Phi_{y})$. Also, note that $ \ell(y) = \ell(x) = \ell(w_{0}) -1$. We can now interchange the roles of $ x$ and $ y$ to conclude that $ s$ is the unique element of $ S$ such that $ sy = w_{0}$. Since $ sx = w_{0}  = sy$, we conclude that $ x = y$.
\end{proof}

\begin{prop}
    Let $ W$ be a finite Coxeter group with longest element $w_{0}$. Suppose that 
    
    $$ w(\Phi_{x}) = \Phi_{y}   $$
    then $w$ is a right divisor of $ w_{0}x^{-1}$.
\end{prop}

\begin{proof}
    The condition $w(\Phi_{x}) = \Phi_{y}$ implies that $\ell(wx) = \ell(w) + \ell(x)$. Let $g\in W$ such that $ \ell(w_{0}) = \ell(g(wx))$. Such a $g$ exists since we can always increase the length of an element by continually multiplying on the left by simple reflections until we reach the length of the longest element (page 16 of \cite{humphreys1992reflection}). Note that since $ \ell(w_{0}) = \ell(g(wx))$, and because the longest element of a Coxeter group is unique if it exists, we must have that $ w_{0} = g(wx)$. Note that both:
    
    $$ \ell(w_{0}) = \ell(g(wx)) = \ell(g) + \ell(wx) = \ell(g) + \ell(w) + \ell(x) $$
    
    $$ \ell(w_{0}) = \ell((w_{0}x^{-1})x) = \ell(w_{0}x^{-1}) + \ell(x) $$
    Hence, we deduce that $ \ell(w_{0}x^{-1}) = \ell(g) + \ell(w)$. Furthermore, $ (w_{0}x^{-1})x = w_{0} = g(wx)$, which implies $ w_{0}x^{-1} = gw$. Both $ w_{0}x^{-1} = gw$ and $\ell(w_{0}x^{-1}) = \ell(g) + \ell(w)$ imply that $w$ is a right divisor of $ w_{0}x^{-1}$.
\end{proof}

\begin{prop}
    Let $ W$ be a finite Coxeter group with longest element $w_{0}$. Then:
    
    $$ (w_{0}x^{-1})(\Phi_{x}) = \Phi_{w_{0}x^{-1}w_{0}} $$
    Furthermore, if $ w(\Phi_{x}) = \Phi_{y}$, then $(w_{0}y^{-1}wxw_{0}) (\Phi_{w_{0}x^{-1}w_{0}}) = \Phi_{w_{0}y^{-1}w_{0}}$.
\end{prop}

\begin{proof}
    First, observe that 
    
    $$\ell(w_{0}x^{-1}w_{0}) = \ell(w_{0}(x^{-1}w_{0})) = \ell(w_{0}) - \ell(x^{-1}w_{0}) $$

    $$ = \ell(w_{0}) -(\ell(w_{0}) - \ell(x^{-1})) = \ell(x^{-1}) = \ell(x)$$
Thus, we get $ \ell(w_{0}x^{-1}w_{0}) = \ell(x)$, which implies that $ |\Phi_{x}| = \ell(x) = \ell(w_{0}x^{-1}w_{0}) =|\Phi_{w_{0}x^{-1}w_{0}}|$ by Proposition \ref{equivalenceofwro}.
More precisely, $ \Phi_{x}$ and $ \Phi_{w_{0}x^{-1}w_{0}}$ are finite sets of the same cardinality. Thus, $(w_{0}x^{-1})(\Phi_{x})$ and $ \Phi_{w_{0}x^{-1}w_{0}} $ are both finite sets of the same cardinality. Therefore, to show that these sets are equal, it suffice to show only one containment. I will show that $(w_{0}x^{-1})(\Phi_{x}) \subseteq \Phi_{w_{0}x^{-1}w_{0}} $, thus proving that $(w_{0}x^{-1})(\Phi_{x}) = \Phi_{w_{0}x^{-1}w_{0}}$.

Let $ \beta \in \Phi_{x}$ (note $\beta \in \Phi^{+}$). Then $x^{-1}(\beta) \in \Phi^{-} $. But then we have that $ w_{0}x^{-1}(\beta) \in \Phi^{+}$, since the longest element $w_{0}$ takes all of the negative roots into $\Phi^{+}$. Hence, $(w_{0}x^{-1})(\Phi_{x}) \subseteq \Phi^{+}$. All that is left to do is prove that applying $(w_{0}x^{-1}w_{0})^{-1} $ to $w_{0}x^{-1}(\beta)$ gives us a negative root. But this calculation is trivial to verify using the fact that $w_{0}^{-1} = w_{0}$ and that $w_{0}$ sends all positive roots into $ \Phi^{-}$. Thus, $ (w_{0}x^{-1})(\Phi_{x}) \subseteq \Phi_{w_{0}x^{-1}w_{0}} $, so we conclude that $(w_{0}x^{-1})(\Phi_{x}) = \Phi_{w_{0}x^{-1}w_{0}}$.

If $ w(\Phi_{x}) = \Phi_{y}$, then by making the substitutions $ \Phi_{x} = (xw_{0})\Phi_{w_{0}x^{-1}w_{0}}$ and $ \Phi_{y} = (yw_{0})\Phi_{w_{0}y^{-1}w_{0}}$, one quickly gets $ (w_{0}y^{-1}wxw_{0}) (\Phi_{w_{0}x^{-1}w_{0}}) = \Phi_{w_{0}y^{-1}w_{0}}$. 
\end{proof}

\section{Coxeter Squares and Coxeter $n$-Cubes} \label{squaresandcubes}

We begin this section by standardizing some notation regarding the $ A_{n}$ Coxeter system. Whenever the $A_{n}$ Coxeter system is mentioned, we let $ \{ s_{1} , s_{2} , \dots , s_{n}  \}$ denote the Coxeter generators with Coxeter diagram:
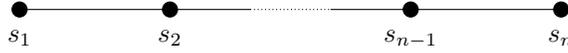
\vspace{.5cm}
\begin{center}
    \dynkin[text style/.style={scale=1}, labels={s_{1}, s_{2}, s_{n-1}, s_{n}}, scale=2, edge length=1cm ] A{}
\end{center}
\vspace{.5cm}
For background on Coxeter diagrams, see Section 2.1 of \cite{humphreys1992reflection}.

We define a \emph{Coxeter square} to be an ordered quadruple $ (w,x, y, z) \in (W \sm \{ 1 \} )^{4}$ such that $ wx = yz$ and $ w(\Phi_{x}) = \Phi_{y}$. Coxeter squares were first introduced by Dyer, Wang in Section 3 of \cite{dyer2019characterization}. Coxeter squares can be interpreted visualy as the following commutative diagram:

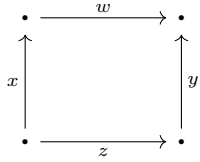
\[ \begin{tikzcd}
\sbullet \arrow{rr}{ w} & & \sbullet  \\
\\
\sbullet \arrow{uu}{x} \arrow{rr}[swap]{z}& &\sbullet \arrow{uu}[swap]{y}
\end{tikzcd}
\]
Proposition \ref{flippingthesquare} can be reinterpreted as follows: if $ (w,x,y,z)$ is a Coxeter square, then $ (y,z,w,x)$ is a Coxeter square. The Coxeter square $ (y,z, w,x)$ is shown below:

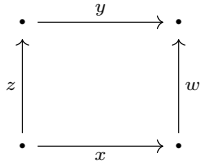
\[ \begin{tikzcd}
\sbullet \arrow{rr}{ y} & & \sbullet  \\
\\
\sbullet \arrow{uu}{z} \arrow{rr}[swap]{x}& & \sbullet \arrow{uu}[swap]{w}
\end{tikzcd}
\]
Note that the two statements $ wx = yz$ and $ w(\Phi_{x}) = \Phi_{y}$ are equivalent to the following two statements: $ w^{-1}y = xz^{-1}$ and $ w^{-1}(\Phi_{y}) = \Phi_{x}$. Similarly, the two statements $ yz = wx$ and $ y(\Phi_{z}) = \Phi_{w}$ are equivalent to the following two statements: $ y^{-1}w = zx^{-1}$ and $ y^{-1}(\Phi_{w}) = \Phi_{z}$. Thus, by Proposition \ref{flippingthesquare}, we have shown that if $ (w,x,y,z)$ is a Coxeter square, then $ (w^{-1} , y , x , z^{-1})$ and $ (y^{-1} ,w , z , x^{-1})$ are both Coxeter squares. The two Coxeter squares can be visualized as such:

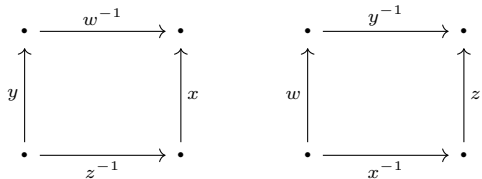
\[ \begin{tikzcd}
\sbullet \arrow{rr}{ w^{-1}} & & \sbullet  \\
\\
\sbullet \arrow{uu}{y} \arrow{rr}[swap]{z^{-1}}& & \sbullet \arrow{uu}[swap]{x}
\end{tikzcd}
\hspace{1cm}
 \begin{tikzcd}
\sbullet \arrow{rr}{ y^{-1}} & & \sbullet  \\
\\
\sbullet \arrow{uu}{w} \arrow{rr}[swap]{x^{-1}}& &\sbullet \arrow{uu}[swap]{z}
\end{tikzcd}
\]

The above observations lead to the following proposition.

\begin{prop} \label{reorientingsquare}
    If $ (w,x,y,z)$ is a Coxeter square, then $(y,z, w,x) $, $(w^{-1} , y , x , z^{-1}) $, and $ (y^{-1} ,w , z , x^{-1})$ are Coxeter squares.
\end{prop}

Given a Coxeter square, flipping a pair of parallel edges is called an \emph{edge flip}. Flipping a Coxeter square by its lower left/upper right diagonal onto its \textquotedblleft backside" is called a \emph{diagonal flip}. Performing a sequence of edge flips and diagonal flips on a Coxeter square is called a \emph{reorientation of the Coxeter square}. Proposition \ref{reorientingsquare} shows that reorienting a Coxeter square still results in a Coxeter square. Given a collection of six Coxeter squares with matching edges, we can construct a \emph{Coxeter cube}:

 
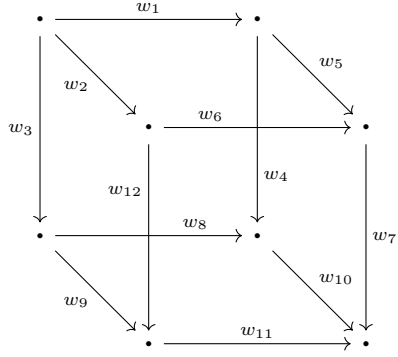
\[
\begin{tikzcd}[row sep=3em, column sep = 3em]
\sbullet \arrow[rr ,"w_{1}"] \arrow[dr, swap,"w_{2}"] \arrow[dd,swap, "w_{3}"] &&
\sbullet \arrow[dd, "w_{4}" near end] \arrow[dr,"w_{5}"] \\
& \sbullet \arrow[rr , "w_{6}" near start] \arrow[dd, swap, "w_{12}" near start] &&
\sbullet \arrow[dd,"w_{7}"] \\
\sbullet \arrow[rr, "w_{8}" near end] \arrow[dr, swap, "w_9"] &&  \sbullet \arrow[dr, "w_{10}"] \\
& \sbullet \arrow[rr, "w_{11}"] && \sbullet
\end{tikzcd}
\]
We can also construct higher dimensional \emph{Coxeter $n$-cubes} by gluing together Coxeter squares with matching edges. Similar to how Coxeter squares can be reoriented by edge flips and diagoanl flips, Coxeter $n$-cubes can be reoreinted by flipping collections of parallel edges or by applying some sort of isometry of the cube that fixes the initial and terminal vertices. This is called \emph{reorienting a Coxeter $n$-cube}. For simplicity, we define a Coxeter $ 1$-cube to be a an oriented line-segment with an edge label:

\[
\begin{tikzcd}[row sep=1.5em, column sep = 1.5em]
\sbullet \arrow[rr ,"w_{1}"]  &&
\sbullet 
\end{tikzcd}
\]
Given a Coxeter system $ (W,S)$, we define $ \Cubes{n}{W,S}$ to be \emph{the number of distinct Coxeter $n$-cubes modulo reorientation in the Coxeter system $ (W,S)$}. 

\begin{prop} \label{ncubeexists}
    Let $(W,S)$ be a finite Coxeter system $(|W| < \infty)$. Let $|S| = n$. Then $\Cubes{n}{W,S} \geq 1$.
\end{prop}

\begin{remark}
    If $ |W| = \infty$, then Coxeter $ n$-cubes do not exist in $ (W,S)$ where $ |S| = n$. This will be proven later.
\end{remark}

\begin{proof}
    We proceed by induction on $ n = |S|$.  If $n = 1$, then we are dealing with the $A_{1}$ Coxeter system where $ W = \{ 1,s \} $. Clearly, the following is a $1$-cube in $A_{1}$:

    \[
\begin{tikzcd}[row sep=1.5em, column sep = 1.5em]
\sbullet \arrow[rr ,"s"]  &&
\sbullet 
\end{tikzcd}
\]
 Suppose now that we have proven the proposition true when $|S| = k $. We desire to prove the case when $ |S| = k+1$. Let $ |S| = k+1$ and let $ J \seq S$ such that $ |J| = k$. Since $ (W_{J} ,J)$ is a finite Coxeter system with $ |J| = k$, the induction hypothesis implies that there is at least one Coxeter $k$-cube inside the Coxeter system $ (W_{J} ,J)$. For the sake of visual simplicity, let us suppose that that $k =3$. Thus, the Coxeter $ k$-cube inside $ (W_{J} ,J)$ will have the following cubical commutative diagram:

  
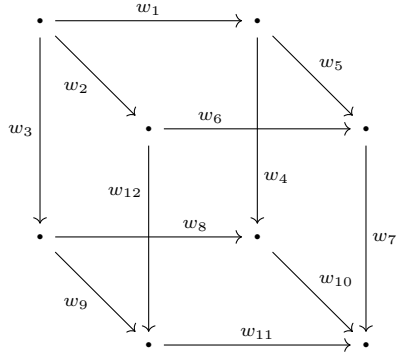
\[
\begin{tikzcd}[row sep=3em, column sep = 3em]
\sbullet \arrow[rr ,"w_{1}"] \arrow[dr, swap,"w_{2}"] \arrow[dd,swap, "w_{3}"] &&
\sbullet \arrow[dd, "w_{4}" near end] \arrow[dr,"w_{5}"] \\
& \sbullet \arrow[rr , "w_{6}" near start] \arrow[dd, swap, "w_{12}" near start] &&
\sbullet \arrow[dd,"w_{7}"] \\
\sbullet \arrow[rr, "w_{8}" near end] \arrow[dr, swap, "w_9"] &&  \sbullet \arrow[dr, "w_{10}"] \\
& \sbullet \arrow[rr, "w_{11}"] && \sbullet
\end{tikzcd}
\]
Let $ s $ be the unique element of $ S \sm J$. Note that since $ |W| < \infty$, it follows that $ \Phi_{J \cup \{ s \} } \sm \Phi_{J}$ is finite, so $ \nu(\alpha_{s}, \Pi_{J})$ exists (and $\nu(\alpha_{s} , \Pi_{J}) \neq 1$). Let $K \seq S$ such that $ \Pi_{K} = \nu(\alpha_{s} , \Pi_{J}) \Pi_{J}$. Since $ \nu(\alpha_{s} , \Pi_{J})$ sends $ \Pi_{J}$ to $ \Pi_{K}$, conjugation by $ \nu(\alpha_{s} , \Pi_{J})$ induces an isomorphism from $ (W_{J} ,J)$ to $ (W_{K}, K)$. Thus, if $ w \in W_{J}$, it follows that $ \nu(\alpha_{s} , \Pi_{J})w\nu(\alpha_{s} , \Pi_{J})^{-1} = w'$ for some $ w' \in W_{K}$. This can be rewritten as $ \nu(\alpha_{s} , \Pi_{J})w = w' \nu(\alpha_{s} , \Pi_{J})$. Furthermore, the induced isomorphism also implies $ \nu(\alpha_{s} , \Pi_{J}) (\Phi_{w}) = \Phi_{w'}$. Hence, $ (\nu(\alpha_{s} , \Pi_{J}) , w , w' , \nu(\alpha_{s} , \Pi_{J}))$ is a Coxeter square. Furthermore, since $ \nu( \alpha_{s} , \Pi_{J})$ establishes an isomorphism from $ (W_{J} ,J)$ to $ (W_{K}, K)$, $ \nu(\alpha_{s} , \Pi_{J})$ transfers Coxeter $ k$-cubes in $ (W_{J} ,J)$ to Coxeter $k$-cubes in $ (W_{K}, K)$. Thus, in the following diagram, all commutative squares are Coxeter squares: 

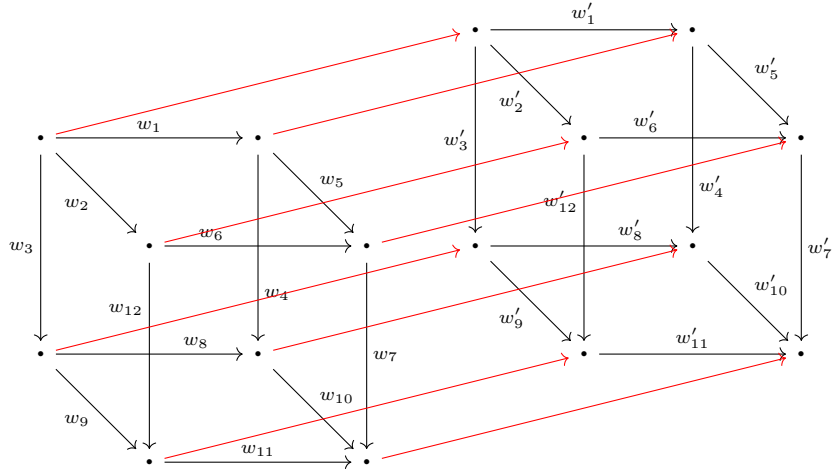
\[
\begin{tikzcd}[row sep=3em, column sep = 3em]
& & & & \sbullet \arrow[rr ,"w_{1}'"] \arrow[dr, swap,"w_{2}'"] \arrow[dd,swap, "w_{3}'"] & & \sbullet \arrow[dd, "w_{4}'" near end] \arrow[dr,"w_{5}'"]\\
\sbullet \arrow[rr ,"w_{1}"] \arrow[dr, swap,"w_{2}"] \arrow[dd,swap, "w_{3}"] \arrow[urrrr, red] & &
\sbullet \arrow[dd, "w_{4}" near end] \arrow[dr,"w_{5}"] \arrow[urrrr, red] & & & \sbullet \arrow[rr , "w_{6}'" near start] \arrow[dd, swap, "w_{12}'" near start]& & \sbullet \arrow[dd,"w_{7}'"]\\
& \sbullet \arrow[rr , "w_{6}" near start] \arrow[dd, swap, "w_{12}" near start] \arrow[urrrr, red] & &
\sbullet \arrow[dd,"w_{7}"] \arrow[urrrr, red] & \sbullet \arrow[rr, "w_{8}'" near end] \arrow[dr, swap, "w_9'"] & & \sbullet \arrow[dr, "w_{10}'"] \\
\sbullet \arrow[rr, "w_{8}" near end] \arrow[dr, swap, "w_9"] \arrow[urrrr, red] & &  \sbullet \arrow[dr, "w_{10}"] \arrow[urrrr, red] & & & \sbullet \arrow[rr, "w_{11}'"]  & & \sbullet \\
& \sbullet \arrow[rr, "w_{11}"]  \arrow[urrrr, red]& & \sbullet \arrow[urrrr, red]
\end{tikzcd}
\]
The arrows in red have an edge label of $\nu(\alpha_{s} , \Pi_{J}) $. Hence, the above is a Coxeter $(k+1)$-cube. This proves the induction.

\end{proof}

\begin{prop} \label{highercubesdont}
    Let $ (W,S)$ be a Coxeter system. Let $ |S| = n$. If $ k > n$, then $ \Cubes{k}{W,S} = 0$.
\end{prop}

\begin{proof}
    Note that if $(w,x,y,z) \in (W \sm \{ 1 \} )^{4}$ is a Coxeter square, then the equation $ w(\Phi_{x}) = \Phi_{y}$ implies $ w^{-1}(\Phi_{y}) \seq \Phi^{+}$, which gives $ \Phi_{w} \cap \Phi_{y} = \emptyset$. This shows that if $ w$ and $ y$ are the terminal edges of a Coxeter square, then $ \Phi_{w}$ and $ \Phi_{y}$ must be disjoint.

    Let $ x_{1}, x_{2} , \dots , x_{k}$ denote the terminal edges of a Coxeter $ k$-cube. Note that if $ i \neq j$, then $ x_{i}$ and $ x_{j}$ will be the terminal edges of some Coxeter square. Hence, $ \Phi_{x_{i}}$ and $ \Phi_{x_{j}}$ will be disjoint when $ i \neq j$. Thus, the left inversion sets $ \Phi_{x_{i}}$ for $ i=1,2, \dots, k$ are all pairwise disjoint. But note that each edge of a Coxeter $ k$-cube is a non-identity element. Hence, $ \Phi_{x_{i}}$ must contain at least one simple root for each $ i =1,2, \dots , k$. If $ k > n = |S|$, then at least one distinct pair $ x_{i}$ and $ x_{j}$ will have the same simple root in their left inversion sets. This contradicts disjointness of the left inversion sets. Thus, $  k\leq n$.
\end{proof}

Given a finite Coxeter system $ (W,S)$ with $ |S| = n$, Propositions \ref{ncubeexists} and \ref{highercubesdont} imply that Coxeter $ n$-cubes exist and are the largest higher dimensional Coxeter cubes inside $ (W,S)$. We now give some examples.

\begin{ex} \label{A_{2}square}
    Consider the $ A_{2}$ Coxeter system with Coxeter generators $ S = \{ s_{1} , s_{2} \} $.  Since $ |S| = 2$, we know that there must exist at least one Coxeter square and that $ \Cubes{n}{A_{2}} = 0$ if $ n \geq 3$. We construct a Coxeter square in $ A_{2}$ using the inductive method as described in the proof of \ref{ncubeexists}. Consider $ J : = \{ s_{1} \} \seq S$. Clearly, there is a Coxeter $ 1$-cube inside $ (W_{J}, J)$:

    \[
\begin{tikzcd}[row sep=1.5em, column sep = 1.5em]
\sbullet \\
\\
\sbullet \arrow[uu, "s_{1}"] 
\end{tikzcd}
\]
Since $A_{2}$ is a finite Coxeter system, we know that $ \nu( \alpha_{2} ,  \Pi_{J} )$ exists where $ \alpha_{2}$ is the simple root corresponding to $ s_{2}$. More precisely, $ \nu( \alpha_{2} , \Pi_{J}) = s_{1}s_{2}$. Since $ \nu(\alpha_{2} , \Pi_{J}) \{ \alpha_{1} \} = \{ \alpha_{2} \} $, we get the following Coxeter square:

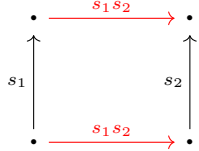
\[ \begin{tikzcd}
\sbullet \arrow[rr, red, "s_{1}s_{2}"] & & \sbullet \\
\\
\sbullet \arrow[uu, "s_{1}"] \arrow[rr, red, "s_{1}s_{2}"] & & \sbullet \arrow[uu, "s_{2}"]
\end{tikzcd}
\]
It is not too hard to show that any other Coxeter square in the $A_{2}$ Coxeter system is a reoreintation of the above Coxeter square. Hence, $ \Cubes{2}{A_{2}} = 1$ for the $ A_{2}$ Coxeter system.

\end{ex}

\begin{ex} \label{cubesexample}
    Let us now consider the $ A_{3}$ Coxeter system with Coxeter generators $ S = \{ s_{1}, s_{2}, s_{3} \}$. Using the inductive method as described in the proof of \ref{ncubeexists}, one can find the following two Coxeter cubes:

\[ \begin{tikzcd}
\sbullet \arrow[dd, "s_{1}"] \arrow[dr , "s_{1}s_{2}"] \arrow[rrr, "s_{1}s_{2}s_{3}"] &  & & \sbullet \arrow[dd, "s_{2}" near end] \arrow[dr, "s_{2}s_{3}"]\\
 &  \sbullet \arrow[dd, swap, "s_{2}" near start]  \arrow[rrr, "s_{1}s_{2}s_{3}" near start]  &  & &\sbullet \arrow[dd, "s_{3}"]\\
 \sbullet \arrow[dr , swap, "s_{1}s_{2}"] \arrow[rrr, "s_{1}s_{2}s_{3}" near end] &  & &\sbullet \arrow[dr, "s_{2}s_{3}" near start]\\
 &  \sbullet \arrow[rrr, "s_{1}s_{2}s_{3}"] &  & &\sbullet 
\end{tikzcd}
\hspace{1cm}
\begin{tikzcd}
\sbullet \arrow[dd, "s_{1}"] \arrow[dr , "s_{3}"] \arrow[rrr, "s_{2}s_{1}s_{3}s_{2}"] &  & & \sbullet \arrow[dd, "s_{3}" near end] \arrow[dr, "s_{1}"]\\
 &  \sbullet \arrow[dd, swap, "s_{1}" near start]  \arrow[rrr, "s_{2}s_{1}s_{3}s_{2}" near start]  &  & &\sbullet \arrow[dd, "s_{3}"]\\
 \sbullet \arrow[dr , swap, "s_{3}"] \arrow[rrr, "s_{2}s_{1}s_{3}s_{2}" near end] &  & &\sbullet \arrow[dr, "s_{1}"]\\
 &  \sbullet \arrow[rrr, "s_{2}s_{1}s_{3}s_{2}"] &  & &\sbullet 
\end{tikzcd}
\]
These two cubes were first noticed by Dyer in Example 12.5 of \cite{dyer2011groupoids2}. We will later prove that these are the only 2 distinct Coxeter cubes in $ A_{3}$ modulo reorientation, hence $ \Cubes{3}{A_{3}} = 2$ in $ A_{3}$.

\end{ex}

\begin{prop} \label{wroandsquare}
    Let $ (W,S)$ be a Coxeter system. If $ (w,x,y,z)$ is a Coxeter square, then $ w \vee y = wx = yz$ where $ w \vee y$ denotes the join of $ w$ and $ y$ in the weak right order. Furthermore,
    
    $$ \Phi_{w\vee y} = \Phi_{w} \coprod \Phi_{y}$$
\end{prop}

\begin{remark}
    This was first proven as Lemma 3.5 of \cite{dyer2019characterization}.
\end{remark}

\begin{proof}
    Since $ w(\Phi_{x}) = \Phi_{y}$, Propositions \ref{lengthsadd} and \ref{equationandcocycle} imply that $ \ell(wx) = \ell(w) + \ell(x)$ and $ wN(x)w^{-1} = N(y)$. Thus, by Proposition \ref{reflectioncocycle}, we have:

    $$ N(wx) = N(w) \coprod wN(x)w^{-1}$$

    $$ = N(w) \coprod N(y)$$
    Hence, $ N(wx) = N(w) \coprod N(y)$. But by Proposition \ref{equivalenceofwro} part (1), we deduce that $ \Phi_{wx} = \Phi_{w} \coprod \Phi_{y}$. By Proposition \ref{equivalenceofwro} part (2), we conclude that $ w \vee y = wx = yz$.
\end{proof}

\begin{prop} \label{wroandcube}
    Let $ (W,S)$ be a Coxeter system. If $ x_{1} , x_{2} , \dots , x_{n}$ are the terminal edges of a Coxeter $n$-cube in $ (W,S)$, then 

    $$ \Phi_{\bigvee_{i=1}^{n} x_{i}} = \coprod_{i=1}^{n} \Phi_{x_{i}}$$
    where $ \bigvee_{i=1}^{n} x_{i}$ denotes the join of the $ x_{i}$'s in the weak right order. Furthermore, if $ |W| < \infty$ and $ |S| = n$, then $w_{0} =\bigvee_{i=1}^{n} x_{i}$ where $ w_{0}$ is the longest element of $ W$. Hence,

    $$ \Phi^{+} = \coprod_{i=1}^{n} \Phi_{x_{i}}$$
\end{prop}

\begin{proof}
    Note that if $ x_{1}, x_{2}, \dots, x_{n}$ are the terminal edges of a Coxeter $n$-cube, then $ x_{1} , x_{2} , \dots , (x_{n-1}\vee x_{n})$ are the terminal edges of a Coxeter $(n-1)$-cube. Indeed, by Proposition \ref{wroandsquare}, $ x_{n-1}\vee x_{n}$ is the composition of two compatible morphisms in the Coxeter square determined by $ x_{n-1}$ and $x_{n}$. This process can be viewed geometrically as \textquotedblleft collapsing" the Coxeter $n$-cube along the Coxeter square that contains $ x_{n-1}$ and $ x_{n}$. In the case when $ n=3$, the \textquotedblleft collapsing" of the cube can be visualized in the following diagram:

    
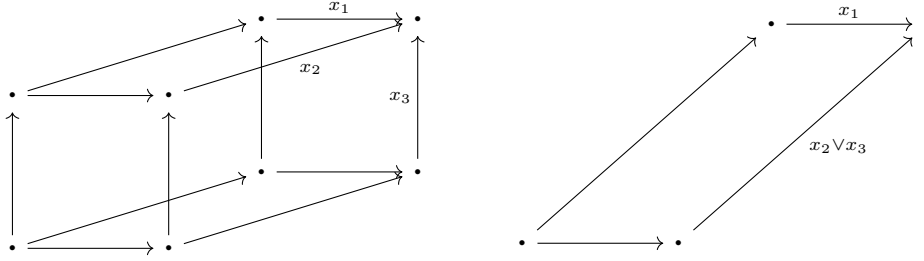
\[    
    \begin{tikzcd}
& & & \sbullet \arrow[rr, "x_{1}"] & & \sbullet \\
\sbullet \arrow[urrr] \arrow[rr] & & \sbullet \arrow[urrr, swap, "x_{2}"] \\
& & & \sbullet \arrow[uu] \arrow[rr] & & \sbullet \arrow[uu, "x_{3}"]\\
\sbullet \arrow[uu] \arrow[rr] \arrow[urrr] & & \sbullet \arrow[uu] \arrow[urrr]
\end{tikzcd}
\hspace{1cm}
\begin{tikzcd}
    & & & \sbullet \arrow[rr, "x_{1}"] & & \sbullet \\
\\
\\
\\
\sbullet  \arrow[rr] \arrow[uuuurrr] & & \sbullet \arrow[uuuurrr, swap, "x_{2} \vee x_{3}"]
\end{tikzcd}
\]
Thus, we can recursively apply Proposition \ref{wroandsquare} to deduce that

$$ \Phi_{\bigvee_{i=1}^{n} x_{i}} = \coprod_{i=1}^{n} \Phi_{x_{i}}$$
Suppose now that $|W| < \infty$ and $ |S| = n$. Each $ \Phi_{x_{i}}$ contains at least one simple root, and since the $\Phi_{x_{i}}$'s are disjoint, it follows that each simple root must occur in exactly one of the $ \Phi_{x_{i}}$'s. Thus, $ \Phi_{\bigvee_{i=1}^{n}x_{i}}$ contains all of the simple roots. Thus, $ (\bigvee_{i=1}^{n}x_{i})^{-1}(\Pi) \seq \Phi^{-}$, which implies $ (\bigvee_{i=1}^{n}x_{i})^{-1}(\Phi^{+}) \seq \Phi^{-}$. Hence, $ (\bigvee_{i=1}^{n}x_{i})^{-1} = w_{0}$. But since $ w_{0}$ is an involution, we deduce that $ \bigvee_{i=1}^{n}x_{i} = w_{0}$. Since $ \Phi_{w_{0}} = \Phi^{+}$, we deduce:

$$ \Phi^{+} = \coprod_{i=1}^{n} \Phi_{x_{i}}$$

\end{proof}

Let $ (W,S)$ be a finite Coxeter system with $ |S| = n$. Propositions \ref{wroandsquare} and \ref{wroandcube} imply that if $ x_{1}, x_{2} , \dots, x_{n}$ are the terminal edges of a Coxeter $n$-cube in $ (W,S)$, then $ w_{0} =\bigvee_{i=1}^{n} x_{i}$ is the group element obtained by composing the group elements along a path from the initial vertex to the terminal vertex of the Coxeter $ n$-cube. 

\begin{cor} \label{summinglengths}
    Let $ (W,S)$ be a finite Coxeter system ($|W|<\infty$) with $ |S| = n$. If $ x_{1} , x_{2} , \dots , x_{n}$ are the terminal edges of a Coxeter $n$-cube in $ (W,S)$, then

    $$ \ell(w_{0}) = \sum_{i=1}^{n} \ell(x_{i})$$
    where $ w_{0}$ denotes the longest element of $ W$.
\end{cor}

\begin{proof}
    Follows quickly from the fact that $ \Phi_{w_{0}} =\Phi^{+} = \coprod_{i=1}^{n}\Phi_{x_{i}}$.
\end{proof}

\begin{cor} \label{infinitedoesnt}
    Let $ (W,S)$ be a Coxeter system with $ |S| = n$. If $ |W| = \infty$, then Coxeter $n$-cubes do not exist in $(W,S)$, meaning that $ \Cubes{n}{W,S} = 0$.
\end{cor}

\begin{proof}
    Suppose that there existed a Coxeter $n$-cube in $W$. Let $ x_{1} , x_{2} , \dots , x_{n}$ be the terminal edges of this $n$-cube. Then by Proposition \ref{wroandcube}, we deduce that

    $$ \Phi_{\bigvee_{i=1}^{n}x_{i}}  = \coprod_{i=1}^{n} \Phi_{x_{i}}$$
    Each $ \Phi_{x_{i}}$ must contain at least one simple root. Since the $ \Phi_{x_{i}}$'s are disjoint, and since $ |S| = n$, it follows each simple root must occur as an element $\coprod_{i=1}^{n} \Phi_{x_{i}} $. This implies that $ (\bigvee_{i=1}^{n}x_{i})^{-1}(\Pi) \seq \Phi^{-}$, which implies $ (\bigvee_{i=1}^{n}x_{i})^{-1}(\Phi^{+}) \seq \Phi^{-} $. But since $ |W| = \infty$, it follows that $ |\Phi^{+}| = \infty$. Thus:

    $$ \ell(\bigvee_{i=1}^{n}x_{i}) = |\Phi^{+}| = \infty$$
    which is a contradiction.
    
\end{proof}

We would like to show that if $ X$ is a Coxeter $n$-cube with terminal edges $ x_{1} , x_{2} , \dots, x_{n}$, then $X$ is uniquely determined by $ x_{1} , x_{2} , \dots , x_{n}$. Before we prove this, we need to establish some notation regarding oriented $n$-dimensional cubes. The vertices of an oriented $n$-dimenisonal cube are labeled by binary strings of length $ n$. The edges of an oriented $n$-dimensional are given labels of the form $ e_{a_{1}a_{2}\dots a_{n}}$ where $ a_{i} \in \{ 0,1,\star \}$ for all $1\leq i \leq n$ and there is exactly one $ j$ such that $ a_{j} = \star$. If $ a_{j} = \star$, then we say that $ e_{a_{1}a_{2}\dots a_{j-1}\star a_{j+1}\dots a_{n}}$ is the oriented edge from the vertex $ a_{1}a_{2}\dots a_{j-1}0a_{j+1}\dots a_{n}$ to the vertex $ a_{1}a_{2}\dots a_{j-1}1a_{j+1}\dots a_{n}$. An example of such a labelling for an oriented $3$-cube is given below:

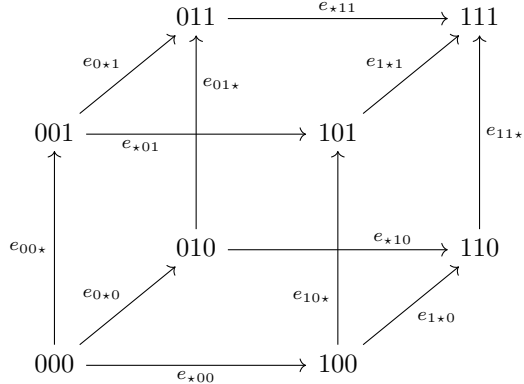
\vspace{.5cm}
 \[
\begin{tikzcd}[row sep=3em, column sep = 3em]
& 011 \arrow[rr, "e_{\star 1 1}"] & & 111 \\
001 \arrow[ur, "e_{0\star 1}"] \arrow[rr, "e_{\star 0 1}" near start, swap] & & 101 \arrow[ur, "e_{1\star  1}"] \\
& 010 \arrow[uu, "e_{01 \star}" near end, swap] \arrow[rr , "e_{\star 1 0}" near end] & & 110 \arrow[uu, "e_{1 1 \star}", swap] \\
000 \arrow[rr, "e_{\star 00}", swap] \arrow[ur, "e_{0\star 0}"] \arrow[uu, "e_{00\star}"]& & 100 \arrow[uu, "e_{10\star}" near start] \arrow[ur, "e_{1\star 0}", swap]
\end{tikzcd}
\]
\vspace{.5cm}

Note that by this labeling convention, the string of all $0$'s denotes the initial vertex of an oriented $n$-cube and the string of all $ 1$'s denotes the terminal vertex of an oriented $n$-cube.

\begin{prop} \label{cubedetermined}
    Let $ (W,S)$ be a Coxeter system. Let $ X$ be a Coxeter $n$-cube in $ (W,S)$. If $ X$ has terminal edges $ x_{1} , x_{2} , \dots , x_{n}$, then $X$ is entirely determined by these terminal edges.
\end{prop}

\begin{proof} 
The statement is trivial for the case $n=1$. Note that for the case $n=2$, we desire to show that a Coxeter square is uniquely determined by its terminal edges. Suppose that $ x_{1}$ and $ x_{2}$ are the terminal edges of the following Coxeter square: 
    
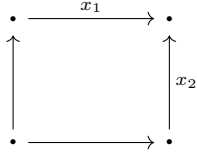
\[ \begin{tikzcd}
\sbullet \arrow{rr}{ x_{1}} & &\sbullet  \\
\\
\sbullet \arrow{uu} \arrow{rr}& &\sbullet \arrow{uu}[swap]{x_{2}}
\end{tikzcd}
\]
But by Proposition \ref{wroandsquare}, we can deduce fill in the other two edges: 
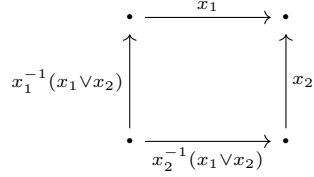
\[ \begin{tikzcd}
\sbullet \arrow{rr}{ x_{1}} & &\sbullet  \\
\\
\sbullet \arrow{uu}{x_{1}^{-1}(x_{1}\vee x_{2})} \arrow{rr}[swap]{x_{2}^{-1}(x_{1}\vee x_{2})}& &\sbullet \arrow{uu}[swap]{x_{2}}
\end{tikzcd}
\]
This shows that any Coxeter square is uniquely determined by its terminal edges. 

To consider the case for general $n$, we use induction. We assume that given the terminal edges $ x_{1} , x_{2} , \dots , x_{n}$ of a Coxeter $n$-cube, we can uniquely reconstruct the rest of the Coxeter $n$-cube.

Suppose now that $ y_{1} , y_{2} , \dots , y_{n}, y_{n+1}$ are the terminal edges of a Coxeter $ (n+1)$-cube. Without loss of generality, assume that $ y_{1} , y_{2} , \dots , y_{n}, y_{n+1}$ correspond to the edge labels $ e_{\star 1 \dots 1}$, $ e_{1\star 1\dots 1}$, $\dots$, $ e_{1 \dots 1 \star 1}$, $ e_{1 \dots 1\star}$. For each $1\leq  j \leq n+1$, consider the following subset

$$ \{ e_{\star 1 \dots 1}, e_{1\star 1\dots 1} , \dots , \widehat{e_{1 \dots 1\star 1 \dots 1}}, \dots , e_{1 \dots 1 \star 1} ,  e_{1 \dots 1\star}\}$$
where $ \widehat{e_{1 \dots 1\star 1 \dots 1}}$ denotes omission of the oriented edge with a $\star$ at the $j$-th index. The above subset forms the terminal edges of a Coxeter $n$-cube. Hence, we can apply the induction hypothesis to reconstruct all of the edges of the $n$-cube defined by the above subset. Since we omitted the edge with a $\star$ at the $j$-th index, we can reconstruct all edges of the form $ e_{a_{1}a_{2} \dots a_{j-1}1a_{j+1} \dots a_{n}a_{n+1}}$ where $ a_{i} \in \{ 0,1,\star \}$ for all $i \neq j$. As $j$ varies from $1$ to $ n+1$, we can repeatedly apply the induction hypothesis to reconstruct all edges of the form $ e_{1a_{2} \dots a_{n+1}}$, all edges of the form $e_{a_{1}1a_{3} \dots a_{n+1}} $, all edges of the form $ e_{a_{1}a_{2}1a_{4} \dots a_{n+1}}$ and so on until we've reconstructed all edges of the form $ e_{a_{1}a_{2} \dots a_{n}1}$. Thus, we have reconstructed all edges that have at least one $1$ as one of their indices. We just need to show that we can reconstruct an edge of the form $ e_{0 \dots 0 \star 0 \dots 0}$ where the $\star$ is in the $j$-th location.

Consider $e_{0 \dots 0 \star 0 \dots 0} $ where the $\star$ is in the $j$-th location. Let $k$ be a positive integer such that $ 1 \leq k \leq n+1$ and $ k \neq j$. Consider the vertex of the $(n+1)$-cube that is labeled $ 1\dots 101 \dots 1$ where the $ 0$ is in the $k$-th location. By the previous paragraph, we have already reconstructed all of the edges that point towards and are incident to the point  $ 1\dots 101 \dots 1$. These edges are of the form $ e_{\star 1 \dots 101 \dots 1}$,  $\dots$, $e_{ 1 \dots 1\star01 \dots 1} $, $ e_{ 1 \dots 10\star1 \dots 1}$, $ \dots$, $ e_{ 1 \dots 101 \dots 1\star}$. Note that there are exactly $ n$ of these oriented edges that point towards and are incident to the point $ 1\dots 101 \dots 1$. By the induction hypothesis, the $n$ edges $ e_{\star 1 \dots 101 \dots 1}$,  $\dots$, $e_{ 1 \dots 1\star01 \dots 1} $, $ e_{ 1 \dots 10\star1 \dots 1}$, $ \dots$, $ e_{ 1 \dots 101 \dots 1\star}$ determine a Coxeter $n$-cube whose terminal vertex is the point $1\dots 101 \dots 1$. Hence, we can reconstruct all of the edges of the Coxeter $n$-cube associated to the point $ 1\dots 101 \dots 1$. Since the zero of $1\dots 101 \dots 1$ occurs at the $k$-th location, it follows that we will have reconstucted all edges of the form $ e_{a_{1} \dots a_{k-1} 0 a_{k+1} \dots a_{n+1}}$ where $ a_{i} \in \{ 0,1,\star \}$ for $ i \neq k$. Since $k \neq j$, we will in particular have reconstructed the edge $e_{0 \dots 0 \star 0 \dots 0} $ where the $\star$ is in the $j$-th location. This shows that we can reconstruct any edge of a Coxeter $ (n+1)$-cube given the terminal edges $ y_{1}, y_{2} , \dots y_{n}, y_{n+1}$. By induction, for any given $n$, we can reconstruct a Coxeter $n$-cube from its terminal edges. In particular, if $ x_{1}, x_{2} , \dots , x_{n}$ are the terminal edges of the Coxeter $n$-cube $X$, it follows that any edge of $X$ can be expressed as some combination of the $x_{i}$'s by taking joins in weak order, taking inverses, and taking compositions.
\end{proof}

\begin{cor} \label{oneless}
    Let $ (W,S)$ be a finite Coxeter system with $ |S| = n$. Let $ X$ be a Coxeter $n$-cube in $ (W,S)$. Let $ \{ x_{1} , x_{2} , \dots, x_{n} \}$ be the set of terminal edges of $X$. If $Q \seq \{ x_{1} , x_{2} , \dots, x_{n} \}$ such that $ |Q| = n-1$, then $ Q$ uniquely determines $X$.
\end{cor}

\begin{proof}
    Without loss of generality, assume that $ Q = \{ x_{1}, x_{2} , \dots , x_{n-1} \}$. Since we are assuming that $ |W| < \infty$ and $ |S| = n$, we can invoke Proposition \ref{wroandcube} to deduce that 

    $$ \Phi^{+} \sm (\coprod_{i=1}^{n-1} \Phi_{x_{i}}) 
 = \Phi_{x_{n}}$$
 But since $ \Phi_{x_{n}}$ uniquely determines $ x_{n}$, we have used $Q$ to reconstruct the entire set of terminal edges $ \{ x_{1}, x_{2} , \dots ,x_{n} \}$. Thus, by Proposition \ref{cubedetermined}, we can then determine the rest of the Coxeter $n$-cube.
\end{proof}

\begin{prop} \label{depth2A_{n}}
    Consider the $ A_{n}$ Coxeter system with set of Coxeter generators $ \{ s_{1}, s_{2}  ,  \dots , s_{n} \}$. Let $ x_{1} , x_{2} , \dots , x_{n}$ be the terminal edges of some Coxeter $n$-cube in $A_{n}$. Then there exists some $ i$ such that $ x_{i} \in S$.
\end{prop}

\begin{proof}
    Suppose for the sake of contradiction that $ x_{i} \notin S$ for each $ i$. Since no edge of a Coxeter $ n$-cube can be equal to the identity (which follows from the definition of a Coxeter square), we deduce that $ \ell(x_{i}) \geq 2$ for all $i$. Furthermore, since each $\Phi_{x_{i}}$ must contain exactly one simple root (see proof of Proposition \ref{wroandcube}), $ \ell(x_{i}) \geq 2$ for all $i$ implies that $ \Phi_{x_{i}}$ must contain at least one positive root of depth 2 for each $i$ (see Definition 4.6.1 of \cite{bjorner2005combinatorics} for the definition of the depth of a positive root). But since the $ \Phi_{x_{i}}$'s are disjoint, this would imply that there exists at least $ n$ distinct positive roots of depth 2 in the root poset of $ A_{n}$. This contradicts the well-known \textquotedblleft pyramid" structure of the root poset for $A_{n}$ (see Figure 4.4 of \cite{bjorner2005combinatorics}). Hence, there must exist at least one $ i$ such that $ x_{i} \in S$. 
\end{proof}

\begin{prop}
    The only two Coxeter $3$-cubes that exist in the $A_{3}$ Coxeter system modulo orientation are the two cubes given in Example \ref{cubesexample}. Hence, $ \Cubes{3}{A_{3}}  =2$ for the $A_{3}$ Coxeter system.
\end{prop}

\begin{proof}
    Let $ S = \{ s_{1}, s_{2} , s_{3} \} $ be the Coxeter generators for $ A_{3}$ where $ s_{1}s_{3} = s_{3}s_{1}$. Let $ X$ be a Coxeter $3$-cube. Let $ x_{1} , x_{2}, x_{3}$ be the terminal edges of $X$. By Proposition \ref{depth2A_{n}}, we know that there exists at least one $i$ such that $ x_{i} \in S$. Without loss of generality, assume that $ x_{1} \in S$. By Corollary \ref{summinglengths}, we know that $ 6 = \ell(x_{1}) + \ell(x_{2}) + \ell(x_{3})$ since $ \ell(w_{0}) = 6$ in the $A_{3}$ Coxeter system. Since $ x_{1} \in S$, we have $ \ell(x_{1}) = 1$. Thus,

    $$ 5 = \ell(x_{2}) + \ell(x_{3})$$
    where $ \ell(x_{2}) \geq 1$ and $ \ell(x_{3}) \geq 1$ since no edge of a Coxeter cube can be the identity. The above equation and the inequalities on $\ell(x_{2})$, $ \ell(x_{3})$ imply that $ \ell(x_{i}) \in \{ 1,2 \}$ for some $i$. Without loss of generality, let us assume that $ \ell(x_{2}) \in \{  1,2 \} $.

    Suppose first that $ \ell(x_{2}) =1$. Since $ \ell(x_{1}) =1$ and $ \ell(x_{2})  =1$, we will write $ x_{1} = s \in S$ and $ x_{2} = s' \in S$. Since $ s$ and $ s'$ are terminal edges of a Coxeter cube, $s$ and $ s'$ must be the terminal edges of a Coxeter square. But since $ s$ and $ s'$ are the terminal edges of a Coxeter square, it follows that there must exist some element $\alpha \in \Pi$ such that $ s(\alpha) = \alpha_{s'}$ where $ \alpha_{s'}$ is the simple root corresponding to $ s'$. In other words, we must have that $ s(\alpha_{s'}) = s^{-1}(\alpha_{s'}) \in \Pi$. But the only two simple reflections in $A_{3}$ with this property are $ s_{1}$ and $ s_{3}$. Thus, $ \{ s , s' \} = \{ s_{1} , s_{3} \} $. But since we have determined two of the terminal edges of a Coxeter $n$-cube, Corollary \ref{oneless} implies that we have determined all of $ X$. Thus, $ X$ must be equal to the cube on the right in Example \ref{cubesexample}.

    Suppose now that $ \ell(x_{2}) = 2$. Since $ \ell(x_{1}) = 1$ and $ \ell(x_{2}) = 2$, we write $ x_{1} = s\in S$ and $ x_{2} = s's''$ where $ s' , s''\in S$. We know that $ s$ and $ s's''$ must be the terminal edges of some Coxeter square. Thus, $ s''s'(\alpha_{s}) = (s's'')^{-1}(\alpha_{s}) \in \Pi$ where $ \alpha_{s}$ is the simple root corresponding to $s$. One can compute via casework that there are only 4 possibilities for $ s$ and $ s's''$:

    $$  s= s_{1}\textrm{ and } s's'' = s_{2}s_{1}$$

    $$  s= s_{2} \textrm{ and } s's'' = s_{1}s_{2}$$

    $$  s= s_{2}\textrm{ and } s's'' = s_{3}s_{2}$$

    $$  s= s_{3} \textrm{ and } s's'' = s_{2}s_{3}$$
    But one can check that all 4 pairs above occur as terminal edges of some reorientation of the cube on the left in Example \ref{cubesexample}. Hence, $ X$ must be some reorientation of the cube on the left in Example \ref{cubesexample} by Corollary \ref{oneless}.
\end{proof}

Let $ (W,S)$ be a Coxeter system. Let $ J , K \seq S$ be \emph{orthogonal subsets of simple reflections}, meaning that for all $ s\in J$, $ s' \in K$, we have that $ B(\alpha_{s} , \alpha_{s'}) = 0$ where $ \alpha_{s}$, $ \alpha_{s'}$ are the simple roots associated to $ s$ and $ s'$ respectively. Note that if $ J$ and $K$ are orthogonal subsets of simple reflections, then $ J\cap K = \emptyset$ (because $B(\alpha , \alpha) =1 $ for all $ \alpha \in \Pi$). If $ w \in W_{J} \sm \{ 1 \}$ and $w' \in W_{K} \sm \{ 1 \} $, then orthogonality of $ J$ and $K$ will imply that $ ww' = w'w$ and $ w(\Phi_{w'}) = \Phi_{w'}$. Hence, $ w$ and $ w'$ will be the terminal edges of the following Coxeter square: 

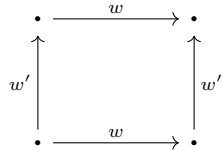
\[ \begin{tikzcd}
\sbullet \arrow{rr}{ w} & &\sbullet  \\
\\
\sbullet \arrow{uu}{w'} \arrow{rr}{w}& &\sbullet \arrow{uu}[swap]{w'}
\end{tikzcd}
\]
The above construction can be generalized further. Let $ J ,K \seq S$ be two subsets of orthogonal simple reflections. Let $X$ be a Coxeter $n$-cube in $(W_{J} ,J)$ with terminal edges $ x_{1} , x_{2} , \dots , x_{n}$, and let $Y$ be a Coxeter $ m$-cube in $ (W_{K} ,K)$ with terminal edges $ y_{1} , y_{2} , \dots y_{m}$. Then there exists a Coxeter $ (m+n)$-cube in $ (W_{J \cup K}, J \cup K)$ whose terminal edges are $ x_{1} , \dots, x_{n}, y_{1} , \dots y_{m}$. The Coxeter $(m+n)$-cube formed by this construction is called \emph{the product cube of $X$ and $Y$}. We denote the product cube of $ X$ and $ Y$ as $ \Prod(X,Y)$

\begin{prop} \label{productcubesnum}
    Let $ (W,S)$ be a Coxeter system. Let $ J , K \seq S$ be orthogonal subsets of simple reflections such that $ |W_{J}|< \infty$ and $ |W_{K}| < \infty$. Then 
    
    $$ \Cubes{|J| + |K|}{W_{J \cup K} ,J \cup K} = \Cubes{|J|}{W_{J} ,J} \Cubes{|K|}{W_{K} ,K} $$
    In other words, all of the Coxeter $(|J|+|K|)$-cubes that occur in $ (W_{J \cup K }, J \cup K)$ are product cubes.
    
\end{prop}

\begin{proof}
    Let $ Z$ be a Coxeter $ (|J| +|K|)$-cube in $ (W_{J \cup K} , J \cup K)$. Let $ z_{1} , z_{2} \dots , \dots , z_{|J|+|K|}$ be the terminal edges of $ Z$. It suffice to show that for each $ i$, $ z_{i} \in W_{J}$ or $z_{i} \in W_{K}$. Once we have shown this, Proposition \ref{highercubesdont} then implies that exactly $ |J|$ of the $z_{i}$'s lie in $ W_{J}$ and the other $ |K|$ of the $z_{i}$'s lie in $ W_{K}$. Hence, $ Z = \Prod(X,Y)$ where $ X$ is a $|J|$-cube in $ W_{J}$ and $ Y$ is a $|K|$-cube in $W_{K}$.

    Since $ J$ and $K$ are orthogonal subsets of simple reflections, it follows that $ W_{J \cup K} = W_{J} \times W_{K}$. Hence, $ |W_{J \cup K}| = |W_{J}||W_{K}| < \infty$. Therefore, we can apply Proposition \ref{wroandcube} to the Coxeter system $ (W_{J \cup K}, J \cup K)$ to deduce that 

    $$ \Phi^{+}_{J \cup K} = \coprod_{i=1}^{|J| +|K|}\Phi_{z_{i}} $$
    Suppose for the sake of contradiction that there existed a $j$ such that $ z_{j} \notin W_{J}$ and $ z_{j} \notin W_{K}$. Since $W_{J \cup K} = W_{J} \times W_{K} $, this would imply that $ z_{j} = w w'$ where $ w \in W_{J} \sm \{ 1 \} $ and $ w' \in W_{K} \sm \{ 1 \}$. Note that since $ W_{J \cup K} = W_{J} \times W_{K}$, we would have that 

    $$ \Phi_{z_{j}} = \Phi_{ww'} = \Phi_{w} \coprod \Phi_{w'}$$
    Hence, we get:

    $$ \Phi^{+}_{J \cup K} = \Phi_{z_{j}}  \coprod  (\coprod_{i \neq j} \Phi_{z_{i}})$$

    $$ \Phi^{+}_{J \cup K} = (\Phi_{w} \coprod \Phi_{w'}) \coprod  (\coprod_{i \neq j} \Phi_{z_{i}})$$
    Note that there are $|J| + |K| + 1$ non-empty left inversion sets on the right-hand side in the above equation. But we know that each left inversion set above has to contain at least one simple reflection from $\Pi_{J \cup K}$. Since $|J| + |K| + 1 > |J| +|K|$, we deduce that there has to exist at least two left inversion sets in the above equation that contain the same simple root. This contradictions the pairwise disjointness of the above left inversion sets.  
\end{proof}

\begin{prop} \label{highestbasedrectangle}
    Let $ X$ be a Coxeter $n$-cube in $ A_{n}$. Let $ x_{1} , x_{2} , \dots , x_{n}$ be the terminal edges of $X$. Let $ \beta \in \Phi^{+}$ denote the highest root in the root poset of $A_{n}$. Then the following are true:

    \begin{enumerate}
        \item There exists a unique $i$ such that $ \beta \in \Phi_{x_{i}}$.

        \item Let $i $ be as in statement (1). Let $ \alpha \in \Pi \cap \Phi_{x_{i}}$ denote the unique simple root in $ \Phi_{x_{i}}$. Then:

        $$ \Phi_{x_{i}} = \{ \gamma \in \Phi^{+} \mid \alpha \in \supp(\gamma) \}   $$
        where 
        
        $$ \supp(\gamma) : = \{ \rho \in \Pi \mid \exists c\in \mathbb{R} \textnormal{ with } c\neq 0 \textnormal{ and } \gamma \in c\rho +\Span(\Pi \sm \{ \rho \}) \}$$
    \end{enumerate}
    
\end{prop}

\begin{proof}
    The existence of such an $i$ in part (1) follows from Proposition \ref{wroandcube}. Similarly, the existence and uniqueness of a simple root $ \alpha \in \Pi \cap \Phi_{x_{i}}$ also follows from Proposition \ref{wroandcube}. Let $ \gamma \in \Phi^{+}$ such that $\alpha \in \supp(\gamma)$. Note that the root poset of $ A_{n}$ is a \textquotedblleft pyramid" shape. An explicit diagram of the root poset of $ A_{4}$ is given below: 
    
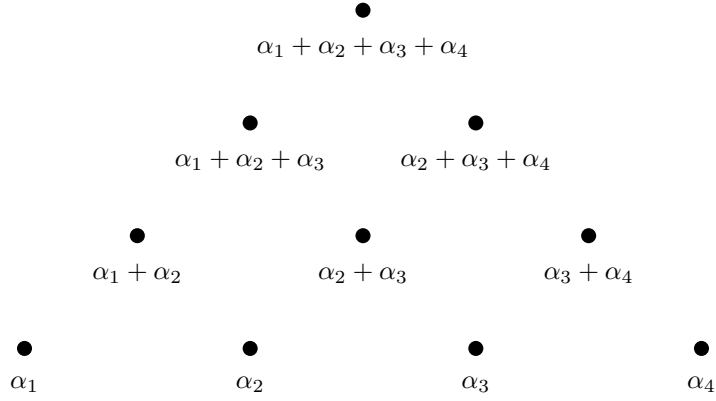
\begin{center}
    \begin{tikzpicture}
\foreach  \y in {0,...,3}
    \foreach \x in {0,...,\y} 
        \fill [black] ( 3*\x -1.5*\y , -1.5*\y) circle [radius=0.1];

\node[below = 7pt] at (-4.5,-4.5) {$\alpha_{1}$};
\node[below = 7pt] at (-1.5,-4.5) {$\alpha_{2}$};
\node[below = 7pt] at (1.5,-4.5) {$\alpha_{3}$};
\node[below = 7pt] at (4.5,-4.5) {$\alpha_{4}$};

\node[below = 7pt] at (-3,-3) {$\alpha_{1} + \alpha_{2}$};
\node[below = 7pt] at (0,-3) {$\alpha_{2} + \alpha_{3}$};
\node[below = 7pt] at (3,-3) {$\alpha_{3} + \alpha_{4}$};

\node[below = 7pt] at (-1.5,-1.5) {$\alpha_{1} + \alpha_{2}+ \alpha_{3}$};
\node[below = 7pt] at (1.5,-1.5) {$\alpha_{2} + \alpha_{3}+ \alpha_{4}$};
\node[below = 7pt] at (0,0) {$\alpha_{1} + \alpha_{2}+ \alpha_{3}+ \alpha_{4}$};

\end{tikzpicture}
\end{center}
    By analyzing the structure of the root poset of $ A_{n}$, we see that

    $$ \beta = \gamma + \sum_{\rho \in \Pi \sm \supp(\gamma) } \rho  $$
    Suppose for the sake of contradiction that $ \gamma \notin \Phi_{x_{i}}$. This would imply that $ x_{i}^{-1}(\gamma) \in \Phi^{+}$. Since $ \alpha$ is the unique simple root in $ \Phi_{x_{i}}$, and because $ \alpha \in \supp(\gamma)$, it follows that $ x_{i}^{-1}(\rho) \in \Phi^{+}$ for all $ \rho \in \Pi \sm \supp(\gamma)$. Hence:

    $$ x_{i}^{-1} (\beta) = x_{i}^{-1}\big(  \gamma + \sum_{\rho \in \Pi \sm \supp(\gamma) } \rho \big) = x_{i}^{-1}(\gamma) + \sum_{\rho \in \Pi \sm \supp(\gamma)}x_{i}^{-1}(\rho) $$
    where the right hand side is a sum of positive roots. Hence, $ x_{i}^{-1}(\beta) \in \Phi^{+}$, which contradicts $ \beta \in \Phi_{x_{i}}$.

    If  $ \gamma \in \Phi^{+}$ such that $ \alpha \notin \supp(\gamma)$, then we could express $\gamma$ as a linear combination of simple roots from the set $ \Pi \sm \{ \alpha \} $. Since $ \alpha$ is the unique simple root that lies in $ \Phi_{x_{i}}$, it would follow that $ x_{i}^{-1}(\gamma)$ would be a sum of positive roots, and thus $ x_{i}^{-1}(\gamma)$ would have to be a positive root. Hence, $ \gamma \notin \Phi_{x_{i}}$. This establishes that

    $$  \Phi_{x_{i}} = \{ \gamma \in \Phi^{+} \mid \alpha \in \supp(\gamma) \}$$

\end{proof}

\begin{prop} \label{smallercubeinside}
    Let $ X$ be a Coxeter $n$-cube coming from the $ A_{n}$ Coxeter system. Then $ X$ contains a Coxeter $ (n-1)$-cube that is fully contained in a proper standard parabolic subgroup of the $A_{n}$ Coxeter system. 
\end{prop}

\begin{proof}
    Let $ x_{1} , x_{2}, \dots , x_{n}$ be the terminal edges of $X$. By the Proposition \ref{highestbasedrectangle}, there exists an $i$ such that $ \Phi_{x_{i}} = \{  \gamma \in \Phi^{+} \mid \alpha \in \supp(\gamma) \} $ where $ \alpha \in \Pi$. Let $ s\in S$ denote the simple reflection associated to the simple root $\alpha$. Thus, for $j \neq i$, $\Phi_{x_{j}} \seq \Phi_{S \sm \{ s \}  }$, implying that $ x_{j} \in W_{S \sm \{ s \} }$ for $ j \neq i$. Thus, the edges $ x_{1}, \dots , x_{i-1}, x_{i+1} , \dots x_{n}$ lie in the standard parabolic subgroup $ W_{S \sm \{ s \} }$. Therefore, the $ (n-1)$-cube determined by $ x_{1}, \dots , x_{i-1}, x_{i+1} , \dots x_{n}$ lies entirely within the proper standard parabolic subgroup  $ W_{S \sm \{ s \} }$.
\end{proof}

\begin{prop} \label{bhgenrectangle}
    Consider the $A_{n}$ Coxeter system. Let $ s\in S$. Define $ J : = S \sm \{ s \}$. Then: 

    \begin{enumerate}
        \item $ |\Phi_{\nu(\alpha_{s} , \Pi_{J})} \cap \Pi|   = 1$ and $ \Phi_{\nu(\alpha_{s} , \Pi_{J})}$ contains the unique positive root of highest depth in the $A_{n}$ root poset.

        \item $ \Phi_{\nu(\alpha_{s} , \Pi_{J})} = \{ \gamma \in \Phi^{+} \mid \rho \in \supp(\gamma) \}$ where $ \{ \rho
 \} = \Phi_{\nu(\alpha_{s} , \Pi_{J})} \cap \Pi$ 

 \item If $x$ is an arbitrary element of the $A_{n}$ Coxeter system such that $ |\Phi_{x} \cap \Pi| = 1$ and $ \Phi_{x}$ contains the unique positive root of highest depth in the $A_{n}$ root poset, then there exists $ r\in S$, $ K := S \sm \{ r \}$ such that $ x = \nu(\alpha_{r} , \Pi_{K})$.
    \end{enumerate}
\end{prop}

\begin{proof} 
    The inverse of a Brink-Howlett generator is also a Brink-Howlett generator. Thus, there exists $ s' \in S$ and $ J' : = S \sm \{ s' \}$ such that $ \nu(\alpha_{s} , \Pi_{J})^{-1} = \nu(\alpha_{s'}, \Pi_{J'})$. But the subset of positive roots that are sent into $ \Phi^{-}$ by $ \nu(\alpha_{s'}, \Pi_{J'})$ is exactly the subset $ \Phi_{J' \cup s'}^{+} \sm \Phi^{+}_{J'} = \Phi^{+} \sm \Phi^{+}_{J'}$. Since $ J ' = S \sm \{ s' \}$, there is exactly one simple root in $\Phi^{+} \sm \Phi^{+}_{J'} $. Furthermore, the positive root of highest depth in the root poset of $A_{n}$ is not contained in any proper standard parabolic root subsystem. Thus, the positive root of highest depth lies in $ \Phi^{+} \sm \Phi_{J'}$. Therefore, $ \Phi_{\nu(\alpha_{s} , \Pi_{J})}$ has the desired properties as in statement (1).

    Statement (2) immediately follows from statement (1) and Proposition \ref{highestbasedrectangle}.

    For statement (3), let $ \alpha $ be the unique element of $ \Phi_{x} \cap \Pi$. Let $ \beta \in \Phi^{+}$ be the unique positive root of highest depth in the root poset of $ A_{n}$. Since we are assuming $ \beta \in \Phi_{x}$, we have that $ x^{-1}(\beta) \in \Phi^{-}$. Since $ \alpha$ is the unique element of $ \Phi_{x} \cap \Pi$, and since $ x^{-1}(\beta) \in \Phi^{-}$, we can reuse the proof of Proposition \ref{highestbasedrectangle} to conclude that

    $$ \Phi_{x} : = \{  \gamma \in \Phi^{+} \mid \alpha \in \supp(\gamma)  \}$$
    In other words, $  \{  \gamma \in \Phi^{+} \mid \alpha \in \supp(\gamma)  \}$ is exactly the set of positive roots that are sent into $\Phi^{-}$ by the element $ x^{-1}$. However, it is not too hard to show that $\{  \gamma \in \Phi^{+} \mid \alpha \in \supp(\gamma)  \}$ is also the exact set of positive roots that are sent into $\Phi^{-}$ by the element $ \nu(\alpha , \Pi \sm \{ \alpha \} )$. Since an element of a Coxeter system is uniquely determined by the subset of positive roots that it sends into $\Phi^{-}$, we deduce that $ x^{-1} = \nu(\alpha , \Pi \sm \{ \alpha \} )$, which gives $ x = \nu(\alpha , \Pi \sm \{ \alpha \} )^{-1}$. But the inverse of a Brink-Howlett generator is also a Brink-Howlett generator. Therefore, $x = \nu(\alpha_{r} , \Pi_{K})$ for some $ r \in S$, $ K : = S \sm \{ r \} $.

\end{proof}

Consider the root poset of $A_{n}$. For ease of visualization, take $ n =4$. Then the root poset of $ A_{4}$ looks as follows:

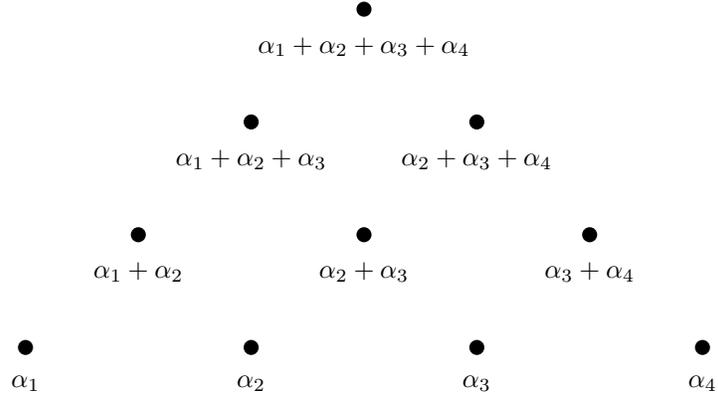
\begin{center}

\begin{tikzpicture}
\foreach  \y in {0,...,3}
    \foreach \x in {0,...,\y} 
        \fill [black] ( 3*\x -1.5*\y , -1.5*\y) circle [radius=0.1];

\node[below = 7pt] at (-4.5,-4.5) {$\alpha_{1}$};
\node[below = 7pt] at (-1.5,-4.5) {$\alpha_{2}$};
\node[below = 7pt] at (1.5,-4.5) {$\alpha_{3}$};
\node[below = 7pt] at (4.5,-4.5) {$\alpha_{4}$};

\node[below = 7pt] at (-3,-3) {$\alpha_{1} + \alpha_{2}$};
\node[below = 7pt] at (0,-3) {$\alpha_{2} + \alpha_{3}$};
\node[below = 7pt] at (3,-3) {$\alpha_{3} + \alpha_{4}$};

\node[below = 7pt] at (-1.5,-1.5) {$\alpha_{1} + \alpha_{2}+ \alpha_{3}$};
\node[below = 7pt] at (1.5,-1.5) {$\alpha_{2} + \alpha_{3}+ \alpha_{4}$};
\node[below = 7pt] at (0,0) {$\alpha_{1} + \alpha_{2}+ \alpha_{3}+ \alpha_{4}$};

\end{tikzpicture}

\end{center}
We define a \emph{highest based rectangle} to be a subset $R_{h}$ of the root poset of $A_{n}$ that is of the following form:

$$ R_{h} : = \{ \gamma \in \Phi^{+} \mid \alpha \in \supp(\gamma)  \}$$
where $ \alpha \in \Pi$. An example of a highest based rectangle is given below in the $A_{4}$ root poset:

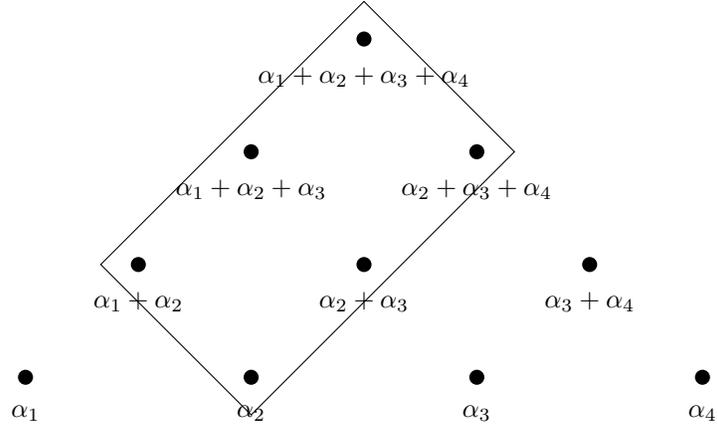
\begin{center}

\begin{tikzpicture}
\foreach  \y in {0,...,3}
    \foreach \x in {0,...,\y} 
        \fill [black] ( 3*\x -1.5*\y , -1.5*\y) circle [radius=0.1];

\node[below = 7pt] at (-4.5,-4.5) {$\alpha_{1}$};
\node[below = 7pt] at (-1.5,-4.5) {$\alpha_{2}$};
\node[below = 7pt] at (1.5,-4.5) {$\alpha_{3}$};
\node[below = 7pt] at (4.5,-4.5) {$\alpha_{4}$};

\node[below = 7pt] at (-3,-3) {$\alpha_{1} + \alpha_{2}$};
\node[below = 7pt] at (0,-3) {$\alpha_{2} + \alpha_{3}$};
\node[below = 7pt] at (3,-3) {$\alpha_{3} + \alpha_{4}$};

\node[below = 7pt] at (-1.5,-1.5) {$\alpha_{1} + \alpha_{2}+ \alpha_{3}$};
\node[below = 7pt] at (1.5,-1.5) {$\alpha_{2} + \alpha_{3}+ \alpha_{4}$};
\node[below = 7pt] at (0,0) {$\alpha_{1} + \alpha_{2}+ \alpha_{3}+ \alpha_{4}$};

\draw (-1.5 , -5) -- ( 2, -1.5) --(0,.5 )--(-3.5,-3)--cycle;

\end{tikzpicture}

\end{center}
Proposition \ref{bhgenrectangle} has a nice combinatorial/geometric interpretation in the root poset of $A_{n}$. The proposition states that $ x = \nu(\alpha_{r} , \Pi_{K})$ for some $ r\in S$, $ K : = S \sm \{ r \} $ if and only if $ \Phi_{x}$ has the shape of a \emph{highest based rectangle}. If $ \Phi_{x}$ denotes the highest based rectangle in the diagram above, one can verify that $ x^{-1} = \nu(\alpha_{2} , \Pi_{J})$ where $ J : = S \sm \{ s_{2} \} $. Hence, $ x = \nu(\alpha_{3} , \Pi_{K})$ where $ K:= S \sm \{ s_{3} \} $. More generally, a \emph{based rectangle} $R$ is a subset of the root poset of $A_{n}$ of the following form:

$$ R = \{ \gamma \in \Phi^{+} \mid \alpha \in \supp(\gamma) \seq \supp(\rho) \}$$
where $ \alpha \in \Pi$ and $ \rho \in \Phi^{+}$ such that $ \alpha \in \supp(\rho)$. The below is an example of based rectangle in $ A_{4}$ where $ \alpha = \alpha_{2}$ and $ \rho = \alpha_{1} + \alpha_{2} + \alpha_{3}$:

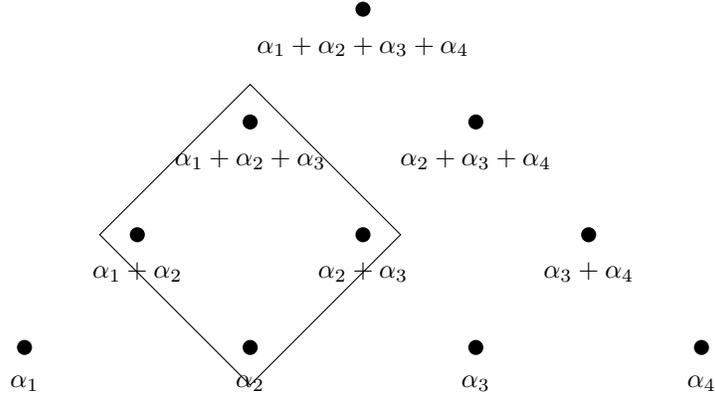
\begin{center}

\begin{tikzpicture}
\foreach  \y in {0,...,3}
    \foreach \x in {0,...,\y} 
        \fill [black] ( 3*\x -1.5*\y , -1.5*\y) circle [radius=0.1];

\node[below = 7pt] at (-4.5,-4.5) {$\alpha_{1}$};
\node[below = 7pt] at (-1.5,-4.5) {$\alpha_{2}$};
\node[below = 7pt] at (1.5,-4.5) {$\alpha_{3}$};
\node[below = 7pt] at (4.5,-4.5) {$\alpha_{4}$};

\node[below = 7pt] at (-3,-3) {$\alpha_{1} + \alpha_{2}$};
\node[below = 7pt] at (0,-3) {$\alpha_{2} + \alpha_{3}$};
\node[below = 7pt] at (3,-3) {$\alpha_{3} + \alpha_{4}$};

\node[below = 7pt] at (-1.5,-1.5) {$\alpha_{1} + \alpha_{2}+ \alpha_{3}$};
\node[below = 7pt] at (1.5,-1.5) {$\alpha_{2} + \alpha_{3}+ \alpha_{4}$};
\node[below = 7pt] at (0,0) {$\alpha_{1} + \alpha_{2}+ \alpha_{3}+ \alpha_{4}$};

\draw (-1.5 , -5) -- ( .5, -3) --(-1.5,-1 )--(-3.5,-3)--cycle;

\end{tikzpicture}

\end{center}
Note that a based rectangle can be viewed as a highest based rectangle of some irreducible standard parabolic subgroup. In the example above, the based rectangle can be viewed as a highest based rectangle of the standard parabolic subsystem $(W_{J} ,J)$ where $J = \{ s_{1} , s_{2} , s_{3} \} $.

\begin{thm} \label{basedrectanglepartition}
    Let $ x_{1} , x_{2} , \dots , x_{n}$ be the terminal edges of a Coxeter $n$-cube in the $A_{n}$ Coxeter system. Then the $ \Phi_{x_{i}}$'s partition the root poset of $A_{n}$ into based rectangles.
\end{thm}

\begin{proof} 
    By Proposition \ref{highestbasedrectangle}, we know that there exists an $ i$ such that $ \Phi_{x_{i}}$ is a highest based rectangle of the entire root poset of $A_{n}$. Let $ s\in S$ such that $ \alpha_{s}$ is the unique element of $ \Phi_{x_{i}}$. But by Proposition \ref{smallercubeinside}, we conclude that $ x_{1} , \dots , x_{i-1} , x_{i+1} , \dots ,x_{n}$ must lie within the standard parabolic subsystem $( W_{S \sm \{ s \} } , S \sm \{ s \}) $. Because we have removed the vertex $ s$ from the Coxeter diagram of $A_{n}$, it follows that the Coxeter diagram of $ ( W_{S \sm \{ s \} } , S \sm \{ s \})$ is isomorphic to either $ A_{n-1}$ or $ A_{i-1} \times A_{n-i}$ for some $ 2\leq i \leq n-1$. If $ ( W_{S \sm \{ s \} } , S \sm \{ s \})$ is isomorphic to $A_{i-1} \times A_{n-i} $, note that Proposition \ref{productcubesnum} will imply that for $ j \neq i$, $x_{j}$ lies entirely within either the $ A_{i-1}$ component or the $ A_{n-i}$ component. Without loss of generality, assume that $ x_{1} , \dots , x_{i-1}$ lie entirely within $ A_{i}$ and that $ x_{i+1} , \dots ,x_{n}$ lie entirely within $ A_{n-i}$. Note also that the elements of $ \Phi^{+} \sm \Phi_{x_{i}}$ will be positive roots of the subsystem $ ( W_{S \sm \{ s \} } , S \sm \{ s \})$. Thus, because $ x_{1} , \dots , x_{i-1} , x_{i+1} , \dots x_{n}$ are the terminal edges of a Coxeter $(n-1)$-cube in $( W_{S \sm \{ s \} } , S \sm \{ s \}) $, and since $ ( W_{S \sm \{ s \} } , S \sm \{ s \})$ is isomorphic to either $ A_{n-1}$ or $ A_{i-1} \times A_{n-i}$ for some $ 2\leq i \leq n-1$, we can now apply Proposition \ref{highestbasedrectangle} to the root poset of $ (W_{S \sm \{  s\} } , S \sm \{  s \} )$ (more precisely, in the case of $(W_{S \sm \{  s\} } , S \sm \{  s \} ) \cong  A_{i-1} \times A_{n-i}$, we are applying Proposition \ref{highestbasedrectangle}to the component $A_{i-1}$ and to the component $ A_{n-i}$). Thus, via induction on $|S|$, the $\Phi_{x_{i}}$'s partition the root poset of $ A_{n}$ into based rectangles.
\end{proof}

A partition of the root poset of $A_{n}$ by the $ \Phi_{x_{i}}$'s as stated in Theorem \ref{basedrectanglepartition} is called a \emph{based rectangle partition of $A_{n}$}. We illustrate an example of the recursive procedure as described in the proof of Theorem \ref{basedrectanglepartition}:

\begin{ex}
    Let $ x_{1}, x_{2} , x_{3} , x_{4}$ be the terminal edges of a Coxeter $4$-cube in $A_{4}$. We know that one of the $\Phi_{x_{i}}$'s give rise to a highest based rectangle in $A_{4}$. Without loss of generality, let us assume that it is $\Phi_{x_{3}}$. Hence, $ \Phi_{x_{3}}$ might look like the highest based rectangle shown below:

    
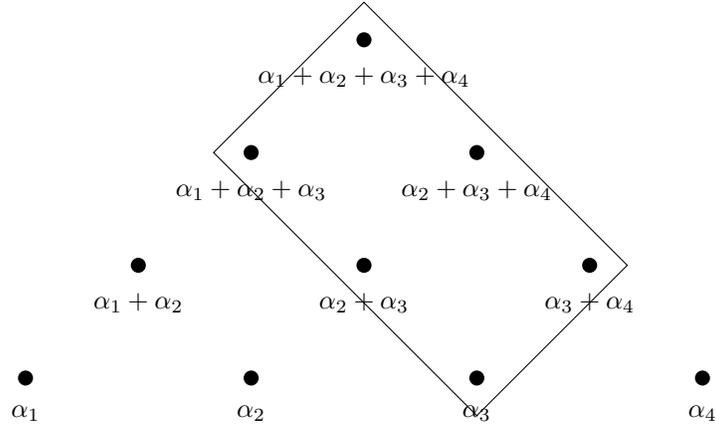
\begin{center}

\begin{tikzpicture}
\foreach  \y in {0,...,3}
    \foreach \x in {0,...,\y} 
        \fill [black] ( 3*\x -1.5*\y , -1.5*\y) circle [radius=0.1];

\node[below = 7pt] at (-4.5,-4.5) {$\alpha_{1}$};
\node[below = 7pt] at (-1.5,-4.5) {$\alpha_{2}$};
\node[below = 7pt] at (1.5,-4.5) {$\alpha_{3}$};
\node[below = 7pt] at (4.5,-4.5) {$\alpha_{4}$};

\node[below = 7pt] at (-3,-3) {$\alpha_{1} + \alpha_{2}$};
\node[below = 7pt] at (0,-3) {$\alpha_{2} + \alpha_{3}$};
\node[below = 7pt] at (3,-3) {$\alpha_{3} + \alpha_{4}$};

\node[below = 7pt] at (-1.5,-1.5) {$\alpha_{1} + \alpha_{2}+ \alpha_{3}$};
\node[below = 7pt] at (1.5,-1.5) {$\alpha_{2} + \alpha_{3}+ \alpha_{4}$};
\node[below = 7pt] at (0,0) {$\alpha_{1} + \alpha_{2}+ \alpha_{3}+ \alpha_{4}$};

\draw (1.5 , -5) -- ( 3.5, -3) --(0,.5 )--(-2,-1.5)--cycle;

\end{tikzpicture}

\end{center}
Note that $ \Phi^{+} \sm \Phi_{x_{3}}$ will be the root poset for $ (W_{S \sm \{ s_{3} \} }, S \sm \{ s_{3} \} )$, where $ (W_{S \sm \{ s_{3} \} }, S \sm \{ s_{3} \} )$ is clearly isomorphic to $ A_{2} \times A_{1}$. As aforementioned in the proof of Theorem \ref{basedrectanglepartition}, we know that for $ j\neq 3$, $x_{j}$ will lie entirely within the $A_{2}$ component or entirely within the $ A_{1}$ component. Hence, without loss of generality, assume that $x_{1} , x_{2}$ lie entirely within $ A_{2}$ and that $ x_{4}$ lies entirely within $ A_{1}$ (note that it is impossible for two of the $ x_{j}$'s to lie entirely within the $A_{1}$ component by Proposition \ref{highercubesdont}). But now we can apply Proposition \ref{highestbasedrectangle} to the components $A_{2}$ and $A_{1}$. One possibility is for $ \Phi_{x_{2}} = \{ \alpha_{2} , \alpha_{1}+\alpha_{2} \} $ and for $ \Phi_{x_{3}} = \{ \alpha_{4} \}$. Thus, the diagram would become:

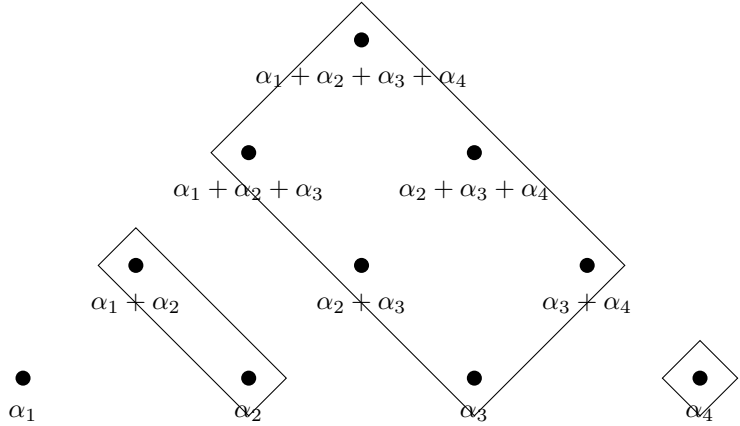
\begin{center}

\begin{tikzpicture}
\foreach  \y in {0,...,3}
    \foreach \x in {0,...,\y} 
        \fill [black] ( 3*\x -1.5*\y , -1.5*\y) circle [radius=0.1];

\node[below = 7pt] at (-4.5,-4.5) {$\alpha_{1}$};
\node[below = 7pt] at (-1.5,-4.5) {$\alpha_{2}$};
\node[below = 7pt] at (1.5,-4.5) {$\alpha_{3}$};
\node[below = 7pt] at (4.5,-4.5) {$\alpha_{4}$};

\node[below = 7pt] at (-3,-3) {$\alpha_{1} + \alpha_{2}$};
\node[below = 7pt] at (0,-3) {$\alpha_{2} + \alpha_{3}$};
\node[below = 7pt] at (3,-3) {$\alpha_{3} + \alpha_{4}$};

\node[below = 7pt] at (-1.5,-1.5) {$\alpha_{1} + \alpha_{2}+ \alpha_{3}$};
\node[below = 7pt] at (1.5,-1.5) {$\alpha_{2} + \alpha_{3}+ \alpha_{4}$};
\node[below = 7pt] at (0,0) {$\alpha_{1} + \alpha_{2}+ \alpha_{3}+ \alpha_{4}$};

\draw (1.5 , -5) -- ( 3.5, -3) --(0,.5 )--(-2,-1.5)--cycle;

\draw (-3, -2.5) -- (-3.5 , -3) -- (-1.5, -5) -- (-1, -4.5)--cycle;

\draw (4.5, -4) --(4, -4.5) -- (4.5 ,-5 )-- (5,-4.5)--cycle;

\end{tikzpicture}

\end{center}
Lastly, we apply Proposition \ref{highestbasedrectangle} to $ (W_{\{ s_{1} \} } , \{ s_{1} \} )$ and to $ \Phi_{x_{1}} = \{ \alpha_{1} \} $, which gives the following partition:

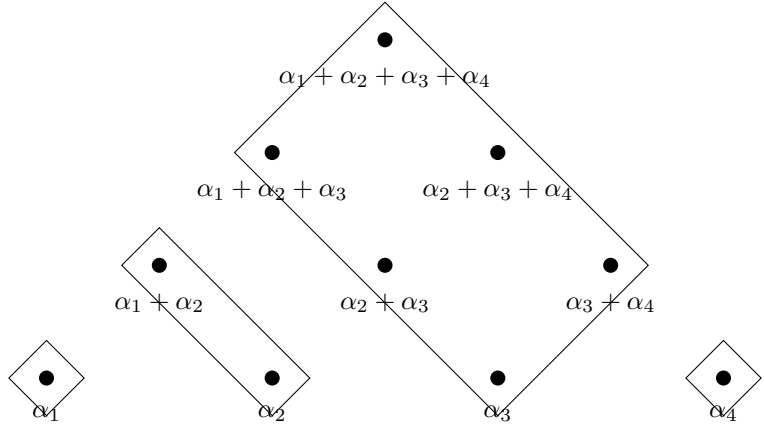
\begin{center}

\begin{tikzpicture}
\foreach  \y in {0,...,3}
    \foreach \x in {0,...,\y} 
        \fill [black] ( 3*\x -1.5*\y , -1.5*\y) circle [radius=0.1];

\node[below = 7pt] at (-4.5,-4.5) {$\alpha_{1}$};
\node[below = 7pt] at (-1.5,-4.5) {$\alpha_{2}$};
\node[below = 7pt] at (1.5,-4.5) {$\alpha_{3}$};
\node[below = 7pt] at (4.5,-4.5) {$\alpha_{4}$};

\node[below = 7pt] at (-3,-3) {$\alpha_{1} + \alpha_{2}$};
\node[below = 7pt] at (0,-3) {$\alpha_{2} + \alpha_{3}$};
\node[below = 7pt] at (3,-3) {$\alpha_{3} + \alpha_{4}$};

\node[below = 7pt] at (-1.5,-1.5) {$\alpha_{1} + \alpha_{2}+ \alpha_{3}$};
\node[below = 7pt] at (1.5,-1.5) {$\alpha_{2} + \alpha_{3}+ \alpha_{4}$};
\node[below = 7pt] at (0,0) {$\alpha_{1} + \alpha_{2}+ \alpha_{3}+ \alpha_{4}$};

\draw (1.5 , -5) -- ( 3.5, -3) --(0,.5 )--(-2,-1.5)--cycle;

\draw (-3, -2.5) -- (-3.5 , -3) -- (-1.5, -5) -- (-1, -4.5)--cycle;

\draw (4.5, -4) --(4, -4.5) -- (4.5 ,-5 )-- (5,-4.5)--cycle;

\draw (-4.5 , -4)--(-4 , -4.5) --(-4.5,-5)--(-5,-4.5)--cycle;

\end{tikzpicture}

\end{center}

\end{ex}
The next theorem establishes a converse to Theorem \ref{basedrectanglepartition}:

\begin{thm} \label{partitiontoncube}
    Given a based rectangle partition $P$ of the root poset of $ A_{n}$, there exists a Coxeter $n$-cube with terminal edges $ x_{1} , x_{2} , \dots, x_{n}$ such that $ P = \{ \Phi_{x_{i}} \mid 1\leq i \leq n  \} $. Consequentially, any based rectangle of the root poset of $ A_{n}$ corresponds to the left inversion set of some terminal edge of a Coxeter $n$-cube in $A_{n}$.
\end{thm}

\begin{proof}
    Since $ P$ is a based rectangle partition, $P$ must contain a highest based rectangle (a based rectangle that contains the positive root of maximum depth). But by Proposition \ref{bhgenrectangle} part (3), we know that this highest based rectangle is a left inversion set of some $ x = \nu(\alpha_{r} , \Pi_{K})$ where $ r\in S$ and $ K:= S \sm \{ r \} $. As we saw in the proof of Theorem \ref{basedrectanglepartition}, $ \Phi^{+} \sm \Phi_{x}$ is the set of positive roots for the standard parabolic subsystem $ (W_{K} , K)$. Furthermore, $ (W_{K} ,K)$ is isomorphic to either $ A_{n-1}$ or $ A_{i-1} \times A_{n-i}$ for some $ 2\leq i \leq n-1$. Thus, we can recursively apply our analysis to the Coxeter system $ (W_{K} , K)$ and the root poset of $ \Phi^{+} \sm \Phi_{x}$. Therefore, we conclude that each based rectangle of $ P$ corresponds to the left inversion set of some Brink-Howlett generator. But since each of these left inversion sets correspond to a Brink-Howlett generator, we conclude that $ P$ arises from the terminal edges of a Coxeter $n$-cube that was constructed using the recursive method as described in the proof of Proposition \ref{ncubeexists}.

    Suppose we are a given a based rectangle $ R$ of the root poset of $A_{n}$. If $ R$ is a highest based rectangle, then we know that $ \Phi^{+} \sm R$ will be the root poset of some proper standard parabolic subgroup that is isomorphic to either $ A_{n-1}$ or $ A_{i-1} \times A_{n-i}$ for some $ 2\leq i \leq n-1$. Hence, we will be able to partition the rest of $ \Phi^{+} \sm R$ with based rectangles. Thus, $ R$ will be a based rectangle of some based rectangle partition of the root poset of $A_{n}$. By what we proved in the previous paragraph, this implies that $R$ is the left inversion set of some terminal edge of a Coxeter $n $-cube in $ A_{n}$. 
    
    If $ R$ is not a highest based rectangle (meaning that the positive root of highest depth is not in $R$), then we can construct a highest based rectangle $R_{h}$ in the root poset of $ A_{n}$ that does not intersect with our original based rectangle $R$. Thus, $ R \seq \Phi^{+} \sm R_{h}$, and since $ \Phi{+} \sm R_{h}$ is the root poset of either $ A_{n-1}$ or $ A_{i-1} \times A_{n-1}$, we can apply recursion to construct a based rectangle partition of the root poset of $A_{n}$ with our original based rectangle $ R$. Hence, we again conclude that $ R$ must be the left inversion set of some terminal edge of a Coxeter $n$-cube in $ A_{n}$.
\end{proof}

Even though we have shown that any based rectangle partition $ P$ of the root poset of $ A_{n}$ corresponds to some Coxeter $ n$-cube in $ A_{n}$, it is possible that two different based rectangle partitions correspond to the same Coxeter $ n$-cube modulo orientation. We give a breif example:

\begin{ex}
    Consider the two following based rectangle partitions of $ A_{2}$:
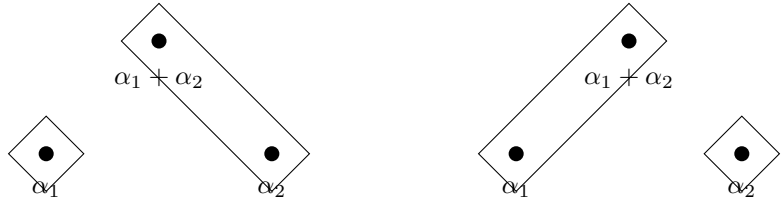
\vspace{.5cm}
\begin{center}
    \begin{tikzpicture}
\foreach  \y in {0,...,1}
    \foreach \x in {0,...,\y} 
        \fill [black] ( 3*\x -1.5*\y , -1.5*\y) circle [radius=0.1];

        \node[below = 7pt] at (-1.5,-1.5) {$\alpha_{1}$};
        \node[below = 7pt] at (1.5,-1.5) {$\alpha_{2}$};
        \node[below = 7pt] at (0,0) {$\alpha_{1}+\alpha_{2}$};

        \draw (-1.5 , -1) --(-2, -1.5)--(-1.5,-2)--(-1,-1.5)--cycle;
        \draw (0 , .5) --(-.5, 0)--(1.5,-2)--(2,-1.5)--cycle;

\end{tikzpicture}
\hspace{2cm}
\begin{tikzpicture}
\foreach  \y in {0,...,1}
    \foreach \x in {0,...,\y} 
        \fill [black] ( 3*\x -1.5*\y , -1.5*\y) circle [radius=0.1];

        \node[below = 7pt] at (-1.5,-1.5) {$\alpha_{1}$};
        \node[below = 7pt] at (1.5,-1.5) {$\alpha_{2}$};
        \node[below = 7pt] at (0,0) {$\alpha_{1}+\alpha_{2}$};

        \draw (1.5 , -1) --(2, -1.5)--(1.5,-2)--(1,-1.5)--cycle;
        \draw (0 , .5) --(.5, 0)--(-1.5,-2)--(-2,-1.5)--cycle;

\end{tikzpicture}
\end{center}

Note $ \Phi_{s_{1}} = \{ \alpha_{1} \} $, $ \Phi_{s_{2}} = \{ \alpha_{2} \} $, $ \Phi_{s_{1}s_{2}} = \{ \alpha_{1} , \alpha_{1} + \alpha_{2} \} $, and $ \Phi_{s_{2}s_{1}} = \{ \alpha_{2} , \alpha_{1} + \alpha_{2} \} $. Hence, the two based rectangle partitions above correspond to the following two Coxeter squares respectively:

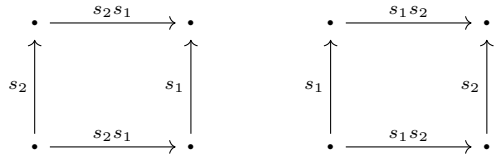
\begin{center}
\vspace{.5cm}
\begin{tikzcd}
\sbullet \arrow[rr,  "s_{2}s_{1}"] & & \sbullet \\
\\
\sbullet \arrow[uu, "s_{2}"] \arrow[rr,  "s_{2}s_{1}"] & & \sbullet \arrow[uu, "s_{1}"]
\end{tikzcd}
\hspace{1cm}
\begin{tikzcd}
\sbullet \arrow[rr, "s_{1}s_{2}"] & & \sbullet \\
\\
\sbullet \arrow[uu, "s_{1}"] \arrow[rr, "s_{1}s_{2}"] & & \sbullet \arrow[uu, "s_{2}"]
\end{tikzcd}
\end{center}
But note that the square on the right can be reoriented by flipping the edges labeled $ s_{1}s_{2}$ to match the square on the left. Hence, the two Coxeter squares above are the same modulo orientation.

\end{ex}

Consider a based rectangle partition $P$ of the root poset of $ A_{n}$. We define a \emph{compatible subtriangle $Q$ of $P$} such that:

\begin{enumerate}
    \item $ Q \seq \Phi^{+}$ such that $ Q$ is the set of positive roots associated to some standard parabolic subsystem $ (W_{K}, K)$ where $ (W_{K} , K)$ is isomorphic to the $ A_{k}$ Coxeter system for some $ k =1,2, \dots, n$.

    \item $ Q$ is a union of based rectangles from the based rectangle partition $P$.
\end{enumerate}

\begin{ex}

    Consider the based rectangle partition $P$ of the root poset of $A_{n}$:

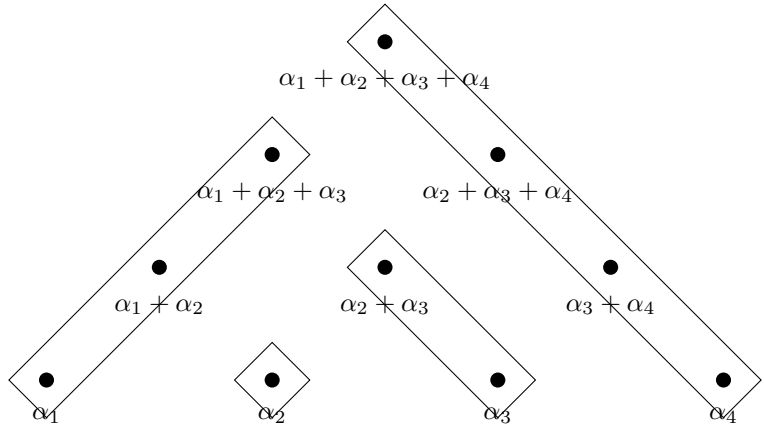
\vspace{.5cm}
    \begin{center}

\begin{tikzpicture}
\foreach  \y in {0,...,3}
    \foreach \x in {0,...,\y} 
        \fill [black] ( 3*\x -1.5*\y , -1.5*\y) circle [radius=0.1];

\node[below = 7pt] at (-4.5,-4.5) {$\alpha_{1}$};
\node[below = 7pt] at (-1.5,-4.5) {$\alpha_{2}$};
\node[below = 7pt] at (1.5,-4.5) {$\alpha_{3}$};
\node[below = 7pt] at (4.5,-4.5) {$\alpha_{4}$};

\node[below = 7pt] at (-3,-3) {$\alpha_{1} + \alpha_{2}$};
\node[below = 7pt] at (0,-3) {$\alpha_{2} + \alpha_{3}$};
\node[below = 7pt] at (3,-3) {$\alpha_{3} + \alpha_{4}$};

\node[below = 7pt] at (-1.5,-1.5) {$\alpha_{1} + \alpha_{2}+ \alpha_{3}$};
\node[below = 7pt] at (1.5,-1.5) {$\alpha_{2} + \alpha_{3}+ \alpha_{4}$};
\node[below = 7pt] at (0,0) {$\alpha_{1} + \alpha_{2}+ \alpha_{3}+ \alpha_{4}$};

\draw (0,.5)--(-.5,0)--(4.5,-5)--(5,-4.5)--cycle;

\draw (-1.5,-1)--(-1, -1.5)--(-4.5,-5)--(-5,-4.5)--cycle;

\draw (0 , -2.5)--(-.5,-3)--(1.5,-5)--(2,-4.5)--cycle;

\draw (-1.5 , -4)--(-2, -4.5)--(-1.5,-5)--(-1, -4.5)--cycle;

\end{tikzpicture}

\end{center}
One such example of a compatible subtriangle is the subset $ \Phi^{+}_{ \{ s_{1} , s_{2} , s_{3} \} }$, which can be viewed as the following blue subtriangle:

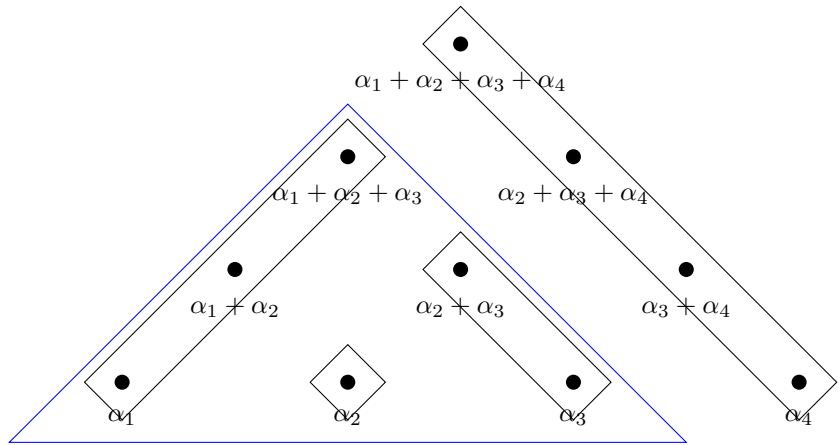
\vspace{.5cm}
    \begin{center}

\begin{tikzpicture}
\foreach  \y in {0,...,3}
    \foreach \x in {0,...,\y} 
        \fill [black] ( 3*\x -1.5*\y , -1.5*\y) circle [radius=0.1];

\node[below = 7pt] at (-4.5,-4.5) {$\alpha_{1}$};
\node[below = 7pt] at (-1.5,-4.5) {$\alpha_{2}$};
\node[below = 7pt] at (1.5,-4.5) {$\alpha_{3}$};
\node[below = 7pt] at (4.5,-4.5) {$\alpha_{4}$};

\node[below = 7pt] at (-3,-3) {$\alpha_{1} + \alpha_{2}$};
\node[below = 7pt] at (0,-3) {$\alpha_{2} + \alpha_{3}$};
\node[below = 7pt] at (3,-3) {$\alpha_{3} + \alpha_{4}$};

\node[below = 7pt] at (-1.5,-1.5) {$\alpha_{1} + \alpha_{2}+ \alpha_{3}$};
\node[below = 7pt] at (1.5,-1.5) {$\alpha_{2} + \alpha_{3}+ \alpha_{4}$};
\node[below = 7pt] at (0,0) {$\alpha_{1} + \alpha_{2}+ \alpha_{3}+ \alpha_{4}$};

\draw (0,.5)--(-.5,0)--(4.5,-5)--(5,-4.5)--cycle;

\draw (-1.5,-1)--(-1, -1.5)--(-4.5,-5)--(-5,-4.5)--cycle;

\draw (0 , -2.5)--(-.5,-3)--(1.5,-5)--(2,-4.5)--cycle;

\draw (-1.5 , -4)--(-2, -4.5)--(-1.5,-5)--(-1, -4.5)--cycle;

\draw[blue] (-1.5 , -.8)--(-6 , -5.3)--(3 , -5.3)--cycle;

\end{tikzpicture}

\end{center}
The subset $ \Phi_{ \{  s_{1} , s_{2} , s_{3} \} }^{+}$ is the set of positive roots associated to the standard parabolic subsystem $ (W_{ \{ s_{1} , s_{2} , s_{3} \} } , \{ s_{1} , s_{2} , s_{3}  \} )$, which is clearly isomorphic to the $ A_{3}$ Coxeter system. Furthermore, $ \Phi_{ \{ s_{1} , s_{2}, s_{3} \} }^{+}$ is a union of based rectangles from the partition $P$. The other compatible subtriangles inside the partition $P$ are the following: $ \Phi^{+}_{ \{ s_{2} \} }$, $\Phi^{+}_{ \{ s_{2} , s_{3}  \} } $, and the entire set of positive roots $\Phi^{+} $. In general, for any based rectangle partition $P$ of the root poset of $ A_{n}$, there will be exactly $ n$ compatible subtriangles of $P$.

\end{ex}
Let $ P$ be a based rectangle partition of the root poset of $ A_{n}$. Let $ Q$ be a compatible subtriangle of $ P$. We define \emph{a flip by the compatible subtraingle $Q$} or simply \emph{a flip by $Q$} to be the operation that transforms $P$ to the based rectangle partition $P'$, where $ P'$ is formed from $P$ by flipping $ Q$ by its vertical axis of symmetry. We present an example of such a compatible subtriangle flip:

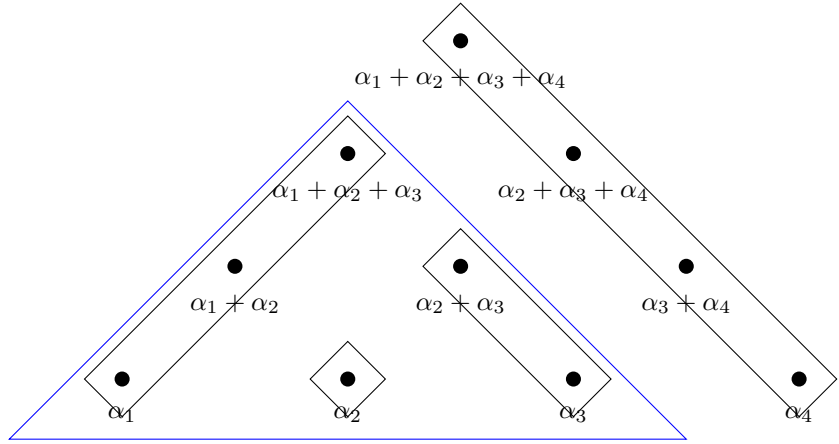
\vspace{.5cm}
    \begin{center}

\begin{tikzpicture}
\foreach  \y in {0,...,3}
    \foreach \x in {0,...,\y} 
        \fill [black] ( 3*\x -1.5*\y , -1.5*\y) circle [radius=0.1];

\node[below = 7pt] at (-4.5,-4.5) {$\alpha_{1}$};
\node[below = 7pt] at (-1.5,-4.5) {$\alpha_{2}$};
\node[below = 7pt] at (1.5,-4.5) {$\alpha_{3}$};
\node[below = 7pt] at (4.5,-4.5) {$\alpha_{4}$};

\node[below = 7pt] at (-3,-3) {$\alpha_{1} + \alpha_{2}$};
\node[below = 7pt] at (0,-3) {$\alpha_{2} + \alpha_{3}$};
\node[below = 7pt] at (3,-3) {$\alpha_{3} + \alpha_{4}$};

\node[below = 7pt] at (-1.5,-1.5) {$\alpha_{1} + \alpha_{2}+ \alpha_{3}$};
\node[below = 7pt] at (1.5,-1.5) {$\alpha_{2} + \alpha_{3}+ \alpha_{4}$};
\node[below = 7pt] at (0,0) {$\alpha_{1} + \alpha_{2}+ \alpha_{3}+ \alpha_{4}$};

\draw (0,.5)--(-.5,0)--(4.5,-5)--(5,-4.5)--cycle;

\draw (-1.5,-1)--(-1, -1.5)--(-4.5,-5)--(-5,-4.5)--cycle;

\draw (0 , -2.5)--(-.5,-3)--(1.5,-5)--(2,-4.5)--cycle;

\draw (-1.5 , -4)--(-2, -4.5)--(-1.5,-5)--(-1, -4.5)--cycle;

\draw[blue] (-1.5 , -.8)--(-6 , -5.3)--(3 , -5.3)--cycle;

\end{tikzpicture}

\end{center}
In the above based rectanlge partition $P$ of $ A_{4}$, we consider the compatible subtriangle $ \Phi^{+}_{ \{ s_{1} , s_{2} , s_{3} \}  }$ (denoted in blue). A flip by $ \Phi^{+}_{ \{ s_{1} , s_{2} , s_{3} \}  }$'s vertical axis of symmetry results in the new based rectangle partition $ P'$, given below:

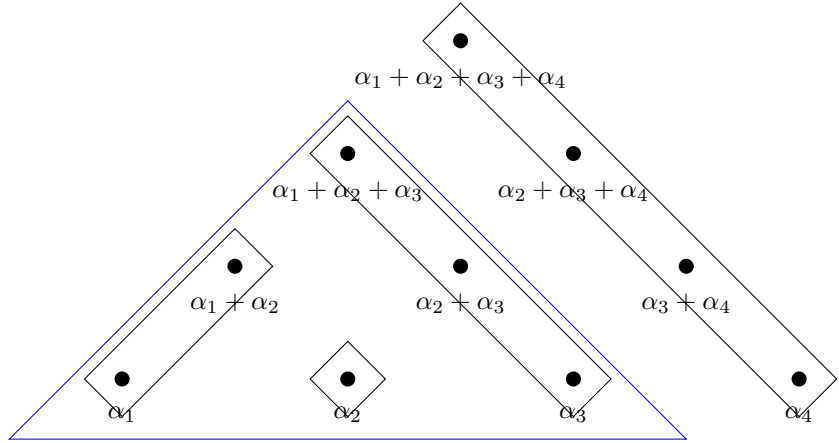
\vspace{.5cm}
    \begin{center}

\begin{tikzpicture}
\foreach  \y in {0,...,3}
    \foreach \x in {0,...,\y} 
        \fill [black] ( 3*\x -1.5*\y , -1.5*\y) circle [radius=0.1];

\node[below = 7pt] at (-4.5,-4.5) {$\alpha_{1}$};
\node[below = 7pt] at (-1.5,-4.5) {$\alpha_{2}$};
\node[below = 7pt] at (1.5,-4.5) {$\alpha_{3}$};
\node[below = 7pt] at (4.5,-4.5) {$\alpha_{4}$};

\node[below = 7pt] at (-3,-3) {$\alpha_{1} + \alpha_{2}$};
\node[below = 7pt] at (0,-3) {$\alpha_{2} + \alpha_{3}$};
\node[below = 7pt] at (3,-3) {$\alpha_{3} + \alpha_{4}$};

\node[below = 7pt] at (-1.5,-1.5) {$\alpha_{1} + \alpha_{2}+ \alpha_{3}$};
\node[below = 7pt] at (1.5,-1.5) {$\alpha_{2} + \alpha_{3}+ \alpha_{4}$};
\node[below = 7pt] at (0,0) {$\alpha_{1} + \alpha_{2}+ \alpha_{3}+ \alpha_{4}$};

\draw (0,.5)--(-.5,0)--(4.5,-5)--(5,-4.5)--cycle;

\draw (-1.5,-1)--(2, -4.5)--(1.5,-5)--(-2,-1.5)--cycle;

\draw (-3 , -2.5)--(-2.5, -3)--(-4.5 , -5)--(-5, -4.5)--cycle;

\draw (-1.5 , -4)--(-2, -4.5)--(-1.5,-5)--(-1, -4.5)--cycle;

\draw[blue] (-1.5 , -.8)--(-6 , -5.3)--(3 , -5.3)--cycle;

\end{tikzpicture}

\end{center}
Given a based rectangle partition $P$ of the root poset of $A_{n}$, we use the terminology \emph{$P$ modulo compatible subtriangle flips} to denote the collection of all based rectangle partitions $P'$ that can be reached from $ P$ by performing a sequence of compatible subtriangle flips.

\begin{thm} \label{firstbijection}
    Let $P$ be a based rectangle partition of the root poset of $ A_{n}$. Let $ X$ denote the Coxter $n$-cube associated to $ P$ (by Theorem \ref{partitiontoncube}). Performing a sequence of compatible subtriangle flips on $ P$ corresponds to a reorientation of the Coxeter $n$-cube $X$. Hence, there is a bijection between the set of Coxeter $n$-cubes in $A_{n}$ modulo reorientation and the set of based rectangle partitions of the $A_{n}$ root poset modulo compatible subtriangle flips.
\end{thm}

\begin{proof}
    Let $ x_{1} , x_{2} , \dots , x_{n}$ be the terminal edges of $X$. We need to show that performing a compatible subtriangle flip on $P$ corresponds to reversing the direction of all oriented edges that are parallel to some terminal edge $ x_{k}$ (including $x_{k}$ itself). We also need to show that if we flip all edges parallel to the terminal edge $x_{k}$ (including $x_{k}$ itself), then there must be a compatible subtriangle flip corresponding to flipping all edges parallel to the edge $x_{k}$.

    Let $ J \seq S$ such that $ \Phi^{+}_{J}$ is a set of positive roots corresponding to a compatible subtriangle flip in $P$. Note that $ (W_{J} ,J)$ is isomorphic to $ A_{k}$ for some $ k =1,2, \dots , n$. Since $ (W_{J} , J)$ is isomorphic to $A_{k}$, note that there will be $k$ based rectangles within the compatible subtriangle $ \Phi^{+}_{J}$. These $k$ based rectangles within $ \Phi^{+}_{J}$ correspond to $k$ terminal edges of $X$. Without loss of generality, let us assume that these $k$ terminal edges are $ x_{1} , \dots , x_{k}$. Furthermore, since $ (W_{J} , J)$ is a Coxeter system that is isomorphic to $A_{k}$, one of the based rectangles in $ \Phi^{+}_{J}$ must be a highest based rectangle for the root poset $ \Phi^{+}_{J}$. Without loss of generality, let $ x_{k}$ be the terminal edge corresponding to this highest based rectangle in $ \Phi^{+}_{J}$.

    Since $ x_{k}$ corresponds to the highest based rectangle of $ \Phi^{+}_{J}$, it follows that $\Phi_{x_{k}}$ is equal to this highest based rectangle in $ \Phi_{J}^{+}$. I claim that performing the compatible subtriangle flip defined by $ \Phi^{+}_{J}$ corresponds to replacing $ x_{k}$ by $ x_{k}^{-1}$ and replacing $ x_{i}$ by $ x_{i}' : = x_{k}^{-1}x_{i}x_{k}$ for $ i =1,2, \dots , k-1$. Since $ \Phi_{J}^{+}$ is the set of positive roots for some Coxeter system that is isomorphic to the $A_{k}$ Coxeter system, and since $ \Phi_{x_{k}}$ is a highest based rectangle within $ \Phi^{+}_{J}$, we can apply Proposition \ref{bhgenrectangle} to conclude that $ x_{k} = \nu(\alpha_{r} , \Pi_{K} )$ for some $ r \in J$ and where $ K = J \sm \{ r \} $. Note that by construction of $ \nu(\alpha_{r} , \Pi_{K} )$, we have that:

    $$ x_{k} =\nu(\alpha_{r} , \Pi_{K} ) = w_{J}w_{K} $$
    where $ w_{J}$ and $ w_{K}$ denote the longest elements of $ W_{J}$ and $W_{K}$ respectively. Since these longest elements are involutions, we get:

    $$ x_{k}^{-1} = (w_{J}w_{K})^{-1} = w_{K}^{-1}w_{J}^{-1} = w_{K}w_{J}$$

    $$  = w_{J}w_{J}^{-1}w_{K}w_{J} = w_{J}w_{J}w_{K}w_{J}^{-1} = w_{J}w_{w_{J}Kw_{J}^{-1}} = \nu(\alpha_{w_{J}rw_{J}^{-1}} , \Pi_{w_{J}Kw_{J}^{-1}})$$
    But because $ (W_{J} , J)$ is isomorphic to the $ A_{k}$ Coxeter system, we know that conjugating $ J$ by $ w_{J}$ induces a non-trivial automorphism of the Coxeter diagram of $ (W_{J} ,J)$. More precisely, if $ s \in J$, then $ w_{J}sw_{J}^{-1} \in J$ will be the simple reflection obtained by reflecting $s$ along the center of the Coxeter diagram of $ (W_{J} , J)$. It follows that conjugating the subset $ K \seq J$ by the longest element $ w_{J}$ sends $ K$ to its reflection by the center of the Coxeter diagram of $ (W_{J} ,J)$. Since $ w_{J}Kw_{J}^{-1}$ is the reflection of $ K$ by the center of the Coxeter diagram of $ (W_{J} ,J)$, we conclude that the left inversion set of $ x_{k}^{-1} = \nu(\alpha_{w_{J}rw_{J}^{-1}} , \Pi_{w_{J}Kw_{J}^{-1}})$ is the mirror reflection of the left inversion set of $ x_{k}$ in the root poset of $ \Phi_{J}^{+}$. This shows that performing the  compatible subtriangle flip defined by $ \Phi_{J}^{+}$ replaces $ x_{k}$ with $ x_{k}^{-1}$.

    Let us now consider what happens to the $ x_{1} , x_{2} , \dots , x_{k-1}$. Note that $ x_{1}, x_{2} , \dots , x_{k}$ are the terminal edges of a Coxeter $ k$-cube in $ (W_{J} , J)$. Furthermore, since $ (W_{J} ,J)$ is isomorphic to $ A_{k}$, and since $ x_{k}$ corresponds to the highest based rectangle of $ \Phi^{+}_{J}$, we must have that the $k$-cube determined by $ x_{1} , x_{2} , \dots , x_{k}$ was built by applying the recursive process as described in the proof of Proposition \ref{ncubeexists}. In particular, we have the following Coxeter squares: 

    
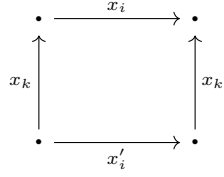
\[ \begin{tikzcd}
\sbullet \arrow[rr, "x_{i}"] & & \sbullet \\
\\
\sbullet \arrow[uu, "x_{k}"] \arrow[rr, swap, "x_{i}'"] & & \sbullet \arrow[uu, swap, "x_{k}"]
\end{tikzcd}
\]
where $ x_{k}x_{i}' = x_{i}x_{k}$ for all $ i =1,2, \dots ,k-1$. Flipping the edges labeled $ x_{k}$ in the above Coxeter square results in the following Coxeter square:  
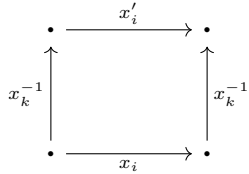
\[ \begin{tikzcd}
\sbullet \arrow[rr, "x_{i}'"] & & \sbullet \\
\\
\sbullet \arrow[uu, "x_{k}^{-1}"] \arrow[rr, swap, "x_{i}"] & & \sbullet \arrow[uu, swap, "x_{k}^{-1}"]
\end{tikzcd}
\]
Also, note that solving for $ x_{i}'$ gives $ x_{i}' = x_{k}^{-1}x_{i}x_{k}$. Hence, by flipping the edges labeled with $ x_{k}$, we obtain a new set of terminal edges given by $ x_{1}' , x_{2}', \dots, x_{k-1}' , x_{k}^{-1}$. Furthermore, since we proved earlier that $ x_{k}^{-1} = \nu(\alpha_{w_{J}rw_{J}^{-1}} , \Pi_{w_{J}Kw_{J}^{-1}})$ where the left inversion set of $ x_{k}$ is the highest based rectangle in $ \Phi_{J}^{+}$, it follows that the left inversion sets of $x_{1} , x_{2} , \dots, x_{k-1}$ are all contained in the standard parabolic subsystem $ (W_{w_{J}Kw_{J}^{-1}} ,w_{J}Kw_{J}^{-1} )$. Since $ x_{k}^{-1} = \nu(\alpha_{w_{J}rw_{J}^{-1}} , \Pi_{w_{J}Kw_{J}^{-1}})$ sends $ w_{J}Kw_{J}^{-1}$ to $ K$, where $ K$ and $ w_{J}Kw_{J}^{-1}$ are reflections of each other by the center of the Coxeter diagram of $ (W_{J} , J)$, it follows that the left inversion set of $ x_{i}' = x_{k}^{-1}x_{i}(x_{k}^{-1})^{-1}$ will be the mirror image of the left inversion set of $ x_{i}$ in $\Phi^{+}_{J}$ for $ i =1,2, \dots , k-1$. Hence, when the compatible subtriangle flip is performed on $ \Phi^{+}_{J}$, the left inversion sets of $ x_{1} , x_{2} , \dots x_{k-1}, x_{k}$ become the left inversion sets of $ x_{1}' , x_{2}' , \dots x_{k-1}', x_{k}^{-1} $. 

Since the based rectangle partition $P$ corresponds to the Coxeter $n$-cube $ X$, it follows that $X$ was formed via the recursive process as described in the proof of Proposition \ref{ncubeexists}. In particular, since $ x_{k+1} , x_{k+2} , \dots ,x_{n}$ are the terminal edges whose left inversion sets lie outside of the compatible subtriangle $ \Phi^{+}_{J}$, it follows that we obtain the following Coxeter squares:

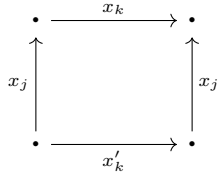
\[ \begin{tikzcd}
\sbullet \arrow[rr, "x_{k}"] & & \sbullet \\
\\
\sbullet \arrow[uu, "x_{j}"] \arrow[rr, swap, "x_{k}'"] & & \sbullet \arrow[uu, swap, "x_{j}"]
\end{tikzcd}
\]
where $ x_{k}x_{j} = x_{j}x_{k}'$ for $ j = k+1 , k+2, \dots , n$. Note that when we flip the edges that are parallel to the edge labeled $ x_{k}$ in the above Coxeter square (including the edge labeled $ x_{k}$), the element $ x_{j}$ will still be a terminal edge for $ j = k+1, k+2, \dots, n$. Thus, $ x_{k+1}, x_{k+2}, \dots , x_{n}$ will be terminal edges of the Coxeter cube that arises after performing the compatible subtriangle flip determined by $ \Phi^{+}_{J}$.

To summarize, we have have shown that performing the compatible subtriangle flip corresponding to $ \Phi^{+}_{J}$ sends the left inversion sets of $ x_{1} , \dots , x_{k-1}, x_{k} , x_{k+1}, \dots ,x_{n}$ to the left inversion sets of $ x_{1}' , \dots , x_{k-1}' ,x_{k}^{-1} , x_{k+1}, \dots , x_{n}$. But we showed that the elements $x_{1}' , \dots , x_{k-1}' ,x_{k}^{-1} , x_{k+1}, \dots , x_{n}$ are just the a different set of terminal edges after performing an edge flip on the Coxeter $n$-cube $X$. This shows that performing a compatible subtriangle flip corresponds to reversing the direction of all oriented edges that are parallel to some terminal edge $x_{k}$. 

Suppose now that $ x_{k}$ is some arbitrary terminal edge and that we wish to flip all edges that are parallel to $ x_{k}$ (including $x_{k}$). Note that $ \Phi_{x_{k}}$ will be some based rectangle in the root poset of $A_{n}$. But note that any based rectangle can be viewed as a highest based rectangle of some irreducible standard parabolic subsystem $ (W_{J} , J)$ of $A_{n}$. Hence, $ \Phi_{J}^{+}$ will be a compatible subtriangle of $P$. Furthermore, the compatible subtriangle flip by $ \Phi_{J}^{+}$ corresponds to flipping all edges parallel to $ x_{k}$. This shows that if we decide to flip all edges that are parallel to some terminal edge $ x_{k}$, then there is some compatible subtriangle flip corresponding to it.

This completes the proof that performing a sequence of compatible subtriangle flips corresponds to reorienting $X$.
\end{proof}

Before we proceed with the next few results, we need to define some terminology regarding graph theory. We define a \emph{binary tree} to be a directed graph that satisfies the following properties:

\begin{enumerate}
    \item There is a designated vertex called the \emph{root node}. The root node has an indegree of zero.

    \item The graph is connected, meaning that if $ v_{1}$ and $ v_{2}$ are vertices of the graph, then there exists a path of edges connecting $ v_{1}$ and $ v_{2}$ (paths ignore the direction of the edges).

    \item The graph does not have any cycles (cycles also ignore the direction of the edges).

    \item If $ v$ is a vertex of the graph, then the outdegree of $v$ is either $0$ or $2$.

\end{enumerate}

\begin{ex}
    Here are two examples of binary trees. The root nodes have been placed at the top of the trees. All edges are oriented downwards away from the root node.
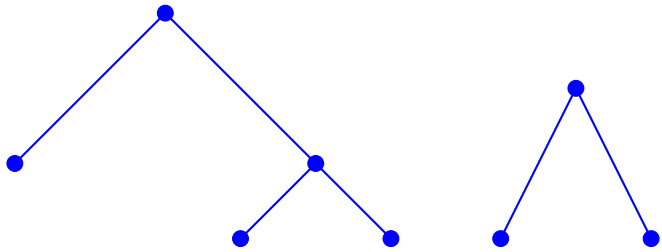
\vspace{.5cm}
\begin{center}
\begin{tikzpicture}
    \coordinate (A) at (0,0);
    \coordinate (B) at (2,-2);
    \coordinate (C) at (-2, -2);
    
    \coordinate (E) at (1, -3);
    \coordinate (F) at (3, -3);

    \draw[fill=blue , blue] (A) circle (3pt);
    \draw[fill=blue , blue] (B) circle (3pt);
    \draw[fill=blue , blue] (C) circle (3pt);
    
    \draw[fill=blue , blue] (E) circle (3pt);
    \draw[fill=blue , blue] (F) circle (3pt);

    \draw[blue , thick] (A)--(B);
    \draw[blue , thick] (A)--(C);
    \draw[blue , thick] (B)--(E);
    \draw[blue , thick] (B)--(F);
\end{tikzpicture}
\hspace{1cm}
\begin{tikzpicture}
    \coordinate (A) at (0,0);
    \coordinate (B) at (1,-2);
    \coordinate (C) at (-1,-2);

    \draw[fill=blue , blue] (A) circle (3pt);
    \draw[fill=blue , blue] (B) circle (3pt);
    \draw[fill=blue , blue] (C) circle (3pt);

    \draw[blue , thick] (A)--(B);
    \draw[blue , thick] (A)--(C);
    
\end{tikzpicture}
\end{center}

\end{ex}

A \emph{leaf} of a binary tree is a vertex of the binary tree whose outdegree is zero. A \emph{child} of a vertex $A$ is a vertex $B$ such that there is a single directed edge from $ A$ to $B$. If a vertex $ A$ has two children $B$ and $C$, we say that $ B$ is the left child of $A$ if $B$ lies to the left of $ C$ in the diagram of the binary tree. Analogously, we call $ C$ the right child of $A$. The \emph{subtree defined by the vertex $A$} is the collection of all vertices (including $A$) and directed edges that can be reached from $ A$ by going \textquotedblleft downward" in the diagram of the binary tree. 

\begin{ex}
    Let $F_{1}$ denote the following binary tree:

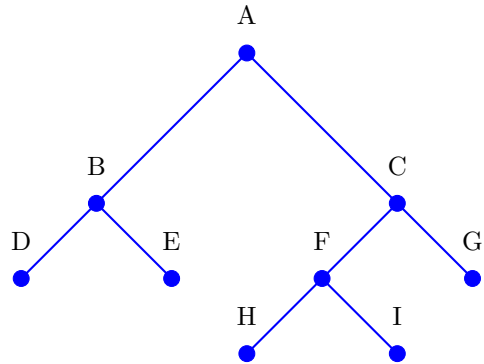
\vspace{.25cm}
\begin{center}
    \begin{tikzpicture}
    \coordinate (A) at (0,0);
    \coordinate (B) at (-2, -2); 
    \coordinate (C) at (2,-2);
    \coordinate (D) at (-3,-3);
    \coordinate (E) at (-1 ,-3 );
    \coordinate (F) at ( 1,-3 );
    \coordinate (G) at (3 ,-3 );
    \coordinate (H) at (0 ,-4 );
    \coordinate (I) at (2 , -4);

    \node[] at (0,.5) {A};
    \node[] at (-2,-1.5) {B};
    \node[] at (2,-1.5) {C};
    \node[] at (-3,-2.5) {D};
    \node[] at (-1, -2.5) {E};
    \node[] at (1,-2.5) {F};
    \node[] at (3, -2.5) {G};
    \node[] at (0,-3.5) {H};
    \node[] at (2, -3.5) {I};

    \draw[fill=blue , blue] (A) circle (3pt);
    \draw[fill=blue , blue] (C) circle (3pt);
    \draw[fill=blue , blue] (B) circle (3pt);
    \draw[fill=blue , blue] (D) circle (3pt);
    \draw[fill=blue , blue] (E) circle (3pt);
    \draw[fill=blue , blue] (F) circle (3pt);
    \draw[fill=blue , blue] (G) circle (3pt);
    \draw[fill=blue , blue] (H) circle (3pt);
    \draw[fill=blue , blue] (I) circle (3pt);

    \draw[blue , thick] (A)--(C);
    \draw[blue , thick] (A)--(B);
    \draw[blue , thick] (B)--(D);
    \draw[blue , thick] (B)--(E);
    \draw[blue , thick] (C)--(F);
    \draw[blue , thick] (C)--(G);
    \draw[blue , thick] (F)--(H);
    \draw[blue , thick] (F)--(I);
    
\end{tikzpicture}
\end{center}
The above binary tree has five leaves: $ D$, $E$, $H$, $I$, and $G$. The vertex $C$ has two children: $ F$ and $ G$. $F$ is the left child of $C$ and $ G$ is the right child of $C$. The subtree defined by $C$ is the following subgraph:

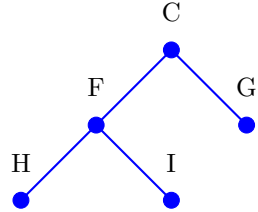
\begin{center}
    \begin{tikzpicture}
     
    \coordinate (C) at (2,-2);
    
    \coordinate (F) at ( 1,-3 );
    \coordinate (G) at (3 ,-3 );
    \coordinate (H) at (0 ,-4 );
    \coordinate (I) at (2 , -4);

    \node[] at (2,-1.5) {C};
    
    \node[] at (1,-2.5) {F};
    \node[] at (3, -2.5) {G};
    \node[] at (0,-3.5) {H};
    \node[] at (2, -3.5) {I};

    \draw[fill=blue , blue] (C) circle (3pt);
    
    \draw[fill=blue , blue] (F) circle (3pt);
    \draw[fill=blue , blue] (G) circle (3pt);
    \draw[fill=blue , blue] (H) circle (3pt);
    \draw[fill=blue , blue] (I) circle (3pt);

    \draw[blue , thick] (C)--(F);
    \draw[blue , thick] (C)--(G);
    \draw[blue , thick] (F)--(H);
    \draw[blue , thick] (F)--(I);
    
\end{tikzpicture}
\end{center}

\end{ex}

Let $ F_{1}$ and $F_{2}$ denote two binary trees. A \emph{morphism of binary trees} is a morphism of directed graphs $\phi : F_{1} \rightarrow F_{2}$ such that $\phi$ maps the root node of $ F_{1}$ to the root node of $F_{2}$. An \emph{isomorphism of binary trees} is a morphism of binary trees $ \phi : F_{1} \rightarrow F_{2}$ such that $ \phi^{-1} : F_{2} \rightarrow F_{1}$ exists and $ \phi^{-1}$ is a morphism of binary trees. If there exists an isomorphism of binary trees $ \phi : F_{1} \rightarrow F_{2}$, then we say that $F_{1}$ and $F_{2}$ are \emph{isomorphic}.

Let $ F_{1}$ denote a binary tree. Let $ X$ be a vertex of $F_{1}$. We define a \emph{subtree flip at the vertex $X$} to be the operation on the diagram of $ F_{1}$ that replaces the subtree of $X$ with its mirror image by the vertical line that runs through the vertex $ X$.

\begin{ex}
    Let $F_{1}$ denote the following binary tree:

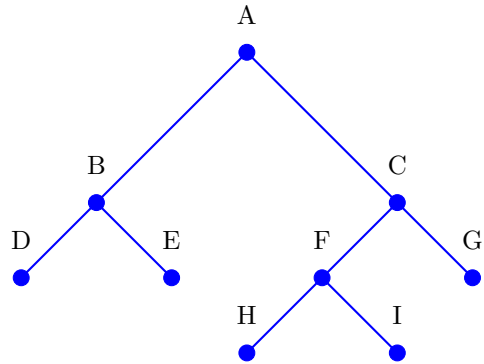
\begin{center}
    \begin{tikzpicture}
    \coordinate (A) at (0,0);
    \coordinate (B) at (-2, -2); 
    \coordinate (C) at (2,-2);
    \coordinate (D) at (-3,-3);
    \coordinate (E) at (-1 ,-3 );
    \coordinate (F) at ( 1,-3 );
    \coordinate (G) at (3 ,-3 );
    \coordinate (H) at (0 ,-4 );
    \coordinate (I) at (2 , -4);

    \node[] at (0,.5) {A};
    \node[] at (-2,-1.5) {B};
    \node[] at (2,-1.5) {C};
    \node[] at (-3,-2.5) {D};
    \node[] at (-1, -2.5) {E};
    \node[] at (1,-2.5) {F};
    \node[] at (3, -2.5) {G};
    \node[] at (0,-3.5) {H};
    \node[] at (2, -3.5) {I};

    \draw[fill=blue , blue] (A) circle (3pt);
    \draw[fill=blue , blue] (C) circle (3pt);
    \draw[fill=blue , blue] (B) circle (3pt);
    \draw[fill=blue , blue] (D) circle (3pt);
    \draw[fill=blue , blue] (E) circle (3pt);
    \draw[fill=blue , blue] (F) circle (3pt);
    \draw[fill=blue , blue] (G) circle (3pt);
    \draw[fill=blue , blue] (H) circle (3pt);
    \draw[fill=blue , blue] (I) circle (3pt);

    \draw[blue , thick] (A)--(C);
    \draw[blue , thick] (A)--(B);
    \draw[blue , thick] (B)--(D);
    \draw[blue , thick] (B)--(E);
    \draw[blue , thick] (C)--(F);
    \draw[blue , thick] (C)--(G);
    \draw[blue , thick] (F)--(H);
    \draw[blue , thick] (F)--(I);
    
\end{tikzpicture}
\end{center}
The subtree defined by $ C$ is the following:

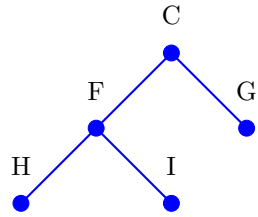
\begin{center}
    \begin{tikzpicture}
     
    \coordinate (C) at (2,-2);
    
    \coordinate (F) at ( 1,-3 );
    \coordinate (G) at (3 ,-3 );
    \coordinate (H) at (0 ,-4 );
    \coordinate (I) at (2 , -4);

    \node[] at (2,-1.5) {C};
    
    \node[] at (1,-2.5) {F};
    \node[] at (3, -2.5) {G};
    \node[] at (0,-3.5) {H};
    \node[] at (2, -3.5) {I};

    \draw[fill=blue , blue] (C) circle (3pt);
    
    \draw[fill=blue , blue] (F) circle (3pt);
    \draw[fill=blue , blue] (G) circle (3pt);
    \draw[fill=blue , blue] (H) circle (3pt);
    \draw[fill=blue , blue] (I) circle (3pt);

    \draw[blue , thick] (C)--(F);
    \draw[blue , thick] (C)--(G);
    \draw[blue , thick] (F)--(H);
    \draw[blue , thick] (F)--(I);
    
\end{tikzpicture}
\end{center}
The mirror image of the above subtree is the following:

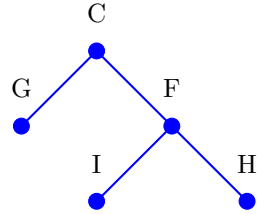
\begin{center}
    \begin{tikzpicture}
     
    \coordinate (C) at (-2,-2);
    
    \coordinate (F) at ( -1,-3 );
    \coordinate (G) at (-3 ,-3 );
    \coordinate (H) at (-0 ,-4 );
    \coordinate (I) at (-2 , -4);

    \node[] at (-2,-1.5) {C};
    
    \node[] at (-1,-2.5) {F};
    \node[] at (-3, -2.5) {G};
    \node[] at (-0,-3.5) {H};
    \node[] at (-2, -3.5) {I};

    \draw[fill=blue , blue] (C) circle (3pt);
    
    \draw[fill=blue , blue] (F) circle (3pt);
    \draw[fill=blue , blue] (G) circle (3pt);
    \draw[fill=blue , blue] (H) circle (3pt);
    \draw[fill=blue , blue] (I) circle (3pt);

    \draw[blue , thick] (C)--(F);
    \draw[blue , thick] (C)--(G);
    \draw[blue , thick] (F)--(H);
    \draw[blue , thick] (F)--(I);
    
\end{tikzpicture}
\end{center}
Hence, when we perform the subtree flip at the vertex $C$ on the diagram of $F_{1}$, we obtain the new diagram:

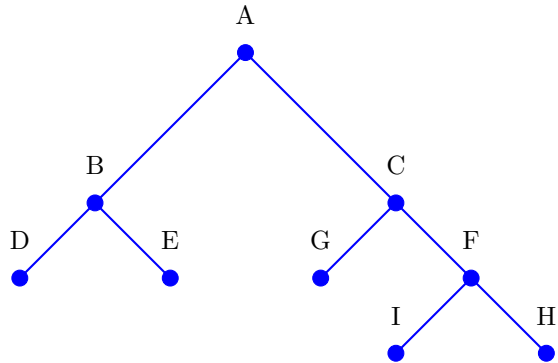
\begin{center}
    \begin{tikzpicture}
    \coordinate (A) at (0,0);
    \coordinate (B) at (-2, -2); 
    \coordinate (C) at (2,-2);
    \coordinate (D) at (-3,-3);
    \coordinate (E) at (-1 ,-3 );
    \coordinate (G) at ( 1,-3 );
    \coordinate (F) at (3 ,-3 );
    \coordinate (H) at (2 ,-4 );
    \coordinate (I) at (4 , -4);

    \node[] at (0,.5) {A};
    \node[] at (-2,-1.5) {B};
    \node[] at (2,-1.5) {C};
    \node[] at (-3,-2.5) {D};
    \node[] at (-1, -2.5) {E};
    \node[] at (1,-2.5) {G};
    \node[] at (3, -2.5) {F};
    \node[] at (2,-3.5) {I};
    \node[] at (4, -3.5) {H};

    \draw[fill=blue , blue] (A) circle (3pt);
    \draw[fill=blue , blue] (C) circle (3pt);
    \draw[fill=blue , blue] (B) circle (3pt);
    \draw[fill=blue , blue] (D) circle (3pt);
    \draw[fill=blue , blue] (E) circle (3pt);
    \draw[fill=blue , blue] (F) circle (3pt);
    \draw[fill=blue , blue] (G) circle (3pt);
    \draw[fill=blue , blue] (H) circle (3pt);
    \draw[fill=blue , blue] (I) circle (3pt);

    \draw[blue , thick] (A)--(C);
    \draw[blue , thick] (A)--(B);
    \draw[blue , thick] (B)--(D);
    \draw[blue , thick] (B)--(E);
    \draw[blue , thick] (C)--(F);
    \draw[blue , thick] (C)--(G);
    \draw[blue , thick] (F)--(H);
    \draw[blue , thick] (F)--(I);
    
\end{tikzpicture}
\end{center}
\end{ex}

\begin{prop} \label{binarytreeisom}
    Let $ F_{1}$ and $F_{2}$ denote two binary trees. Then $F_{1}$ and $F_{2}$ are isomorphic if and only if the diagram of $F_{2}$ can be obtained from the diagram of $F_{1}$ via a sequence of subtree flips. 
\end{prop}

\begin{proof}
    Note that performing a subtree flip on a binary tree does not change the actualy structure of the binary tree. Hence, if the diagram of $ F_{2}$ can be obtained from the diagram of $F_{1}$ via a sequence of subtree flips, then it follows that $ F_{1}$ and $F_{2}$ are isomorphic.

    Suppose now that $ F_{1}$ and $F_{2}$ are isomorphic. Let $ \phi : F_{1} \to F_{2}$ denote an isomorphism. Let $ A$ be the root node of $F_{1}$. Let $ B$ denote the left child of $A$, and let $C$ denote the right child of $A$. Note that $ \phi(A)$ will be the root node of $F_{2}$. Furthermore, $ \phi(B)$ and $ \phi(C)$ will be the children of $ \phi(A)$. Since $ A \mapsto \phi(A)$, where $ \phi$ is an isomorphism of binary trees, it follows that the subtree defined at $A$ is isomorphic to the subtree defined at $\phi(A)$. Similarly, since $ B \mapsto \phi(B)$, it follows that the subtree defined at $B$ is isomorphic to the subtree defined at $\phi(B)$. There are two cases to consider:

    \begin{itemize}
        \item Case 1: $\phi(B)$ is the left child of $ \phi(A)$, and $\phi(C)$ is the right child of $\phi(A)$.

        \item Case 2: $ \phi(C)$ is the left child of $ \phi(A)$, and $\phi(B)$ is the right child of $\phi(A)$.
    \end{itemize}

    Suppose that Case 1 occurs. Since $ B$ and $ \phi(B)$ are the left children of $ A$ and $ \phi(A)$ respectively, and since the subtree defined at $B$ has strictly fewer vertices than the entire binary tree $F_{1}$, one can use induction (on the number of vertices of a binary tree) to show that the diagram of the subtree defined at $B$ can be transformed into the diagram of the subtree defined at $ \Phi(B)$ via a sequence of subtree flips. Similarly, since $ C$ and $ \phi(C)$ are the right children of $A$ and $ \phi(A)$ respectively, one can use induction to show that the diagram of the subtree defined at $C$ can be transformed into the diagram of the subtree defined at $ \Phi(C)$ via a sequence of subtree flips. This process shows that the entire diagram of $F_{1}$ can be transformed into the diagram of $F_{2}$ via a sequence of subtree flips.

    Suppose now that Case 2 occurs. Thus, $ \phi(C)$ is the left child of $ \phi(A)$, and $\phi(B)$ is the right child of $\phi(A)$. Perform a subtree flip at $ \phi(A)$ on the diagram of $F_{2}$. This will cause $ \phi(B)$ to become the left child of $\phi(A)$ and $\phi(C)$ to become the right child of $ \phi(A)$. But now we have transformed Case 2 into Case 1. Hence, by the conclusion of Case 1, we deduce that the diagram of $ F_{1}$ can be transformed into the diagram of $ F_{2}$ via a sequence of subtree flips.
\end{proof}

\begin{thm}
    There is a bijection between the set of Coxeter $n$-cubes modulo reoriention in the $A_{n}$ Coxeter system and the set of binary trees with $ n+1$ leaves modulo isomorphism.
\end{thm}

\begin{proof}
    By Theorem \ref{firstbijection}, it suffices to establish a bijection between based rectangle partitions of $A_{n}$ modulo compatible subtrianngle flips and binary trees with $n+1$ leaves.

    Let $ P$ denote a based rectangle partition of the root poset of $A_{n}$. We construct a diagram of a binary tree associated to $P$ as follows:

    \begin{enumerate}
        \item Place a vertex at the highest root of the $A_{n}$ root poset. This vertex will be the root node of the binary tree.

        \item  Given some already constructed vertex $X$ of the binary tree that lies on a root of the root poset of $A_{n}$, draw a (directed) line segment from $X$ diagonally downward and to the left of $X$ along the root poset (if such a directed line segment already exists, omit this step). Also draw a (directed) line segment from $X$ diagonally downward and to the right of $X$ along the root poset (if such a directed line segment already exists, omit this step). Once these line segments exit the diagram of the root poset of $A_{n}$, draw a vertex at the end of each of these line segments. 

        \item  Suppose that one of the (directed) line segments of the already constructed graph passes through a root $\gamma$, such that $ \gamma$ is the highest root of some based rectangle. If $\gamma$ does not already have a vertex placed on it, place a vertex at the root $\gamma$.

        \item Keep iterating through steps (2) and (3) until you cannot add any more vertices or line segments.
        
    \end{enumerate}
    Note that this procedure guarantees that the directed graph produced will be a binary tree since at each vertex within the root poset of $A_{n}$, there will always exist two directed line segments leaving such a vertex. Furthermore, since any based rectangle partition of $A_{n}$ will have $n$ based rectangles, it follows that the binary tree will bifurcate $n$ times, resulting in $n+1$ leaves. We now give an explicit example of this process. Consider the following based rectangle partition $P$ of the root poset of $A_{4}$:

    
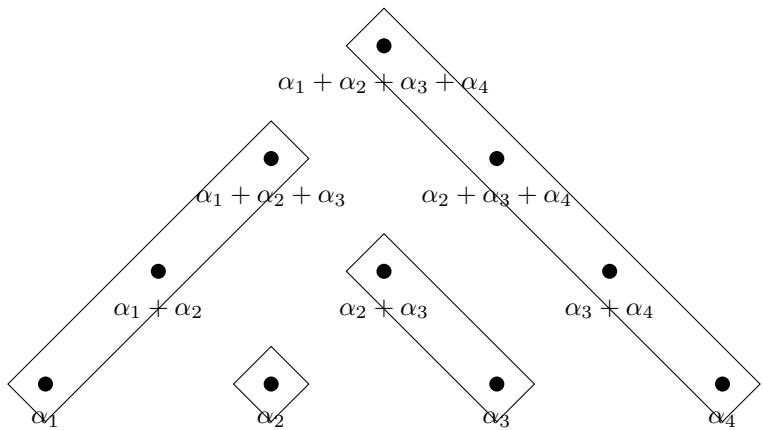
\begin{center}

\begin{tikzpicture}
\foreach  \y in {0,...,3}
    \foreach \x in {0,...,\y} 
        \fill [black] ( 3*\x -1.5*\y , -1.5*\y) circle [radius=0.1];

\node[below = 7pt] at (-4.5,-4.5) {$\alpha_{1}$};
\node[below = 7pt] at (-1.5,-4.5) {$\alpha_{2}$};
\node[below = 7pt] at (1.5,-4.5) {$\alpha_{3}$};
\node[below = 7pt] at (4.5,-4.5) {$\alpha_{4}$};

\node[below = 7pt] at (-3,-3) {$\alpha_{1} + \alpha_{2}$};
\node[below = 7pt] at (0,-3) {$\alpha_{2} + \alpha_{3}$};
\node[below = 7pt] at (3,-3) {$\alpha_{3} + \alpha_{4}$};

\node[below = 7pt] at (-1.5,-1.5) {$\alpha_{1} + \alpha_{2}+ \alpha_{3}$};
\node[below = 7pt] at (1.5,-1.5) {$\alpha_{2} + \alpha_{3}+ \alpha_{4}$};
\node[below = 7pt] at (0,0) {$\alpha_{1} + \alpha_{2}+ \alpha_{3}+ \alpha_{4}$};

\draw (0,.5)--(-.5,0)--(4.5,-5)--(5,-4.5)--cycle;

\draw (-1.5,-1)--(-1, -1.5)--(-4.5,-5)--(-5,-4.5)--cycle;

\draw (0 , -2.5)--(-.5,-3)--(1.5,-5)--(2,-4.5)--cycle;

\draw (-1.5 , -4)--(-2, -4.5)--(-1.5,-5)--(-1, -4.5)--cycle;

\end{tikzpicture}

\end{center}
We start by placing a vertex (denoted by a blue dot) at the highest root of the root poset:

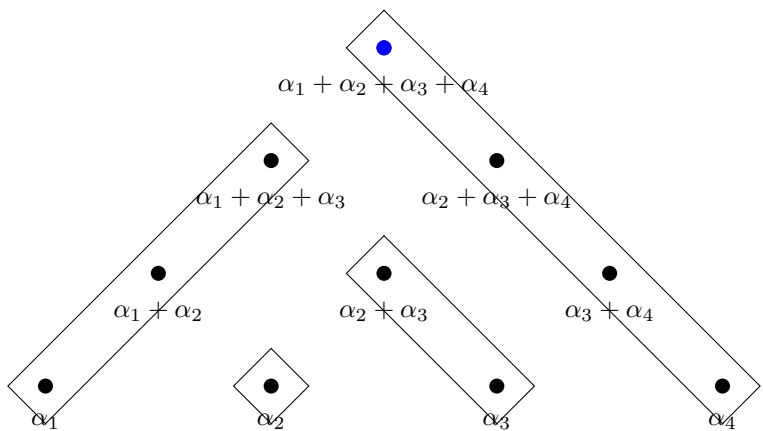
\begin{center}

\begin{tikzpicture}
\foreach  \y in {0,...,3}
    \foreach \x in {0,...,\y} 
        \fill [black] ( 3*\x -1.5*\y , -1.5*\y) circle [radius=0.1];

\fill [blue] (0,0) circle [radius=0.1];

\node[below = 7pt] at (-4.5,-4.5) {$\alpha_{1}$};
\node[below = 7pt] at (-1.5,-4.5) {$\alpha_{2}$};
\node[below = 7pt] at (1.5,-4.5) {$\alpha_{3}$};
\node[below = 7pt] at (4.5,-4.5) {$\alpha_{4}$};

\node[below = 7pt] at (-3,-3) {$\alpha_{1} + \alpha_{2}$};
\node[below = 7pt] at (0,-3) {$\alpha_{2} + \alpha_{3}$};
\node[below = 7pt] at (3,-3) {$\alpha_{3} + \alpha_{4}$};

\node[below = 7pt] at (-1.5,-1.5) {$\alpha_{1} + \alpha_{2}+ \alpha_{3}$};
\node[below = 7pt] at (1.5,-1.5) {$\alpha_{2} + \alpha_{3}+ \alpha_{4}$};
\node[below = 7pt] at (0,0) {$\alpha_{1} + \alpha_{2}+ \alpha_{3}+ \alpha_{4}$};

\draw (0,.5)--(-.5,0)--(4.5,-5)--(5,-4.5)--cycle;

\draw (-1.5,-1)--(-1, -1.5)--(-4.5,-5)--(-5,-4.5)--cycle;

\draw (0 , -2.5)--(-.5,-3)--(1.5,-5)--(2,-4.5)--cycle;

\draw (-1.5 , -4)--(-2, -4.5)--(-1.5,-5)--(-1, -4.5)--cycle;
\end{tikzpicture}
\end{center}
We now draw  downward (directed) diagonal line segments to the left and to the right of this blue root node. After these line segments exit the root poset, we place blue vertices at their endpoints: 
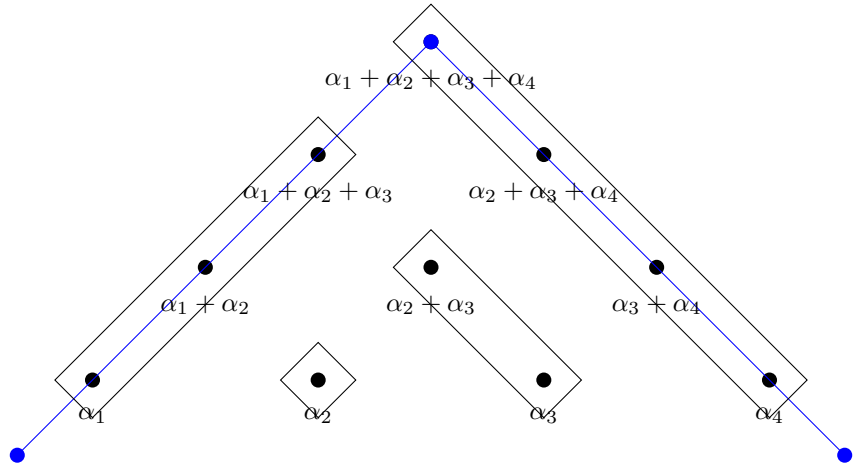
\begin{center}
\begin{tikzpicture}
\foreach  \y in {0,...,3}
    \foreach \x in {0,...,\y} 
        \fill [black] ( 3*\x -1.5*\y , -1.5*\y) circle [radius=0.1];

\fill [blue] (0,0) circle [radius=0.1];

\draw[blue] (0,0)-- (-5.5,-5.5);
\draw[blue] (0,0)--(5.5,-5.5);

\fill [blue] (-5.5,-5.5) circle [radius=0.1];
\fill [blue] (5.5,-5.5) circle [radius=0.1];

\node[below = 7pt] at (-4.5,-4.5) {$\alpha_{1}$};
\node[below = 7pt] at (-1.5,-4.5) {$\alpha_{2}$};
\node[below = 7pt] at (1.5,-4.5) {$\alpha_{3}$};
\node[below = 7pt] at (4.5,-4.5) {$\alpha_{4}$};

\node[below = 7pt] at (-3,-3) {$\alpha_{1} + \alpha_{2}$};
\node[below = 7pt] at (0,-3) {$\alpha_{2} + \alpha_{3}$};
\node[below = 7pt] at (3,-3) {$\alpha_{3} + \alpha_{4}$};

\node[below = 7pt] at (-1.5,-1.5) {$\alpha_{1} + \alpha_{2}+ \alpha_{3}$};
\node[below = 7pt] at (1.5,-1.5) {$\alpha_{2} + \alpha_{3}+ \alpha_{4}$};
\node[below = 7pt] at (0,0) {$\alpha_{1} + \alpha_{2}+ \alpha_{3}+ \alpha_{4}$};

\draw (0,.5)--(-.5,0)--(4.5,-5)--(5,-4.5)--cycle;

\draw (-1.5,-1)--(-1, -1.5)--(-4.5,-5)--(-5,-4.5)--cycle;

\draw (0 , -2.5)--(-.5,-3)--(1.5,-5)--(2,-4.5)--cycle;

\draw (-1.5 , -4)--(-2, -4.5)--(-1.5,-5)--(-1, -4.5)--cycle;
\end{tikzpicture}
\end{center}
Now note that the root $ \alpha_{1} + \alpha_{2} + \alpha_{3}$ is the highest root of the based rectangle $\{ \alpha_{1} , \alpha_{1} + \alpha_{2} ,  \alpha_{1} + \alpha_{2} + \alpha_{3} \}$. Hence, we place a blue vertex at $ \alpha_{1} + \alpha_{2} + \alpha_{3}$:
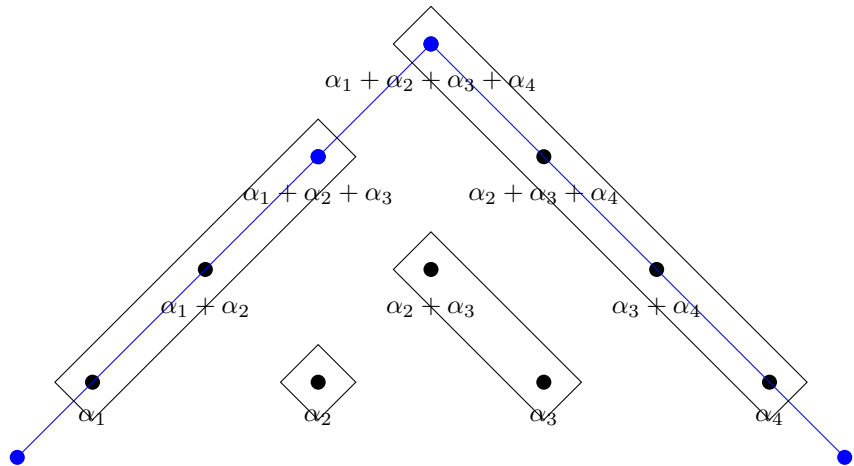
\begin{center}
\begin{tikzpicture}
\foreach  \y in {0,...,3}
    \foreach \x in {0,...,\y} 
        \fill [black] ( 3*\x -1.5*\y , -1.5*\y) circle [radius=0.1];

\fill [blue] (0,0) circle [radius=0.1];

\draw[blue] (0,0)-- (-5.5,-5.5);
\draw[blue] (0,0)--(5.5,-5.5);

\fill [blue] (-5.5,-5.5) circle [radius=0.1];
\fill [blue] (5.5,-5.5) circle [radius=0.1];

\fill [blue] (-1.5,-1.5) circle [radius = 0.1];

\node[below = 7pt] at (-4.5,-4.5) {$\alpha_{1}$};
\node[below = 7pt] at (-1.5,-4.5) {$\alpha_{2}$};
\node[below = 7pt] at (1.5,-4.5) {$\alpha_{3}$};
\node[below = 7pt] at (4.5,-4.5) {$\alpha_{4}$};

\node[below = 7pt] at (-3,-3) {$\alpha_{1} + \alpha_{2}$};
\node[below = 7pt] at (0,-3) {$\alpha_{2} + \alpha_{3}$};
\node[below = 7pt] at (3,-3) {$\alpha_{3} + \alpha_{4}$};

\node[below = 7pt] at (-1.5,-1.5) {$\alpha_{1} + \alpha_{2}+ \alpha_{3}$};
\node[below = 7pt] at (1.5,-1.5) {$\alpha_{2} + \alpha_{3}+ \alpha_{4}$};
\node[below = 7pt] at (0,0) {$\alpha_{1} + \alpha_{2}+ \alpha_{3}+ \alpha_{4}$};

\draw (0,.5)--(-.5,0)--(4.5,-5)--(5,-4.5)--cycle;

\draw (-1.5,-1)--(-1, -1.5)--(-4.5,-5)--(-5,-4.5)--cycle;

\draw (0 , -2.5)--(-.5,-3)--(1.5,-5)--(2,-4.5)--cycle;

\draw (-1.5 , -4)--(-2, -4.5)--(-1.5,-5)--(-1, -4.5)--cycle;
\end{tikzpicture}
\end{center}
Note that the blue vertex at $ \alpha_{1} + \alpha_{2} + \alpha_{3}$ already has a line segment directed downwards and to the left. Thus, we only need to draw a line segment from $ \alpha_{1} + \alpha_{2} + \alpha_{3}$ going downward and to the right, and draw a blue vertex at the end of this line segment once this line segment leaves the root poset:

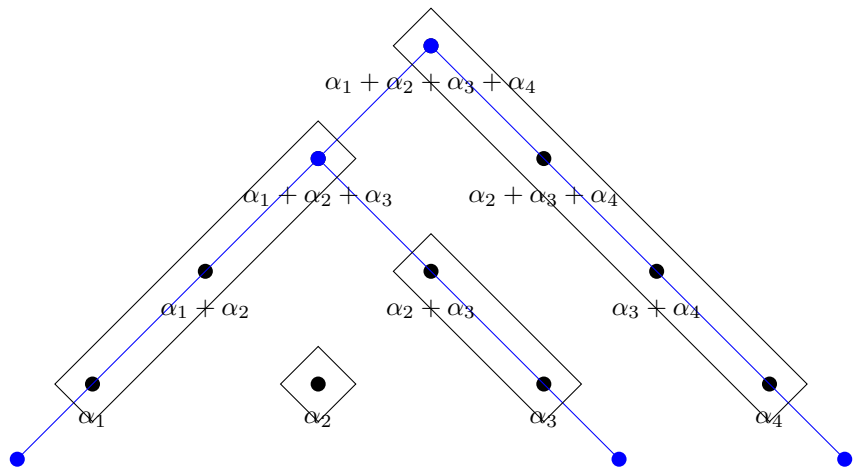
\begin{center}
\begin{tikzpicture}
\foreach  \y in {0,...,3}
    \foreach \x in {0,...,\y} 
        \fill [black] ( 3*\x -1.5*\y , -1.5*\y) circle [radius=0.1];

\fill [blue] (0,0) circle [radius=0.1];

\draw[blue] (0,0)-- (-5.5,-5.5);
\draw[blue] (0,0)--(5.5,-5.5);

\fill [blue] (-5.5,-5.5) circle [radius=0.1];
\fill [blue] (5.5,-5.5) circle [radius=0.1];

\fill [blue] (-1.5,-1.5) circle [radius = 0.1];

\draw[blue] (-1.5,-1.5)--(2.5 ,-5.5);
\fill[blue] (2.5, -5.5) circle [radius = 0.1];

\node[below = 7pt] at (-4.5,-4.5) {$\alpha_{1}$};
\node[below = 7pt] at (-1.5,-4.5) {$\alpha_{2}$};
\node[below = 7pt] at (1.5,-4.5) {$\alpha_{3}$};
\node[below = 7pt] at (4.5,-4.5) {$\alpha_{4}$};

\node[below = 7pt] at (-3,-3) {$\alpha_{1} + \alpha_{2}$};
\node[below = 7pt] at (0,-3) {$\alpha_{2} + \alpha_{3}$};
\node[below = 7pt] at (3,-3) {$\alpha_{3} + \alpha_{4}$};

\node[below = 7pt] at (-1.5,-1.5) {$\alpha_{1} + \alpha_{2}+ \alpha_{3}$};
\node[below = 7pt] at (1.5,-1.5) {$\alpha_{2} + \alpha_{3}+ \alpha_{4}$};
\node[below = 7pt] at (0,0) {$\alpha_{1} + \alpha_{2}+ \alpha_{3}+ \alpha_{4}$};

\draw (0,.5)--(-.5,0)--(4.5,-5)--(5,-4.5)--cycle;

\draw (-1.5,-1)--(-1, -1.5)--(-4.5,-5)--(-5,-4.5)--cycle;

\draw (0 , -2.5)--(-.5,-3)--(1.5,-5)--(2,-4.5)--cycle;

\draw (-1.5 , -4)--(-2, -4.5)--(-1.5,-5)--(-1, -4.5)--cycle;
\end{tikzpicture}
\end{center}
But now we see that the root $ \alpha_{2} + \alpha_{3}$ is the highest root of the based rectangle $ \{  \alpha_{3}, \alpha_{2} + \alpha_{3} \}$. Since $ \alpha_{2} + \alpha_{3}$ lies on a (directed) line segment, we draw a blue vertex at $ \{ \alpha_{2} + \alpha_{3} \}$:
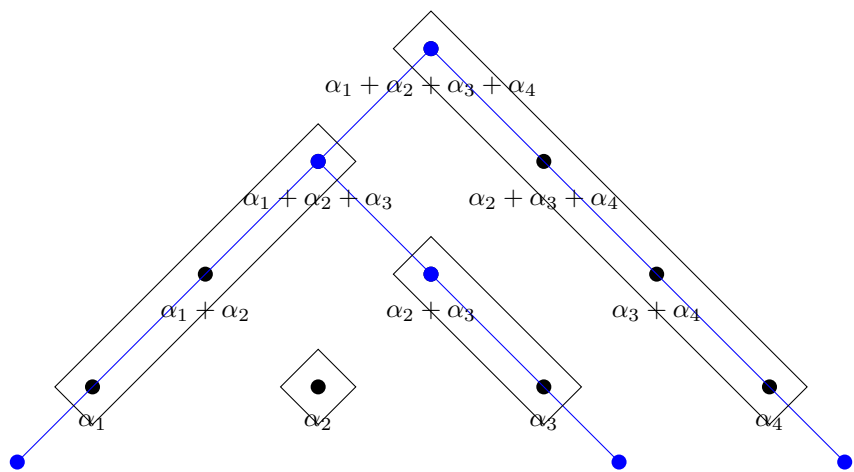
\begin{center}
\begin{tikzpicture}
\foreach  \y in {0,...,3}
    \foreach \x in {0,...,\y} 
        \fill [black] ( 3*\x -1.5*\y , -1.5*\y) circle [radius=0.1];

\fill [blue] (0,0) circle [radius=0.1];

\draw[blue] (0,0)-- (-5.5,-5.5);
\draw[blue] (0,0)--(5.5,-5.5);

\fill [blue] (-5.5,-5.5) circle [radius=0.1];
\fill [blue] (5.5,-5.5) circle [radius=0.1];

\fill [blue] (-1.5,-1.5) circle [radius = 0.1];

\draw[blue] (-1.5,-1.5)--(2.5 ,-5.5);
\fill[blue] (2.5, -5.5) circle [radius = 0.1];

\fill[blue] (0,-3) circle [radius = 0.1];

\node[below = 7pt] at (-4.5,-4.5) {$\alpha_{1}$};
\node[below = 7pt] at (-1.5,-4.5) {$\alpha_{2}$};
\node[below = 7pt] at (1.5,-4.5) {$\alpha_{3}$};
\node[below = 7pt] at (4.5,-4.5) {$\alpha_{4}$};

\node[below = 7pt] at (-3,-3) {$\alpha_{1} + \alpha_{2}$};
\node[below = 7pt] at (0,-3) {$\alpha_{2} + \alpha_{3}$};
\node[below = 7pt] at (3,-3) {$\alpha_{3} + \alpha_{4}$};

\node[below = 7pt] at (-1.5,-1.5) {$\alpha_{1} + \alpha_{2}+ \alpha_{3}$};
\node[below = 7pt] at (1.5,-1.5) {$\alpha_{2} + \alpha_{3}+ \alpha_{4}$};
\node[below = 7pt] at (0,0) {$\alpha_{1} + \alpha_{2}+ \alpha_{3}+ \alpha_{4}$};

\draw (0,.5)--(-.5,0)--(4.5,-5)--(5,-4.5)--cycle;

\draw (-1.5,-1)--(-1, -1.5)--(-4.5,-5)--(-5,-4.5)--cycle;

\draw (0 , -2.5)--(-.5,-3)--(1.5,-5)--(2,-4.5)--cycle;

\draw (-1.5 , -4)--(-2, -4.5)--(-1.5,-5)--(-1, -4.5)--cycle;
\end{tikzpicture}
\end{center}
The blue vertex at $ \alpha_{2} + \alpha_{3}$ already had a downward directed line segment to the right. Hence, we just need to draw a downward directed line segment from $ \alpha_{2} + \alpha_{3}$ to the left: 
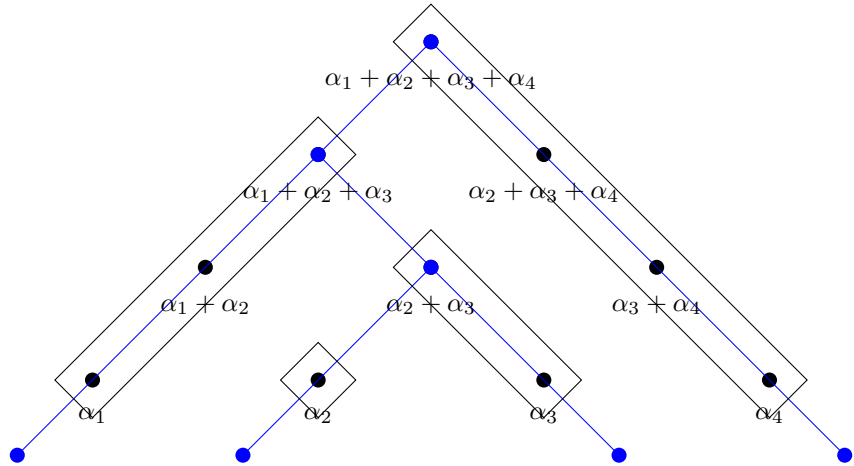
\begin{center}
\begin{tikzpicture}
\foreach  \y in {0,...,3}
    \foreach \x in {0,...,\y} 
        \fill [black] ( 3*\x -1.5*\y , -1.5*\y) circle [radius=0.1];

\fill [blue] (0,0) circle [radius=0.1];

\draw[blue] (0,0)-- (-5.5,-5.5);
\draw[blue] (0,0)--(5.5,-5.5);

\fill [blue] (-5.5,-5.5) circle [radius=0.1];
\fill [blue] (5.5,-5.5) circle [radius=0.1];

\fill [blue] (-1.5,-1.5) circle [radius = 0.1];

\draw[blue] (-1.5,-1.5)--(2.5 ,-5.5);
\fill[blue] (2.5, -5.5) circle [radius = 0.1];

\fill[blue] (0,-3) circle [radius = 0.1];

\draw[blue] (0,-3)--(-2.5 , -5.5);
\fill[blue] (-2.5,-5.5) circle [radius = 0.1];

\node[below = 7pt] at (-4.5,-4.5) {$\alpha_{1}$};
\node[below = 7pt] at (-1.5,-4.5) {$\alpha_{2}$};
\node[below = 7pt] at (1.5,-4.5) {$\alpha_{3}$};
\node[below = 7pt] at (4.5,-4.5) {$\alpha_{4}$};

\node[below = 7pt] at (-3,-3) {$\alpha_{1} + \alpha_{2}$};
\node[below = 7pt] at (0,-3) {$\alpha_{2} + \alpha_{3}$};
\node[below = 7pt] at (3,-3) {$\alpha_{3} + \alpha_{4}$};

\node[below = 7pt] at (-1.5,-1.5) {$\alpha_{1} + \alpha_{2}+ \alpha_{3}$};
\node[below = 7pt] at (1.5,-1.5) {$\alpha_{2} + \alpha_{3}+ \alpha_{4}$};
\node[below = 7pt] at (0,0) {$\alpha_{1} + \alpha_{2}+ \alpha_{3}+ \alpha_{4}$};

\draw (0,.5)--(-.5,0)--(4.5,-5)--(5,-4.5)--cycle;

\draw (-1.5,-1)--(-1, -1.5)--(-4.5,-5)--(-5,-4.5)--cycle;

\draw (0 , -2.5)--(-.5,-3)--(1.5,-5)--(2,-4.5)--cycle;

\draw (-1.5 , -4)--(-2, -4.5)--(-1.5,-5)--(-1, -4.5)--cycle;
\end{tikzpicture}
\end{center}
Now note that $ \alpha_{2}$ is the highest root of the based rectangle $ \{ \alpha_{2}  \}$. Hence, draw a blue vertex at $ \alpha_{2}$, and draw a directed line segment downward from $ \alpha_{2}$ and to the right:
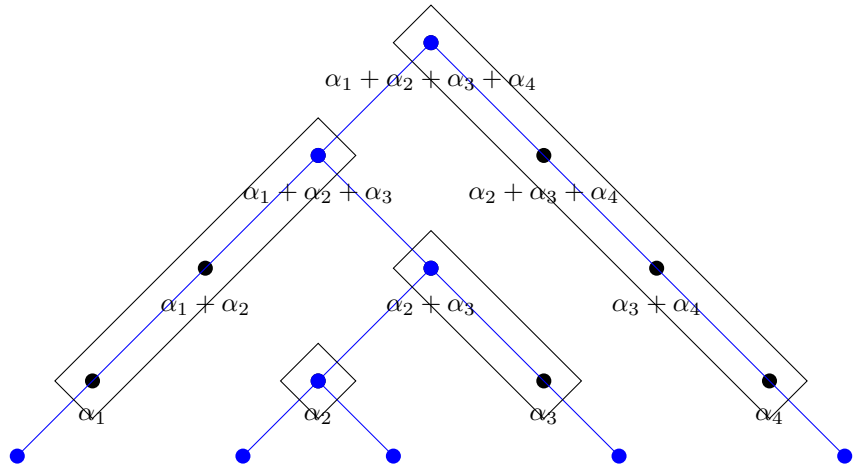
\begin{center}
\begin{tikzpicture}
\foreach  \y in {0,...,3}
    \foreach \x in {0,...,\y} 
        \fill [black] ( 3*\x -1.5*\y , -1.5*\y) circle [radius=0.1];

\fill [blue] (0,0) circle [radius=0.1];

\draw[blue] (0,0)-- (-5.5,-5.5);
\draw[blue] (0,0)--(5.5,-5.5);

\fill [blue] (-5.5,-5.5) circle [radius=0.1];
\fill [blue] (5.5,-5.5) circle [radius=0.1];

\fill [blue] (-1.5,-1.5) circle [radius = 0.1];

\draw[blue] (-1.5,-1.5)--(2.5 ,-5.5);
\fill[blue] (2.5, -5.5) circle [radius = 0.1];

\fill[blue] (0,-3) circle [radius = 0.1];

\draw[blue] (0,-3)--(-2.5 , -5.5);
\fill[blue] (-2.5,-5.5) circle [radius = 0.1];

\fill[blue] (-1.5, -4.5) circle [radius = 0.1];
\draw[blue] (-1.5, -4.5)--(-.5 , -5.5);
\fill[blue] (-.5,-5.5) circle [radius = 0.1];

\node[below = 7pt] at (-4.5,-4.5) {$\alpha_{1}$};
\node[below = 7pt] at (-1.5,-4.5) {$\alpha_{2}$};
\node[below = 7pt] at (1.5,-4.5) {$\alpha_{3}$};
\node[below = 7pt] at (4.5,-4.5) {$\alpha_{4}$};

\node[below = 7pt] at (-3,-3) {$\alpha_{1} + \alpha_{2}$};
\node[below = 7pt] at (0,-3) {$\alpha_{2} + \alpha_{3}$};
\node[below = 7pt] at (3,-3) {$\alpha_{3} + \alpha_{4}$};

\node[below = 7pt] at (-1.5,-1.5) {$\alpha_{1} + \alpha_{2}+ \alpha_{3}$};
\node[below = 7pt] at (1.5,-1.5) {$\alpha_{2} + \alpha_{3}+ \alpha_{4}$};
\node[below = 7pt] at (0,0) {$\alpha_{1} + \alpha_{2}+ \alpha_{3}+ \alpha_{4}$};

\draw (0,.5)--(-.5,0)--(4.5,-5)--(5,-4.5)--cycle;

\draw (-1.5,-1)--(-1, -1.5)--(-4.5,-5)--(-5,-4.5)--cycle;

\draw (0 , -2.5)--(-.5,-3)--(1.5,-5)--(2,-4.5)--cycle;

\draw (-1.5 , -4)--(-2, -4.5)--(-1.5,-5)--(-1, -4.5)--cycle;
\end{tikzpicture}
\end{center}
All based rectangles in the above diagram have blue vertices at their highest roots. Hence, the process has finished, and we have constructed the binary tree with $ 4+1 = 5$ leaves that is associated to the based rectangle partition $P$.

Suppose now that we are given a binary tree $F_{1}$ with $ n+1$ leaves. We now wish to construct a based rectangle partition of the root poset of $ A_{n}$ that corresponds to $F_{1}$. Let $ A$ denote the root node of $F_{1}$. Let $ B$ and $C$ denote the left child of $A$ and the right child of $A$ respectively. Let $ m$ denote the number of leaves on the subtree defined at the vertex $B $. Note that since $ F_{1}$ has $n+1$ leaves where $ n \geq 1$, it follows that $ 1 \leq m \leq n$. Since there are $n+1$ leaves in total on $F_{1}$, it follows that there are $ (n+1) - m$ number of leaves on the subtree defined at the vertex $C$. Construct a highest based rectangle $R_{h}$ of the root poset of $A_{n}$ such that $R_{h}$ contains the simple root $ \alpha_{m}$. Note that $ \Phi^{+} \sm R_{h}$ will be the root poset of $ (W_{J} , J) \times (W_{K} , K)$ where $ J = \{  s_{1} ,s_{2} , \dots, s_{m-1} \}$ and $K = \{ s_{m+1}, s_{m+2}, \dots ,  s_{n} \} $. Note furthermore that $(W_{J} , J) \cong A_{m-1}$ and $ (W_{K} , K) \cong A_{n-m}$. Hence, we can now apply recursion to the subtree defined at $B$ and the root poset of $ (W_{J} ,J)$. Similarly, we can also apply recursion to the subtree defined at $C$ and the root poset of $ (W_{K} ,K)$. This recursive process will continue until you reach a subtree with only one leaf. Once this recursive process ends, we will have constructed a based rectangle partition of the root poset of $ A_{n}$ associated to the binary tree $F_{1}$. We now demonstrate this process with an explicit example. Let $F_{1}$ denote the following binary tree:
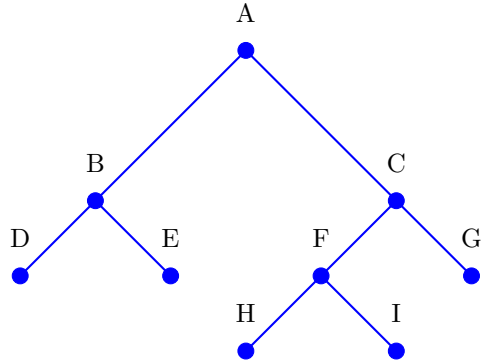
\begin{center}
    \begin{tikzpicture}
    \coordinate (A) at (0,0);
    \coordinate (B) at (-2, -2); 
    \coordinate (C) at (2,-2);
    \coordinate (D) at (-3,-3);
    \coordinate (E) at (-1 ,-3 );
    \coordinate (F) at ( 1,-3 );
    \coordinate (G) at (3 ,-3 );
    \coordinate (H) at (0 ,-4 );
    \coordinate (I) at (2 , -4);

    \node[] at (0,.5) {A};
    \node[] at (-2,-1.5) {B};
    \node[] at (2,-1.5) {C};
    \node[] at (-3,-2.5) {D};
    \node[] at (-1, -2.5) {E};
    \node[] at (1,-2.5) {F};
    \node[] at (3, -2.5) {G};
    \node[] at (0,-3.5) {H};
    \node[] at (2, -3.5) {I};

    \draw[fill=blue , blue] (A) circle (3pt);
    \draw[fill=blue , blue] (C) circle (3pt);
    \draw[fill=blue , blue] (B) circle (3pt);
    \draw[fill=blue , blue] (D) circle (3pt);
    \draw[fill=blue , blue] (E) circle (3pt);
    \draw[fill=blue , blue] (F) circle (3pt);
    \draw[fill=blue , blue] (G) circle (3pt);
    \draw[fill=blue , blue] (H) circle (3pt);
    \draw[fill=blue , blue] (I) circle (3pt);

    \draw[blue , thick] (A)--(C);
    \draw[blue , thick] (A)--(B);
    \draw[blue , thick] (B)--(D);
    \draw[blue , thick] (B)--(E);
    \draw[blue , thick] (C)--(F);
    \draw[blue , thick] (C)--(G);
    \draw[blue , thick] (F)--(H);
    \draw[blue , thick] (F)--(I);
    
\end{tikzpicture}
\end{center}
Since there are 5 = 4+1 leaves, we will be partitioning the root poset of $A_{4} $. Since $ B$ is the left child of $A$, and since the subtree defined at $ B$ has 2 leaves, it follows that the highest based rectangle $R_{h}$ will contain the simple root $ \alpha_{2}$:
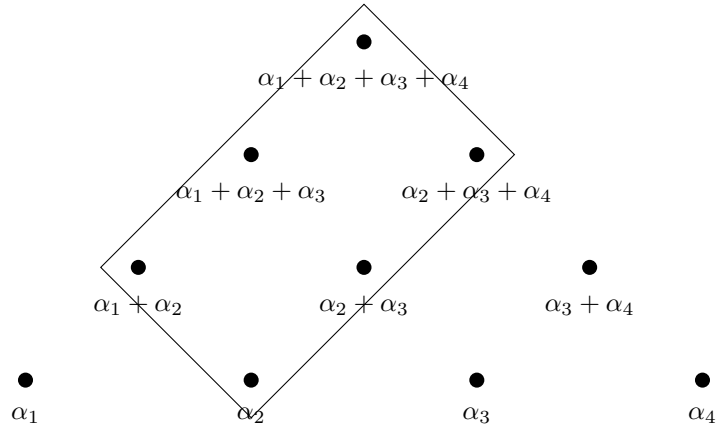
\begin{center}
\begin{tikzpicture}
\foreach  \y in {0,...,3}
    \foreach \x in {0,...,\y} 
        \fill [black] ( 3*\x -1.5*\y , -1.5*\y) circle [radius=0.1];

\node[below = 7pt] at (-4.5,-4.5) {$\alpha_{1}$};
\node[below = 7pt] at (-1.5,-4.5) {$\alpha_{2}$};
\node[below = 7pt] at (1.5,-4.5) {$\alpha_{3}$};
\node[below = 7pt] at (4.5,-4.5) {$\alpha_{4}$};

\node[below = 7pt] at (-3,-3) {$\alpha_{1} + \alpha_{2}$};
\node[below = 7pt] at (0,-3) {$\alpha_{2} + \alpha_{3}$};
\node[below = 7pt] at (3,-3) {$\alpha_{3} + \alpha_{4}$};

\node[below = 7pt] at (-1.5,-1.5) {$\alpha_{1} + \alpha_{2}+ \alpha_{3}$};
\node[below = 7pt] at (1.5,-1.5) {$\alpha_{2} + \alpha_{3}+ \alpha_{4}$};
\node[below = 7pt] at (0,0) {$\alpha_{1} + \alpha_{2}+ \alpha_{3}+ \alpha_{4}$};

\draw[black] (0,.5)--(-3.5,-3)--(-1.5,-5)--(2,-1.5)--cycle;
\end{tikzpicture}
\end{center}
Note that $ \{ \alpha_{1} \}$ is the root poset of $ (W_{ \{ s_{1} \} }, \{ s_{1} \} )$ and that $ \{ \alpha_{3} , \alpha_{3} + \alpha_{4} , \alpha_{4}  \}$ is the root poset of $ (W_{ \{ s_{3}, s_{4} \} }, \{ s_{3}, s_{4} \} )$. Thus, we now apply the recursive process to $ \{ \alpha_{1} \}$ and the subtree defined at $B$:

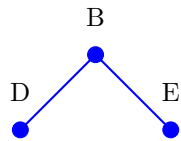
\begin{center}
    \begin{tikzpicture}
   
    \coordinate (B) at (-2, -2); 
    
    \coordinate (D) at (-3,-3);
    \coordinate (E) at (-1 ,-3 );

    \node[] at (-2,-1.5) {B};
    
    \node[] at (-3,-2.5) {D};
    \node[] at (-1, -2.5) {E};

    \draw[fill=blue , blue] (B) circle (3pt);
    \draw[fill=blue , blue] (D) circle (3pt);
    \draw[fill=blue , blue] (E) circle (3pt);

    \draw[blue , thick] (B)--(D);
    \draw[blue , thick] (B)--(E);

\end{tikzpicture}
\end{center}
Similarly, we will also apply the recursive process to $ \{ \alpha_{3} , \alpha_{3} + \alpha_{4} , \alpha_{4}  \}$ and the subtree defined at $C$:
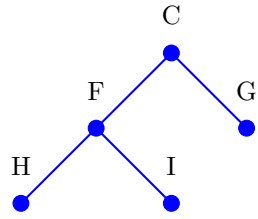
\begin{center}
    \begin{tikzpicture}
    
    \coordinate (C) at (2,-2);
    
    \coordinate (F) at ( 1,-3 );
    \coordinate (G) at (3 ,-3 );
    \coordinate (H) at (0 ,-4 );
    \coordinate (I) at (2 , -4);

    \node[] at (2,-1.5) {C};
    \node[] at (1,-2.5) {F};
    \node[] at (3, -2.5) {G};
    \node[] at (0,-3.5) {H};
    \node[] at (2, -3.5) {I};

    \draw[fill=blue , blue] (C) circle (3pt);
    \draw[fill=blue , blue] (F) circle (3pt);
    \draw[fill=blue , blue] (G) circle (3pt);
    \draw[fill=blue , blue] (H) circle (3pt);
    \draw[fill=blue , blue] (I) circle (3pt);

    \draw[blue , thick] (C)--(F);
    \draw[blue , thick] (C)--(G);
    \draw[blue , thick] (F)--(H);
    \draw[blue , thick] (F)--(I);
    
\end{tikzpicture}
\end{center}
Once these two recursive processes are complete, we will end up with the following based rectangle partition of $A_{4}$:

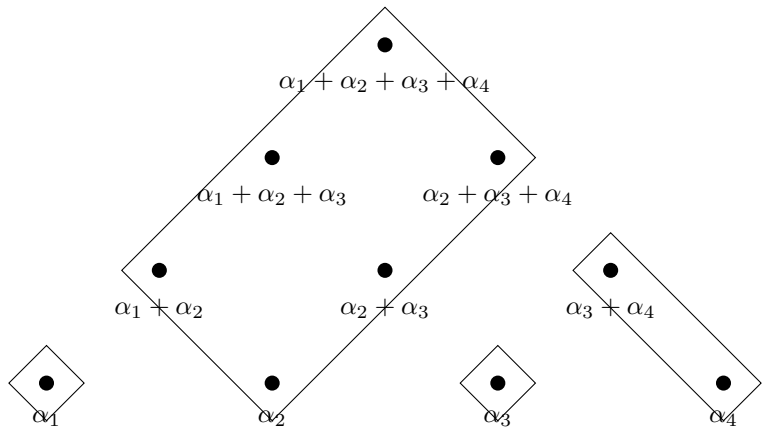
\begin{center}
\begin{tikzpicture}
\foreach  \y in {0,...,3}
    \foreach \x in {0,...,\y} 
        \fill [black] ( 3*\x -1.5*\y , -1.5*\y) circle [radius=0.1];

\node[below = 7pt] at (-4.5,-4.5) {$\alpha_{1}$};
\node[below = 7pt] at (-1.5,-4.5) {$\alpha_{2}$};
\node[below = 7pt] at (1.5,-4.5) {$\alpha_{3}$};
\node[below = 7pt] at (4.5,-4.5) {$\alpha_{4}$};

\node[below = 7pt] at (-3,-3) {$\alpha_{1} + \alpha_{2}$};
\node[below = 7pt] at (0,-3) {$\alpha_{2} + \alpha_{3}$};
\node[below = 7pt] at (3,-3) {$\alpha_{3} + \alpha_{4}$};

\node[below = 7pt] at (-1.5,-1.5) {$\alpha_{1} + \alpha_{2}+ \alpha_{3}$};
\node[below = 7pt] at (1.5,-1.5) {$\alpha_{2} + \alpha_{3}+ \alpha_{4}$};
\node[below = 7pt] at (0,0) {$\alpha_{1} + \alpha_{2}+ \alpha_{3}+ \alpha_{4}$};

\draw[black] (0,.5)--(-3.5,-3)--(-1.5,-5)--(2,-1.5)--cycle;
\draw[black] (-4.5,-4)--(-5,-4.5)--(-4.5,-5)--(-4,-4.5)--cycle;
\draw[black] (3,-2.5)--(2.5,-3)--(4.5, -5)--(5,-4.5)--cycle;
\draw[black] (1.5, -4)--(1, -4.5)--(1.5, -5)--(2, -4.5)--cycle;
\end{tikzpicture}
\end{center}

One can easily verify that the process of associating a binary tree with $n+1$ leaves to a based rectangle partition of $A_{n}$ and the process of associating a based rectangle partition of $A_{n}$ to some binary tree with $n+1$ leaves are inverse processes of one another. Furthermore, if we look closely at the following binary tree lying on top of its associated based rectangle partition, we see that compatible subtriangle flips of the based rectangle partition correspond to subtree flips of the the binary tree:
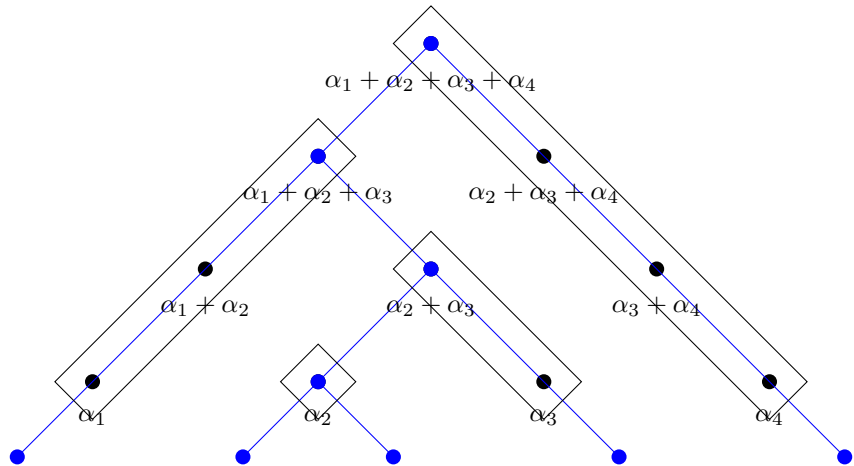
\begin{center}
\begin{tikzpicture}
\foreach  \y in {0,...,3}
    \foreach \x in {0,...,\y} 
        \fill [black] ( 3*\x -1.5*\y , -1.5*\y) circle [radius=0.1];

\fill [blue] (0,0) circle [radius=0.1];

\draw[blue] (0,0)-- (-5.5,-5.5);
\draw[blue] (0,0)--(5.5,-5.5);

\fill [blue] (-5.5,-5.5) circle [radius=0.1];
\fill [blue] (5.5,-5.5) circle [radius=0.1];

\fill [blue] (-1.5,-1.5) circle [radius = 0.1];

\draw[blue] (-1.5,-1.5)--(2.5 ,-5.5);
\fill[blue] (2.5, -5.5) circle [radius = 0.1];

\fill[blue] (0,-3) circle [radius = 0.1];

\draw[blue] (0,-3)--(-2.5 , -5.5);
\fill[blue] (-2.5,-5.5) circle [radius = 0.1];

\fill[blue] (-1.5, -4.5) circle [radius = 0.1];
\draw[blue] (-1.5, -4.5)--(-.5 , -5.5);
\fill[blue] (-.5,-5.5) circle [radius = 0.1];

\node[below = 7pt] at (-4.5,-4.5) {$\alpha_{1}$};
\node[below = 7pt] at (-1.5,-4.5) {$\alpha_{2}$};
\node[below = 7pt] at (1.5,-4.5) {$\alpha_{3}$};
\node[below = 7pt] at (4.5,-4.5) {$\alpha_{4}$};

\node[below = 7pt] at (-3,-3) {$\alpha_{1} + \alpha_{2}$};
\node[below = 7pt] at (0,-3) {$\alpha_{2} + \alpha_{3}$};
\node[below = 7pt] at (3,-3) {$\alpha_{3} + \alpha_{4}$};

\node[below = 7pt] at (-1.5,-1.5) {$\alpha_{1} + \alpha_{2}+ \alpha_{3}$};
\node[below = 7pt] at (1.5,-1.5) {$\alpha_{2} + \alpha_{3}+ \alpha_{4}$};
\node[below = 7pt] at (0,0) {$\alpha_{1} + \alpha_{2}+ \alpha_{3}+ \alpha_{4}$};

\draw (0,.5)--(-.5,0)--(4.5,-5)--(5,-4.5)--cycle;

\draw (-1.5,-1)--(-1, -1.5)--(-4.5,-5)--(-5,-4.5)--cycle;

\draw (0 , -2.5)--(-.5,-3)--(1.5,-5)--(2,-4.5)--cycle;

\draw (-1.5 , -4)--(-2, -4.5)--(-1.5,-5)--(-1, -4.5)--cycle;
\end{tikzpicture}
\end{center}
Thus, by Proposition \ref{binarytreeisom}, we have established a bijection between based rectangle partitions of $A_{n}$ modulo compatible subtriangle flips and binary trees with $n+1$ leaves modulo isomorphism. This completes the proof.
\end{proof}

We would like to classify the elements of the $A_{n}$ Coxeter system that appear as edges of some Coxeter $n$-cube in $A_{n}$. We define $\Edge{n}$ to denote the set of all non-identity elements $x$ that appear as an edge of some orientation of a Coxeter $n$-cube from $A_{n}$. 

\begin{ex}
    Consider the unique Coxeter square modulo orientation in $ A_{2}$:

    
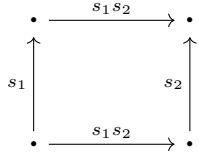
\[ \begin{tikzcd}
\sbullet \arrow[rr, "s_{1}s_{2}"] & & \sbullet \\
\\
\sbullet \arrow[uu, "s_{1}"] \arrow[rr, "s_{1}s_{2}"] & & \sbullet \arrow[uu, "s_{2}"]
\end{tikzcd}
\]
Hence, we see that $ s_{1} , s_{2} , s_{1}s_{2} \in \Edge{2} $. However, the above square can be reoriented as such:

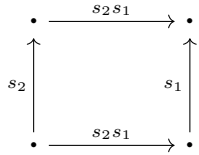
\[ \begin{tikzcd}
\sbullet \arrow[rr, "s_{2}s_{1}"] & & \sbullet \\
\\
\sbullet \arrow[uu, "s_{2}"] \arrow[rr, "s_{2}s_{1}"] & & \sbullet \arrow[uu, "s_{1}"]
\end{tikzcd}
\]
Hence, we also have that $ s_{2}s_{1} \in \Edge{2}$. There are no other elements that appear as an edge of some reorientation of the above Coxeter square. Hence,

$$ \Edge{2} = \{ s_{1} , s_{2} , s_{1}s_{2} , s_{2}s_{1}  \}$$

\end{ex}

\begin{thm}
    We have the following recursive relation:

    $$ |\Edge{n+1}| = |\Edge{n}| + \frac{(n+1)(n+2)}{2}$$
    where $ |\Edge{1}| =1$. Hence:

    $$ |\Edge{n}| = \binom{n+2}{3} = \frac{n(n+1)(n+2)}{6}$$
\end{thm}

\begin{proof}
    If $ s$ denotes the unique non-identity element of $ A_{1}$, then the unique Coxeter $1$-cube of $ A_{1}$ modulo orientation is the following:

     \[
\begin{tikzcd}[row sep=1.5em, column sep = 1.5em]
\sbullet \arrow[rr, "s"] & & \sbullet
\end{tikzcd}
\]
Since $ s = s^{-1}$, we conclude that $ \Edge{1} = \{ s \}$. Hence, $ |\Edge{1}| =1$. Clearly, $ \binom{1+2}{3} = \binom{3}{3} = 1$, so the explicit formula for $ |\Edge{n}|$ in the latter half of the proposition holds true in the base case.

Let us now consider how to compute $ |\Edge{n+1}|$. By Theorem \ref{partitiontoncube}, we know that the elements of $ \Edge{n+1}$ correspond to based rectangles in the root poset $ \Phi^{+}$ of $ A_{n+1}$. Hence, we just need to enumerate the number of based rectangles in the root poset $\Phi^{+}$ of $ A_{n+1}$. Let $S = \{ s_{1} , s_{2} , \dots , s_{n} , s_{n+1}  \}$ denote the Coxeter generators for the $A_{n}$ Coxeter system. Note that for $ J = S \sm \{ 
s_{n+1} \}$, the root subposet $ \Phi^{+}_{J}$ associated to $ (W_{J} , J)$ is isomorphic to the root poset of $ A_{n}$. Hence, the number of based rectnagles within the root subposet $ \Phi^{+}_{J}$ is equal to $ |\Edge{n}|$. We just need to calculate the based rectangles in $ \Phi^{+}$ that are not fully contained in $ \Phi^{+}_{J}$. But note that if a based rectangle is not contained in $ \Phi^{+}_{J}$, then the root of highest depth  from the based rectangle must be an element of $ \Phi^{+} \sm \Phi^{+}_{J}$. Note that the elements of $ \Phi^{+} \sm \Phi^{+}_{J}$ are of the form $ \sum_{i=0}^{k} \alpha_{(n+1) - i}$ where $ k = 0,1,2, \dots, n$. But note that a based rectangle with highest root $\sum_{i=0}^{k} \alpha_{(n+1) - i} $ will have exactly one simple root from the following set: $\{ \alpha_{(n+1)} , \alpha_{n} , \alpha_{n-1}, \dots , \alpha_{n+1-k} \} $. Hence, we will have $ k+1$ options for each fixed $k$. Adding up all possible options as $k$ ranges from $0$ to $n$, we obtain:

$$ 1 +2 +3+ \dots + n+(n+1) = \frac{(n+1)(n+2)}{2}$$
Thus, the total number of based rectangles in the root poset $ \Phi^{+}$ of $A_{n+1}$ is equal to $ |\Edge{n}| + \frac{(n+1)(n+2)}{2}$. Therefore,

$$ |\Edge{n+1}| = |\Edge{n}| + \frac{(n+1)(n+2)}{2}$$
One can now verify that the formula $ |\Edge{n}| = \binom{n+2}{3}$ satisfies the above recurrence relation with $ |\Edge{1}| = 1$. Hence, by induction, $ |\Edge{n}| = \binom{n+2}{3}$ for all $n$.
    
\end{proof}

Recall that the Coxeter system $ A_{n}$ is isomorphic as a group to $ S_{n+1}$, where $ S_{n+1}$ denotes the symmetric group on the set $ [n+1]:= \{ 1,2, 3, \dots, n,n+1\}$. Thus, an element $ x$ of $ A_{n}$ can be viewed as a permutation on the set $ [n+1] $. If $ x$ is an element of $ A_{n}$, we say that $ x$ is a \emph{bigrassmannian permutation} if there exists a unique ordered pair $ (i,j) \in [n]^{2}$ such that $ x(i) > x(i+1)$ and $ x^{-1}(j) > x^{-1}(j+1)$. Equivalently, one can show that $ x$ is a bigrassmannian permutation of $ A_{n}$ if and only if $ |\Pi \cap \Phi_{x}| =1$ and $ |\Pi \cap \Phi_{x^{-1}}| = 1$.

\begin{thm}
    Let $ x$ be an element of $ A_{n}$. Then $ x \in \Edge{n}$ if and only if $ x$ is a bigrassmannian permutation. 
\end{thm}

\begin{proof}
    Let $ \mathcal{B}_{n}$ denote the set of bigrassmanian permutations in the $ A_{n}$ Coxeter system. We prove that $ \Edge{n} = \mathcal{B}_{n}$. Note that if $ x \in \Edge{n}$, then $ \Phi_{x}$ is a based rectangle in the root poset of $ A_{n}$, meaning that $ |\Phi_{x} \cap \Pi| = 1$. Note also that if $ x \in \Edge{n}$, then we can reorient the Coxeter $n$-cube in which $ x$ appears so that $ x^{-1}$ appears as an edge of the Coxeter $n$-cube. Hence, $ x^{-1} \in \Edge{n}$, and thus we can conclude that $ |\Pi \cap \Phi_{x^{-1}}| = 1$. We have just shown that if $ x \in \Edge{n}$, then $ x \in \mathcal{B}_{n}$. Thus, we conclude that 

    $$ \Edge{n} \seq \mathcal{B}_{n}$$
    But note that we also proved that $ |\Edge{n}| = \binom{n+2}{3}$. According to \cite{engbers2015comparability}, the number of bigrassmanian permutations in the $ A_{n}$ Coxeter system is $ \binom{n+2}{3}$. Since $\Edge{n}$ and $ \mathcal{B}_{n} $ are finite sets of the same cardinality with $ \Edge{n} \seq \mathcal{B}_{n}$, it follows that

    $$ \Edge{n} = \mathcal{B}_{n}$$

\end{proof}

\bibliographystyle{abbrv}

\end{document}